\documentclass[11pt,a4paper,reqno]{amsart}
\usepackage[english]{babel}
\usepackage[applemac]{inputenc}
\usepackage[T1]{fontenc}
\usepackage{palatino}
\usepackage{amsmath}
\usepackage{amssymb}
\usepackage{amsthm}
\usepackage{amsfonts}
\usepackage{graphicx}
\usepackage{color}
\usepackage{vmargin}
\usepackage{mathtools}
\usepackage{esint}

\setmarginsrb{20mm}{20mm}{20mm}{20mm}{10mm}{10mm}{10mm}{10mm}

\usepackage[colorlinks = true, citecolor = black]{hyperref}
\newtheorem{theorem}{Theorem}[section]
\newtheorem{obs}{Observation}[section]

\newtheorem{cor}[theorem]{Corollary}
\newtheorem{prop}[theorem]{Proposition}

\newtheorem{lem}[theorem]{Lemma}

\theoremstyle{definition}
\newtheorem{defn}{Definition}[section]

\newtheorem{remark}[theorem]{Remark}

\newtheorem*{theorem*}{Theorem}
\newtheorem{ex}{Example}[section]
\newtheorem{rem}[theorem]{Remark}

\def\vint{\mathop{\mathchoice%
          {\setbox0\hbox{$\displaystyle\intop$}\kern 0.22\wd0%
           \vcenter{\hrule width 0.6\wd0}\kern -0.82\wd0}%
          {\setbox0\hbox{$\textstyle\intop$}\kern 0.2\wd0%
           \vcenter{\hrule width 0.6\wd0}\kern -0.8\wd0}%
          {\setbox0\hbox{$\scriptstyle\intop$}\kern 0.2\wd0%
           \vcenter{\hrule width 0.6\wd0}\kern -0.8\wd0}%
          {\setbox0\hbox{$\scriptscriptstyle\intop$}\kern 0.2\wd0%
           \vcenter{\hrule width 0.6\wd0}\kern -0.8\wd0}}%
          \mathopen{}\int}

\newcommand{\R}{\mathbb{R}}
\newcommand{\Rn}{{\mathbb R}^n}
\newcommand{\N}{\mathbb{N}}

\newcommand{\C}{\mathbb{C}}

\newcommand{\ep}{\epsilon}
\newcommand{\Om}{\Omega}

\newcommand{\ud}{\mathrm {d}}

\newcommand{\dist}{\operatorname{dist}}

\newcommand{\dcc}{{d_{CC}}}

\newcommand{\divGt}{{\rm div}_{G_2^\alpha}}

\definecolor{blau}{rgb}{0.1,0.0,0.9}
\definecolor{violet}{rgb}{0.54, 0.17, 0.89}
\newcommand{\blue}{\color{blau}}

\newcommand{\kom}[1]{}
\renewcommand{\kom}[1]{{\bf \blue /#1/}}

\newcounter{komcounter}
\numberwithin{komcounter}{section}

\def\XXint#1#2#3{{\setbox0=\hbox{$#1{#2#3}{\int}$}
		\vcenter{\hbox{$#2#3$}}\kern-.5\wd0}}

\makeatletter
\newcommand{\leqnomode}{\tagsleft@true\let\veqno\@@leqno}
\makeatother

\usepackage{hyperref}
\hypersetup{%
  colorlinks = true,
  linkcolor  = black
}

\DeclareMathOperator{\sgn}{sgn}
\begin{document}

\title[Harmonic mappings on Grushin planes, part II]{Harmonic mappings on Grushin planes, part II{\small$^*$}}
\author[T.\ Adamowicz]{Tomasz Adamowicz}
\address{T.A.: The Institute of Mathematics, Polish Academy of Sciences \\ ul. \'Sniadeckich 8, 00-656 Warsaw, Poland}
\email{tadamowi@impan.pl}
\author[K. Che\l mi\'nski ]{Krzysztof Che\l mi\'nski}
\author[M. Walicki ]{Marcin Walicki}
\address{K.Ch. and M.W.: Warsaw University of Technology, Faculty of Mathematics and Information Science \\ ul. Koszykowa 75, 00-662, Warsaw, Poland}
\email{krzysztof.chelminski@pw.edu.pl}
\email{marcin.walicki.dokt@pw.edu.pl}

\thanks{{\small$^*$}
The Authors declare that no AI system was employed in any part of the reasoning in the manuscript. }


\begin{abstract}
 We study harmonic mappings between the source domain in the $\alpha$-Grushin plane and the target domain in the $\beta$-Grushin plane for $\alpha$ not necessarily equal to $\beta$ and $\alpha\in [0,1)$. Our harmonic mappings arise as a Euler--Lagrange system of equations related to the Dirichlet energy and are studied in weak and strong forms. The key results enclose the second order regularity of harmonic mappings, the Bochner identity, the weak Harnack estimates for coordinate functions of a harmonic mapping, as well as the Caccioppoli estimates and the Liouville type theorem. Furthermore, we discuss several examples of (Grushin) harmonic mappings and, in particular, relate them to holomorphic mappings in the plane.
 
\end{abstract}

\maketitle

\tableofcontents

\section{Introduction}

The theory of harmonic mappings between Riemannian manifolds is nowadays a well-esta\-blished topic of study with several profound consequences in analysis and geometry, involving techniques of differential geometry, PDEs and the theory of function spaces. One of its fruitful generalizations is the investigation of harmonic mappings in the subriemannian setting (also in the almost-riemannian setting, e.g.~\cite{abs, bl}) where one of the key motivations comes from applications in control theory and neurobiology. There, the constraints of the movement in a model are reflected in the fact that a subbundle of the tangent bundle is not necessarily of constant rank. This may happen, for instance, when the vector fields that span the subbundle vanish along some singular set.  Furthermore, from a PDEs and physics perspective the subriemannian geometries arise in connection with the H\"ormander condition and the Chow theorem, see Montgomery~\cite{mon}. 

One of the particularly interesting and profound toy models of subriemannian geometry is the Grushin spaces. In this work we continue the studies of harmonic mappings from such spaces, initiated in~\cite{aww} for Euclidean target spaces, however in the more challenging and difficult setting of the subriemannian targets of Grushin spaces. The related system of Euler--Lagrange equations for harmonic mappings from an $n$-dimensional Grushin space to an $m$-dimensional one is quite technically involved to study due to the coupling of equations and the highly nontrivial structure of singular sets in the source and target domains, see~\eqref{system}. Therefore, we focus our attention on a simpler, yet still demanding setting of Grushin planes.  Already for harmonic mappings between Grushin planes the corresponding systems of PDEs, see~\eqref{system-EL} or~\eqref{str-system-EL}, lead to a challenging analysis of the regularity and geometry of solutions. Moreover, the analysis involving singular lines present in both the source and the target domains of a harmonic map require tedious work and a substantial adaptation of existing techniques.

Before presenting the main results and the structure of the paper, we would like to point at main features that distinguish our results from the ones in the literature, namely:
\smallskip
\\
\indent (1) the H\"older regularity of the vector fields defining the Grushin space rather than the Lipschitz one;

 (2) the weighted measure with the Muckenhoupt weight instead of the Lebesgue measure;
 
 (3) harmonic mappings in this work are given by the coupled system of equations with singular sets in both the source and the target domains of the map;
 
 (4) our results allow for studying mappings on the domains intersecting the singular line in the Grushin plane, but are new already for domains omitting the singular line.
\smallskip 
\\
\noindent We now briefly address those features and discuss in detail in the rest of the manuscript. Several results in the literature consider the Grushin structure given by the vector fields with the Lipschitz coefficients, see e.g. Ferrari--Valdinoci~\cite{fv, fv2} and Domokos--Manfredi~\cite{dm}.  However, as in our previous work~\cite{aww}, we focus our attention on the H\"older regularity case, with our primary example of the $\alpha$-Grushin plane $G_2^\alpha$ given by the vector fields
\begin{equation}\label{intro-vect} 
 X=\frac{\partial}{\partial x}\quad  Y=|x|^{\alpha}\frac{\partial}{\partial y},\quad 0\leq \alpha<1,
\end{equation}
see Example~\ref{ex12} below. Note that in such a case the H\"ormander condition fails. Furthermore, instead of studying PDEs in the Grushin planes assuming the standard Lebesgue measure, we consider Grushin planes equipped with the following natural weighted measure $\mu$:
\begin{equation*}
 \mu:= \frac{\ud x \ud y}{ |x|^{\alpha}}.
\end{equation*}
As shown in~\cite{aww} such a measure is an Ahlfors regular measure and, moreover, satisfies the $p$-Muckenhoupt condition for any $p>1$. The Grushin space equipped with the weighted measure and the H\"older regularity of coefficients of the vector fields turn out to be sufficient assumptions for a viable theory of the Sobolev spaces of functions and mappings in the spirit of works by Franchi-Serrapioni~\cite{fs87}, Franchi--Serrapioni--Serra-Cassano~\cite{fss}, see also Franchi--Haj\l asz--Koskela~\cite{fhk}.

 Since we view the Grushin planes as the metric measure spaces with the weighted measure, also the associated $2$-Dirichlet energy $E$ depends on the weight and for mappings $u=(u^1, u^2)$ from a domain in the $\alpha$-Grushin plane into the $\beta$-Grushin plane reads:
\begin{align*}
   E(u):=\frac{1}{2}\int_{\Omega} \|D_H u\|_{G_2^\beta}^2 \frac{\ud x\ud y}{|x|^\alpha}, \quad
 \|D_H u\|_{G_2^\beta}^2:=\left(\frac{\partial u^1}{\partial x}\right)^2+\frac{1}{|u^1|^{2\beta}}\left(\frac{\partial u^2}{\partial x}\right)^2+ |x|^{2\alpha}\left(\frac{\partial u^1}{\partial y}\right)^2+\frac{|x|^{2\alpha}}{|u^1|^{2\beta}}\left(\frac{\partial u^2}{\partial y}\right)^2
\end{align*}
Upon denoting by $\nabla_H v:=(\frac{\partial v}{\partial x}, |x|^{\alpha}\frac{\partial v}{\partial y})$ the so-called horizontal gradient of function $v$, the Euler--Lagrange system of equations corresponding to the energy $E(u)$ reads, cf.~\eqref{system-EL} and~\eqref{w-system-EL}:
\begin{equation}\label{intro-system-EL}
\begin{cases}
{\rm div}_{G_2^{\alpha}}\left(\frac{\nabla_H u^1}{|x|^\alpha}\right)+\beta \frac{|\nabla_H u^2|^2 u^1}{|x|^\alpha |u^1|^{2\beta+2}}=0 \\
{\rm div}_{G_2^{\alpha}}\left(\frac{\nabla_H u^2}{|x|^\alpha|u^1|^{2\beta}}\right)=0.
\end{cases}
\end{equation}
 The most challenging and technically involved part of analysis of harmonic mappings in~\eqref{intro-system-EL} comes from the weight $|x|^\alpha$ for the source domain of $u$ and $|x|^\beta$ for its target domain, respectively. In a consequence, the study of the behavior of $u$ becomes difficult close to the singular set in the source $\alpha$-Grushin plane (i.e., close to the $y$-axis). Similarly, an extra effort has to be taken to analyze $u$ when its target domain intersects the $y$-axis in the $\beta$-Grushin plane.

Solutions to the system~\eqref{intro-system-EL} will be called \emph{weak harmonic mappings}, see Definition~\ref{weak-solution}. A complementary approach to harmonic mappings between Grushin planes can be formulated if we consider harmonic mappings given by the system of equations in the non-divergence form. Such mappings will be called \emph{strong harmonic mappings}, see Definition~\ref{str-solution} and are particularly handy in the studies of examples of harmonic mappings, see Section~\ref{sect2}.

\subsection*{The outline of the paper. Main results} In the {\bf Preliminaries} we recall some of the basic definitions and results about Grushin spaces and, in particular, Grushin planes with Example~\ref{ex12} and Lemma~\ref{lem-Ahl-Muck} being the key ones for the following discussion in our manuscript.  {\bf Section~\ref{sect12}} is devoted to describing the setting of Sobolev spaces used throughout the work, in particular, we discuss the density result in Proposition~\ref{prop-dens}. The Dirichlet energy corresponding to mappings between Grushin planes and the derivation of the related Euler--Lagrange system of equations together with the existence of the minimizers are presented in {\bf Sections~\ref{sect13}-\ref{sect14}}. The key definitions for Grushin harmonic mappings are introduced in {\bf Sections~\ref{sect15} and~\ref{sect16}} and followed by the Caccioppoli estimates and the Liouville-type result. 

We decided to devote a separate section to discussing a variety of examples of Grushin-harmonic mappings, see {\bf Section 3}. According to the best of our knowledge such examples appear for the first time in the literature and are part of a strong motivation for further studies of harmonic mappings in the Grushin setting. In particular, we present three categories of examples and relate our mappings to Grushin-conformal mappings and holomorphic mappings.

Our core contributions are presented in {\bf Sections 4-6}. In Section 4 we introduce the second order Sobolev spaces $H^{2,2}$ and discuss the difference quotients method in the weighted setting. Then, in {\bf Lemma~\ref{lem-Sob-char}}, we provide a Grushin counterpart of the well-known Euclidean description of the Sobolev functions in $H^{2,2}$. However, note that the dependence of the weights on the point in the domain leads to significant technical challenges for the difference quotients technique in the Grushin setting. Upon completing the necessary preparations, we show the following result for the Grushin-harmonic mappings from an open set in $\alpha$-Grushin plane $G_2^\alpha$ into $\beta$-Grushin plane $G_2^\beta$. Moreover, we discuss some examples of harmonic mappings satisfying the assumptions of the theorem, see Example~\ref{ex-thm-H22} and relate our studies to Euclidean ones and our previous results, see Remark~\ref{rem44}. We refer to the introduction of Section 4 for a discussion on how our $H^{2,2}$-result differs from the corresponding result in~\cite{aww}, as well as for a description of the challenges arising when the target space of a map is a Grushin-plane instead of the Euclidean plane.

\begin{theorem}\label{thm-H22}
Let $\Om\subset G^{\alpha}_2$ be an open bounded set and $u=(u^1, u^2):\Om \to G^2_{\beta}$ be a weakly Grushin-harmonic map in $H^{1,2}_{loc}(\Om,G^{\beta}_2,\mu)$, i.e., $u$ satisfies the system of equations~\eqref{w-system-EL}, and  such that $\overline{u^1(\Om)}\cap S=\emptyset$.  Moreover, assume that $\alpha, \beta \in [0,1)$ and there exist 
 \begin{equation}
 p> \frac{2}{1-\alpha}\,\,\hbox{ and }\,\,\ \gamma \geq \max\left\{ \frac{2+p+p\alpha}{2\alpha},\, p(2-\alpha)\right\}\,(>2)\qquad \tag{H2-A} \label{ass0-thm-H22}
\end{equation}
 such that for any ball $B(R)\cap S\not=\emptyset$ it holds that:
 \leqnomode
\begin{align}
&\int_{B(R)} \frac{|\nabla_H u^2|^p}{|u^1|^{(\beta+1)p}} \ud \mu<\infty, \qquad
\int_{B(R)} \frac{|u^1_x|^{\frac{2p}{p-2}}}{|x|^{\frac{2\alpha}{p-2}}} \ud \mu<\infty \tag{H2-B} \label{ass1-thm-H22} \\
&\int_{B(R)} \frac{|D_H u|^p}{|x|^{\alpha(\gamma-1)}} \ud \mu<\infty \tag{H2-C}\label{ass2-thm-H22} \\
& \int_{B(R)} \frac{|u^1_x|^{\frac{p}{p-2}}}{|x|^{\frac{p}{p-2}}} \ud \mu<\infty. \tag{H2-D}\label{ass3-thm-H22}
\end{align}
If assumptions \eqref{ass0-thm-H22}-\eqref{ass3-thm-H22} hold, then the second order derivatives $X^2 u^1, XY u^1$ and $X^2 u^2, XY u^2$ are in $L^{2}_{loc}(\Om, \mu)$. 

\noindent If assumption~\eqref{ass1-thm-H22} holds for some $p>2$, then the second order derivatives $YXu^1, Y^2u^1$ and $YXu^2, Y^2u^2$ are in $L^{2}_{loc}(\Om, \mu)$. 

\noindent
Furthermore, if an open set $\Om'\subset \Om$ satisfies $\overline{\Om'}\cap S=\emptyset$, then it holds that $u^1, u^2 \in H^{2,2}(\Om', \R, \mu)$ only under the assumption~\eqref{ass1-thm-H22} holding for some $p>2$.
\end{theorem}

The results of Section 4 give foundations for the studies in {\bf Section 5}, where we discuss the Bochner identity for harmonic mappings between the Grushin planes. Let
\[
\|D^2_H u\|^2:=\|\nabla_H^2 u^1\|_{H}^2+\frac{\|\nabla_H^2 u^2\|_{H}^2}{|u^1|^{2\beta}},
\]
denote the norm of Hessian matrix of $u$, see~\eqref{def-hessian-fun}. Moreover, see Definition~\ref{defns-Tcom} of the  commutator $T$.

\begin{theorem}[Bochner formula]\label{thm-Bochner}
 Let $\alpha \in [0,1)$ and $\Om\subset G_2^{\alpha}$ be an open set. Suppose that $u:\Om\to G_2^\beta$ such that $u\in H^{2,2}_{loc}(\Om, \mu; G_2^\beta)$ is a weak harmonic map, i.e., $u$ satisfies Definition~\ref{weak-solution} in $\Om$
as well as on any ball $B(R)\Subset \Om\setminus S$ we have that
\begin{equation}
\int_{B(R)} \frac{|\nabla_H u^2|^p}{|u^1|^{(\beta+1)p}} \ud \mu<\infty, \qquad
\int_{B(R)} |\nabla_H u^1|^{\frac{2p}{p-2}} \ud \mu<\infty \quad\hbox{for some }p>2.\tag{H2-B'}
\label{ass11-thm-H22}
\end{equation}
Then the following Bochner identity holds in the weak sense in any ball $B(R)\Subset \Om\setminus S$: 
\begin{align}
 &L\Big(\frac12 \|D_Hu\|^2\Big)+\alpha|x|^{\alpha-2}\|D_Hu\|^2 \nonumber \\
 &\phantom{AAA}= \varrho \|D^2_H u\|^2+2\varrho \Bigg(\big((TXu^1)Yu^1-(TYu^1)Xu^1\big)+\frac{\big((TXu^2)Yu^2-(TYu^2)Xu^2\big)}{|u^1|^{2\beta}} \Bigg) \nonumber \\
&\phantom{AAA}+\frac{\beta \varrho}{|u^1|^{2\beta+2}}\Bigg(2u^1\Big\langle \nabla_H u^2,  \nabla_H\langle \nabla_H u^1,  \nabla_H u^2 \rangle \Big\rangle-3u^1\langle \nabla_H u^1, \nabla_H |\nabla_H u^2|^2 \rangle-2\langle \nabla_H u^1,  \nabla_H u^2 \rangle^2 \nonumber \\
&\phantom{AAAAAAAAA}+|\nabla_H u^2|^2 \Big(\beta\|D_Hu\|^2+(2+3\beta)|\nabla_H u^1|^2 \Big) \Bigg). \label{intro-id-Bochner}
\end{align} 
\end{theorem}

Let us remark, that this Bochner identity generalizes some of the previously studied Bochner identities, in particular:
\begin{itemize}
\item[(1)] If $\alpha=\beta=0$ and so $\varrho=1$ and both Grushin planes $G_2^\alpha$ and $G_2^\beta$ become the Euclidean planes, we retrieve the classical Bochner identity for the planar harmonic functions, i.e. 
$$
\Delta \frac{|\nabla v|^2}{2}=|\nabla^2v|^2.
$$
\item[(2)] If $\beta=0$, then the harmonic map $u:G_2^\alpha\to \R^2$, then we retrieve the Bochner identity studied in~\cite[Theorem 1.3]{aww}.
\end{itemize}

The geometric interpretation of the formula~\eqref{intro-id-Bochner} is as follows: its left-hand side consists of the elliptic-type differential operator acting on the function $\|D_Hu\|^2$, whereas on the right-hand side we have the Hessian term, the commutator terms for component functions and terms with projections of horizontal gradients of component functions. Note that the similar formulas with the analogous commutator and projection terms are known in the literature, e.g.:
\begin{itemize}
\item[(1)] in the setting of Carnot groups of step $2$ and $C^3$ subelliptic-harmonic functions, see~\cite[Corollary 3.6]{ga23} 
\[ 
 \Delta_H(|\nabla_Hf|^2)=2\|D^2_Hf\|^2+\frac12 \sum_{i,j=1}^m \left( (X_iX_j-XjX_i)f \right)^2+4\sum_{i,j=1}^m X_if (X_iX_j-X_jX_i)X_if,
\]
where $X_i$ stand for the horizontal vector fields in the given Carnot group.
\item[(2)] for the three-dimensional Sasakian models, see the Bochner formula on page 163, Section 2.2 in~\cite{bg}; see also the discussion of the generalized Bochner identity in~\cite[Section 2.3.2]{bg}, where the Ricci-term $\mathcal{R}(f)$ in \cite[Definition 2.15]{bg} corresponds to our commutator terms, and the second order differential form $\mathcal{S}(f)$  in~\cite[Formula  (2.14)]{bg} corresponds to our projection terms. In this notation our Bochner formula corresponds to the identity (2.17) in Theorem 2.18 in~\cite{bg}. 
\end{itemize}

The deeper relations between the aforementioned Bochner identities and the one in Theorem~\ref{thm-Bochner} are still to be investigated, as well as the Bochner identities for harmonic mappings in the higher dimensional Grushin spaces.

One of the consequences of the Bochner identity is presented in {\bf Corollary~\ref{cor-Bochner}}, where we infer that $\|D_H u\|$ is a subsolution of the following second order differential operator
\begin{equation*}
L\Big(\frac12 \|D_H u\|^2\Big)+\alpha|x|^{\alpha-2}\|D_Hu\|^2+\beta \frac{|x|^\alpha}{|u^1|^{2}}\|D_H u\|^4\geq 0.
\end{equation*} 
 In order to obtain this observation, one needs to assume that the map $u$ omits the singular set in $G_2^\beta$, as well as that $(TX u^i)Yu^i-(TYu^i)Xu^i \geq 0$ for $i=1,2$. The positivity assumptions on the commutators are counterparts of the corresponding assumption in Corollary 4.2 in~\cite{aww} for harmonic mappings with targets in a Euclidean space and, moreover, turn out to have a geometric interpretation in terms of the level sets of functions $u^i$, see Section 4 in~\cite{aww}.

Finally, in {\bf Section~\ref{sect-w-Harnack}} we observe that the component functions of a strong harmonic mapping $u=(u^1,u^2)$ enjoy the weak-Harnack estimates, see {\bf Theorems~\ref{weak-harnack-u2} and~\ref{weak harnack for u^1 thm}}. In particular, we show that the positive and negative parts of the second component function $u^2$ of $u$ satisfy the following estimate for every $p>1$ 
\begin{equation*}
    \sup_{\lambda B_r}|u^2_{\pm}| \le \frac{C}{1-\lambda}\Big(\vint_{B_r} \| u\|_{G_2^{\beta}}^{(1+\beta)p(2+\alpha)} \ud \mu \Big)^{\frac{1}{p(2+\alpha)}}.
\end{equation*}
Here $\alpha \in (0,1)$, $\beta \ge 0$, $\lambda\in (0,1)$ and we also impose condition~\eqref{orth-cond} for a function $K \in L^2_{\textrm{loc}}(\Om,\mu)$. The same geometric condition appears in the Caccioppoli inequality for strong harmonic mappings, see Lemma~\ref{lem-Cac-str}, and describes the growth of the Euclidean scalar product between $\nabla_H u^1$ and $\nabla_H u^2$. Moreover, similar estimate holds for the first component function $u^1$, see~\eqref{weak harnack for u^1 main inequality}, under similar conditions as above. The differences between the assertions for $u^1$ and $u^2$ come from lack of symmetry between the first and the second equations in the strong harmonic system~\eqref{weak-system-EL} and from the appearance of the weight $|u^1|^{2\beta}$ in the subriemannian norms~\eqref{def-DH} and~\eqref{defn-subr-unorm}.
\section{Preliminaries}\label{sect1}

In this section we first recall some basic definitions and properties of the Grushin spaces and planes. Then, in Section~\ref{sect12} we discuss the Sobolev spaces of functions and mappings used in our work. Section~\ref{sect13} is devoted to formal derivation of the Euler--Lagrange system of equations studied in this paper and the related Dirichlet energy, see formulas~\eqref{def-energy} - \eqref{system} and~\eqref{system-EL} - \eqref{str-system-EL}. However, the existence of the minimizers is discussed in the following Section~\ref{sect14}. Our main classes of Grushin harmonic mappings are brought on stage in Sections~\ref{sect15} - \ref{sect16}, where we define the weak- and the strong harmonic mappings, respectively and discuss their Caccioppoli estimates, see Lemmas~\ref{lem-Cac} and~\ref{lem-Cac-str}. As an application of such energy estimates we prove a variant of the Liouville theorem, see Proposition~\ref{thm-Liouv}.

\subsection{Grushin spaces. The measure and geometric setting}\label{sect11}

 The following section is based on the corresponding introduction to the Grushin spaces presented in the beginning of Section 2 in~\cite{aww}. Let us consider the Euclidean space $\Rn$ equipped with the coordinate system $(x_1,x_2,\ldots, x_n)$ and the following vector fields
 \begin{equation}\label{def-vectors}
  X_1=\frac{\partial}{\partial x_1}\quad   X_2=\varrho_2(x_1)\frac{\partial}{\partial x_2}\quad \ldots \quad X_n=\varrho_n(x_1,\ldots, x_{n-1})\frac{\partial}{\partial x_n},
 \end{equation}
where a priori we assume that functions $\varrho_i$ are nonnegative real-valued for $i=1,2,\ldots,n$ and $\varrho_1\equiv 1$.

If additionally $\varrho_i$ are differentiable, then the commutators (Lie brackets) $[X_i, X_j]$ can be computed explicitly.

Moreover, if $\varrho_i$ are polynomials then for any $x_0\in \Rn$ there exists the minimal number of the Lie brackets iterations needed to produce a non-zero vector field,  denoted by $r_{x_0}^j$, i.e.
\[
 [X_{i_1}, [X_{i_2}, [\ldots [X_{i_{r_{x_0}^j}}, X_j]\ldots](x_0)\not=0,
\] 
for each $j=1,2,\ldots$. Hence, the H\"ormander condition follows.


The family of vector fields $\{X_1,\ldots, X_n\}$ as in~\eqref{def-vectors} allows us to define the following \emph{Carnot--Cara\-th\'eo\-dory distance} $\dcc$ in $\Rn$: let $\gamma: [0,1]\to \R^n$ be a curve and set $\gamma(t):=(\gamma_1(t),\ldots, \gamma_n(t))$ for $t \in [0,1]$. Then, we define
\[
 d_{CC}(x,y)=\inf_{\gamma} \int_{0}^{1}\sqrt{\gamma_1'(t)^2+\frac{\gamma_2'(t)^2}{\varrho_2^2(\gamma_1(t))}+\ldots+\frac{\gamma_n'(t)^2}{\varrho_n^2(\gamma_1(t),\ldots, \gamma_{n-1}(t))}}\,\ud t,
\]
where the infimum is taken over all absolutely continuous curves $\gamma$ joining points $x$ and $y$, i.e. $\gamma(0)=x$ and $\gamma(1)=y$.

The metric space $(\Rn, \dcc)$ is called the Grushin space, and denoted by $G_n$, in reference to Grushin's work on the hypoelliptic operators, see~\cite{gr1,gr2}. From the viewpoint of the sub-Riemannian geometry the space $(G_n^\alpha, \dcc)$ is a Riemannian manifold outside the singular set $\{x\in \Rn: \Pi _{i=1}^{n} \varrho_i(x)=0\}$, while on the singular set we have two possibilities corresponding with the polynomial nature of the weights. Namely, if the weights are polynomials then the Grushin space is a subriemannian manifold, whereas in the case the weights are not polynomials, the vector fields $X_i$ may fail to satisfy the H\"ormander condition. 

We discuss now some examples of the Grushin spaces.

\begin{ex}\label{ex11}
Let us consider the following weights
\[
 \varrho_i(x_1,\ldots, x_{i-1})=|x_1|^{\alpha_1}\cdots |x_i|^{\alpha_{i-1}},\quad i=1,\ldots,n, \quad \alpha_i\geq 0.
\]
It follows that we can find explicit formulas for distances equivalent to the $\dcc$ distance, see Sections 2.2 and 3 in~\cite{wu}, \cite{fl}, also distances locally equivalent to $d_{CC}$, see~\cite{bel} and~\cite[Section 2]{ly}. 

The space $(G_n^\alpha, \dcc)$ is a Riemannian manifold outside the singular set $\{x\in \Rn: \Pi _{i=1}^{n-1} x_i=0\}$. From that perspective, the Grushin spaces can be regarded as one of the simplest almost-Riemannian manifolds. Recently, the Gushin spaces $G_n^\alpha$ have become a topic of interest in geometric mapping theory and fractal geometry. For example, the quasisymmetric and the bi-Lipschitz equivalence of $G_n^\alpha$ and $\R^n$, also the existence of bi-Lipschitz embeddings of $G_n^\alpha$ in $\R^{n+1}$ are studied in~\cite{wu}; see also~\cite{gjr, rv, wa} for the questions regarding the (quasi)conformal mappings in $G_2^\alpha$.
\end{ex}

The special case of Example~\ref{ex11} for $n=2$ will play the fundamental role in our work and therefore we discuss it separately.

\begin{ex}[Grushin space $G_2^\alpha$]\label{ex12}

Let us consider $\R^2$ with coordinates $(x,y)$ and the following vector fields on $\R^2$
\[ 
 X=\frac{\partial}{\partial x}\quad   Y=|x|^{\alpha}\frac{\partial}{\partial y}.
\]
Let us equip the plane $\R^2$ with the following Carnot--Carath\'eodory distance $d_{CC}$:
\[
 d_{CC}(p,q)=\inf_{\gamma} \int_{0}^{1}\sqrt{\gamma_1'(t)^2+\frac{\gamma_2'(t)^2}{|\gamma_1(t)|^{2\alpha}}}\,\ud t,
\]
where the infimum is taken over all absolutely continuous curves $\gamma$ joining points $p, q\in G_2^\alpha$, i.e. $\gamma(0)=p$ and $\gamma(1)=q$.
%

For any $0<\alpha<1$, there exists a bi-Lipschitz map from $G_2^\alpha$ onto $\R^2$ which maps $S$ onto a von Koch-snowflake curve, see~\cite[Section 2.1]{wu}. Thus, $G_2^\alpha$ and $\R^2$ are bi-Lipschitz equivalent. Furthermore, the Meyerson map gives the quasisymmetry from $G_2^\alpha$ onto $\R^2$, see Formula (5.1) in~\cite{wu} for $n=2$ . 

When discussing Grushin planes, it is common in the literature to assume that $\alpha\geq 1$, i.e. that the weight is Lipschitz regular, see e.g.~\cite{fv, ddfm}. In our work, as well as in our previous studies in~\cite{aww}, we focus our attention on the H\"older regular case $0<\alpha<1$ which is significantly more challenging and difficult.
\end{ex}
%
\medskip

\noindent {\bf Measures on $G_n$.} The Grushin spaces $G_n$ are often considered when equipped with the Lebesgue measure, see e.g.~\cite{bg, fms, fv}. However, it is more natural to study the subriemannian spaces with the intrinsic measure tailored to the underlying distance and reflecting the geometry of the space, and so in this work we take that approach (cf. the discussion in~\cite{aww}).

 The Hausdorff $2$-measure $H_2$ of balls in $G_2^\alpha$ is finite for $0<\alpha<1$ and equals the Riemannian measure $|x|^{-\alpha}\ud x\ud y$ on $G_2^\alpha\setminus S$. Furthermore, $H_2=C(\alpha) \mathcal{H}_{E}^{\frac{2}{1+\alpha}}$ on the singular set $S$ (i.e. on the $y$-axis), where $\mathcal{H}_{E}$ denotes the Euclidean Hausdorff measure. As a matter of fact, it turns out that the Hausdorff measure $H_n$ of balls in $G_n$ is finite for the weights in Example~\ref{ex11}. This result can be established by means of the ball-box technique, see \cite[Theorem 3.1]{wu} and \cite{fl, fl84} for details of that technique and Lemma 1.1 in~\cite{aww} for the result itself. We recall it in the special case of the Grushin planes $G_2^\alpha$. The fact that the following weighted measure $\ud \mu$ is $2$-Ahlfors regular and the weight in $\mu$ is $p$-Muckenhoupt for $p>1$ will be often used in our work.

In what follows balls in the Grushin metric, i.e. balls in the Carnot-Carath\'eodry distance $d_{CC}$, centered at point $x_0$ with radius $r$ will be denoted as follows
\[
B(x_0,r):=B_{G_2^\alpha}(x_0,r)=B_{d_{CC}}(x_0,r). 
\]
\begin{lem}[Lemma 1.1 in~\cite{aww} for $n=2$]\label{lem-Ahl-Muck}
 Let $(G_2^\alpha, d_{CC})$ be a Grushin plane equipped with the Euclidean coordinates $(x,y)$ whose Grushin structure is given by the following power weights 
\begin{equation}
\varrho_1\equiv 1,\quad \varrho_2(x)=|x|^{\alpha},\hbox{ where } \alpha \in [0,1). 
\end{equation}
Then the following measure $\mu$ is $2$-Ahlfors regular in $G_2^\alpha$:
\begin{equation*}
 \mu(B_{G_2^\alpha}(y,r)):=\int_{B_{G_2^\alpha}(y,r)} \frac{\ud x\ud y}{|x|^\alpha},
\end{equation*}
where $B_{G_2^\alpha}(y,r)$ denotes a ball in Grushin metric, centered at $y\in G_2^\alpha$ with radius $r>0$. More precisely, there exists a constant $C(\alpha)>0$ such that $\mu(B(y,r))\approx C(\alpha)r^2$ for any $y\in G_2^\alpha$ and any radius $r>0$.

Moreover, it turns out that the weight in the measure $\mu$ is the $p$-Muckenhoupt weight for any $p>1$.
\end{lem}

\subsection{Sobolev spaces of mappings between Grushin planes}\label{sect12}

Let $G_2^\alpha=(\R^2, d_{CC}, \mu)$ be the Grushin plane, equipped with the weighted measure 
$$
\ud \mu= \frac{\ud x \ud y}{|x|^\alpha}.
$$
Recall that by Lemma 1.1 in~\cite{aww}, if $\alpha\in [0, 1)$, then the measure $\mu$ is $2$-Ahlfors regular and $p$-Muckenhoupt for all $p>1$.

Let $X,Y$ be a family of vector fields in $\Om$ as in Example~\ref{ex12}. Note that on the contrary to a number of results in the literature, our vector fields $X, Y$ need not to be Lipschitz continuous and are, typically, only H\"older continuous on $\R^2$ and $C^1(\R^2\setminus S)$, where $S$ stands for the singular set in $G_n$.
\smallskip
\\
In analogy with the notion of the horizontal gradient in the subriemannian setting, we set
\begin{equation*}
 \nabla_H u:=(X u, Y u)
 \end{equation*} 
 for any measurable function $u\in L^1_{loc}(\Om, \mu)$, where the derivatives $Xu, Yu\in L^1_{loc}$ are understood in the sense of distributions.

The Sobolev spaces in the setting of Grushin planes have been a subject of vivid investigations also in the context of the PDEs on Grushin planes, see for instance~\cite{fhk}, \cite{fs87} and~\cite{fss}. One has a choice between two definitions (``W'' vs.~``H''), the first one mimicking the definition of the Euclidean Sobolev functions, based on the integrability of a function and its horizontal gradient,
and the second one based on the density of the Euclidean Lipschitz-, metric Lipschitz- or smooth functions in the Sobolev norm.
In our previous work we have followed the second approach (with respect to the metric Lipschitz functions), see the discussion in Section 2.2 in~\cite{aww}. However, in the current work we employ the first approach to the Sobolev spaces and relate both approaches by proving the density of smooth functions in the Sobolev norm, see Proposition~\ref{prop-dens} below. Namely,
we say a function $v:\Om\to \R$ given on an open bounded set $\Om\subset G_2^\alpha$ belong to the \emph{Sobolev space $HW^{1,2}(\Om,\mu)$} if
\begin{equation*}
  v \in L^2(\Om,\mu) \hbox{ and } Xv, Yv \in L^2(\Om,\mu),
\end{equation*}
where $Xv, Yv$ are understood in a sense of distributions. We equip the space $HW^{1,2}$ with the following norm, in which  $HW^{1,2}$ is the Hilbert space:
\begin{equation}
\|v\|_{HW^{1,2}(\Om,\mu)}:=\|v\|_{L^2(\Om,\mu)} + \| \nabla_H v\|_{L^2(\Om,\mu)}.\label{def-HW-norm}
\end{equation}

The space of the Sobolev mappings is defined in an analogous way. Namely, if $u=(u^1,u^2):\Om\subset G_2^\alpha\to G_2^\beta$, then we say that $u\in HW^{1,2}(\Om, \mu; G_2^\beta)$ if $u^1,u^2\in HW^{1,2}(\Om,\mu)$. Moreover, the spaces $HW^{1,2}_{loc}$ of functions and mappings, respectively, are defined in the similar, canonical, ways. When discussing harmonic mappings between the Grushin planes we will impose some further conditions which will reflect the fact that the target space is a Grushin plane, see~\eqref{def-sol-space} and~\eqref{def-DH} in Section~\ref{sect14}.


The rest of this section is devoted to the proof of density of $C^{\infty}(\Om) \cap HW^{1,2}(\Om,\mu)$ in $HW^{1,2}(\Om,\mu)$. 
However, before we proceed with the proof, we make the following auxiliary observation. 
\begin{obs}\label{inclusion into Euclidean Sobolev space}
    Let $\alpha \in [0,1)$ and $\Om \subset G_2^{\alpha}$ be an open bounded set. Then,
    $$
     HW^{1,2}(\Om,\mu) \subset W^{1,1}(\Om, \ud x\ud y),
    $$
    where $W^{1,1}$ stands for the Euclidean Sobolev space, and it holds that $u_x = Xu$ and $u_y = |x|^{-\alpha}Yu$ in the distributional sense. 
      
      Moreover, there exists a constant $C = C(\alpha, \Om)$ such that for each $u \in W^{1,1}(\Om)$ it holds that 
      $$
      \|u\|_{W^{1,1}(\Om)} \le C \|u\|_{HW^{1,2}(\Om,\mu)},
      $$
      i.e., the embedding is continuous.
\end{obs}
\begin{proof}
    Fix $u \in HW^{1,2}(\Om,\mu)$. The direct applications of the H\"older inequality yield
    \begin{equation*}
       \|u\|_{L^1(\Om)} \le (\sup_{\Om}|x|^{\alpha}) \mu(\Om)^{\frac{1}{2}}\|u\|_{L^2(\Om,\mu)} < \infty
      \,\, \hbox{ and }\,\, \| Xu \|_{L^1(\Om)} \le (\sup_{\Om}|x|^{\alpha}) \mu(\Om)^{\frac{1}{2}}\|Xu\|_{L^2(\Om,\mu)}.
    \end{equation*}
 Hence, the weak derivative $u_x$ exists and equals to $Xu=\frac{\partial}{\partial x}u$. In order to prove existence of the $u_y$ weak derivative we first observe that, by the H\"older inequality, it holds that:
\begin{equation*}
   \left \|\  |x|^{-\alpha}Yu \right \|_{L^1(\Om, \ud x\ud y)} \le \mu(\Om)^{\frac{1}{2}} \left \|Yu \right \|_{L^2(\Om,\mu)}.
\end{equation*}
Next, let us fix a test function $\varphi \in C_0^{\infty}(\Om)$ and by using the Lebesgue dominated convergence theorem directly compute the following integral:
\begin{align*}
    \int_{\Om} u \varphi_y \ud x \ud y & = \lim_{\delta \to 0^+} \int_{\Om} u \left(\frac{|x|^{\alpha}}{(x^2+\delta)^{\frac{\alpha}{2}}}\right) \varphi_y \ud x \ud y \\
   &= \lim_{\delta \to 0^+} \int_{\Om} u\, Y\!\left(\frac{\varphi}{(x^2+\delta)^{\frac{\alpha}{2}}}\right)  \ud x \ud y
   = -\lim_{\delta \to 0^+}\int_{\Om} (Yu) \frac{\varphi}{(x^2+\delta)^{\frac{\alpha}{2}}} \ud x \ud y = -\int_{\Om} (|x|^{-\alpha}Yu) \varphi \ud x \ud y.
\end{align*}
Here we also use the inequality $(x^2 + \delta)^{-\frac{\alpha}{2}} \le |x|^{-\alpha}$ for $\delta\geq 0$. Hence, the proof of the first assertion of the observation is complete. The continuity of the embedding follows immediately from the above norm estimates for $u_x$ and $u_y$, completing the whole proof.
\end{proof}

\begin{prop}\label{prop-dens}
    Let $\Om \subset G_2^{\alpha}$ be open connected and $\alpha \in[0,1)$. Then, it holds that $C^{\infty}(\Om) \cap HW^{1,2}(\Om,\mu)$ is dense in $HW^{1,2}(\Om,\mu)$.
\end{prop}
\begin{proof}
    Let $u \in HW^{1,2}(\Om,\mu)$ and $\varphi \in C_0^{\infty}(\mathbb{R}^2)$ be a nonnegative radially symmetric function  such that ${\rm supp}\,\varphi \subset B((0,0),1)$ and $\int_{\mathbb{R}^2} \varphi \ud x \ud y=1$. Given $\epsilon>0$ we consider the standard mollifier 
 $$
 \varphi_{\epsilon}(x,y):= \epsilon^{-2}\varphi(\epsilon^{-1} (x,y))\,\hbox{ and set }\, u_{\epsilon}:=u*\varphi_{\epsilon} \in C^{\infty}(\Om).
 $$
  Since the weight $|x|^\alpha$ is $p$-Muckenhoupt for $p>1$, see~\cite[Lemma 1.1]{aww}, and so in particular, $2$-Muckenhoupt, we may apply an approximation result in the weighted Sobolev spaces, see Theorem 2.1.4 in~\cite{tu}, to obtain the following convergence results in $L^2(\Om,\mu)$:
\begin{equation*}
  C^{\infty}(\Om) \cap L^2(\Om,\mu) \ni u_{\epsilon} \to u; \qquad C^{\infty}(\Om) \cap L^2(\Om,\mu) \ni X(u_{\epsilon}) = (X u)_{\epsilon} \to Xu.
\end{equation*}
 Next we consider the operator $Y(u_{\epsilon})$ and upon the direct computations together with Observation~\ref{inclusion into Euclidean Sobolev space} we get the following relation:
\begin{equation}\label{density of smooth functions eq 1}
\begin{aligned}
    (Yu_{\epsilon})(x,y) &= |x|^{\alpha} \partial_y (\varphi_{\epsilon} * u) (x,y) \\
    &= |x|^{\alpha}\int_{\mathbb{R}^2} \partial_y \varphi_{\epsilon}(x-s,y-t)u(s,t)\, \ud s \ud t \\
    &= -|x|^{\alpha}\int_{\mathbb{R}^2} \partial_t \varphi_{\epsilon}(x-s,y-t)u(s,t)\, \ud s \ud t = |x|^{\alpha}\int_{\mathbb{R}^2} \varphi_{\epsilon}(x-s,y-t) u_t(s,t)\, \ud s \ud t\\
    &= |x|^{\alpha}\int_{\mathbb{R}^2} \varphi_{\epsilon}(x-s,y-t) \big (|s|^{-\alpha}Yu(s,t)\big)\, \ud s \ud t \\
    &= |x|^{\alpha}\,\big(\varphi_{\epsilon}*(|\cdot|^{-\alpha} Yu)\big)(x,y) \\
    &= |x|^{\alpha}\,\big(|x|^{-\alpha} Yu\big)_{\epsilon}.
\end{aligned}
\end{equation}
Furthermore, note that the fact $Yu \in L^2(\Om,\mu)$ can be equivalently expressed as $|x|^{-\alpha} Yu \in L^2(\Om,|x|^{\alpha} \ud x \ud y)$ and hence, since $|x|^{\alpha}$ is an $2$-Muckenhoupt weight as well (see e.g. Remark 1.2.4 (4) in~\cite{tu}) and again, by the Theorem 2.1.4 in \cite{tu} we obtain the following convergence in $L^2(\Om,\mu)$:
\begin{equation}\label{density of smooth functions convergence 2}
    C^{\infty}(\Om) \cap L^2(\Om,|x|^{\alpha} \ud x \ud y) \ni (|x|^{-\alpha}  Yu)_{\epsilon} \to |x|^{-\alpha}  Yu.
\end{equation}
Therefore, by the equality \eqref{density of smooth functions eq 1} and convergence~\eqref{density of smooth functions convergence 2} we have
\begin{equation*}
\int_{\Om} |Yu_\epsilon|^2 \ud \mu = \int_{\Om} \big|\big(|x|^{-\alpha} Yu\big)_{\epsilon}\big|^2 |x|^{\alpha} \ud x \ud y<\infty.
\end{equation*}
Finally, we compute that
\begin{align*}
  \int_{\Om}|Yu_{\epsilon}-Yu|^2 \ud \mu &= \int_{\Om}\Big||x|^{\alpha} (|x|^{-\alpha} Yu)_{\epsilon}-Yu\Big|^2 |x|^{-\alpha} \ud x \ud y \\
  & = \int_{\Om} \Big|(|x|^{-\alpha} Yu)_{\epsilon}-|x|^{-\alpha} Yu\Big|^2 |x|^{\alpha} \ud x \ud y \to 0,\,\, \hbox{ as } \ep\to 0^{+}
\end{align*}
 and, hence, the assertion of Proposition~\ref{prop-dens} is proven.
\end{proof}

\begin{rem}\label{rem-C0-density}
Recall the Sobolev space of functions with zero trace $H_{0}^{1,2}(\Om, \mu)$, discussed in Section 2.2 in~\cite{aww}, defined as the closure of ${\rm Lip}_{|\cdot|}(\Om)\cap \mathcal{E}'$ in the $W^{1,2}(\Om, \R)$-norm, where $\mathcal{E}'$ stands for the compactly supported distributions, cf. Section 2 in~\cite{fs87}. Then, Observation 2.1(a) in~\cite{aww} and its proof imply that $C^{\infty}_0(\Om)\subset H^{1,2}_0(\Om)$ for any open set $\Om\subset G_n$. Since the Euclidean Lipschitz functions can be approximated by smooth functions, see Theorem 1 in Chapter 6.6.1 in~\cite{eg}, it turns out that Proposition~\ref{prop-dens} gives, as corollary, that the space $C^{\infty}_0(\Om)$ is dense in $H^{1,2}_0(\Om)$. This discussion also justifies the definition of the Sobolev space of functions $HW^{1,2}_{0}(\Om, \mu)$ as the closure of the $C_0^{\infty}$ functions in the Sobolev norm~\eqref{def-HW-norm}.
\end{rem}

\subsection{Harmonic mappings and the Euler--Lagrange system of equations}\label{sect13}

We begin this section with presenting formal derivation of the Euler--Lagrange system of equations for harmonic mappings between two given open sets in the Grushin spaces of dimension $n$ and $m$, respectively, see~\eqref{system}. Then, we specify the discussion to the main setting of our work, i.e., $n=m=2$ and present the two representations of the system, see~\eqref{system-EL} and~\eqref{str-system-EL}.

Let $\Omega \subset G_n$ be an open set. Moreover, let us assume that the metric and measure structures on given two Grushin spaces $G_n$ and $G_m$, equipped with the coordinate systems $(x_1,\ldots, x_n)$ and $(y_1,\ldots, y_n)$, respectively, are generated by the following vector fields, respectively, cf.~\eqref{def-vectors}:
\begin{equation*}
\begin{aligned}
   &X_1=\varrho_1 \frac{\partial}{\partial x_1}\quad X_2=\varrho_2(x_1) \frac{\partial}{\partial x_2}\quad \ldots\quad X_n=\varrho_n(x_1,x_2,\ldots, x_{n-1}) \frac{\partial}{\partial x_n},\\
   &Y_1=\nu_1 \frac{\partial}{\partial y_1}\quad Y_2=\nu_2(y_1) \frac{\partial}{\partial y_2}\quad \ldots\quad Y_m=\nu_m(y_1,y_2,\ldots, y_{m-1}) \frac{\partial}{\partial y_m}.
\end{aligned}
\end{equation*}
Here, as in Section~\ref{sect11} and throughout the work, we apply the same notation convention and assume that $\varrho_1 \equiv \nu_1 \equiv 1$. The non-Euclidean geometry in the source and the target domains of considered mappings is reflected in the form of the Riemannian norm of the Jacobi matrix of a map and the related Dirichlet energy functional. Moreover, we assume the weighted measure $\ud \mu$ in the source domains in $G_n$, defined as follows (cf. Section~\ref{sect11}):
\[
 \ud \mu=\frac{\ud x_1\cdots \ud x_n}{\varrho_1\,\varrho_2(x_1)\,\cdots\,\varrho_n(x_1,\ldots,x_{n-1})}.
\]
 More specifically, let $u=(u^1(x_1,\ldots,x_n),\ldots, u^m(x_1,\ldots, x_n)): \Omega \to G_m$ be a sufficiently regular map from an open set $\Om$ in the Grushin space $G_n$ into the Grushin space $G_m$. Then, the Riemannian norm of the Jacobi matrix of $u$ is given by the following formula:
\[
 \|D_H u\|_{G_m}^2 =\sum_{i=1}^{n}\varrho_i^2(x)\sum_{j=1}^m\frac{1}{\nu_j^2(u)}\left(\frac{\partial u^j}{\partial x_i}\right)^2,
\]
which in the particularly important for us case $n=m=2$ reads:
\begin{align*}
 \|D_H u\|_{G_2}^2&=\left(\frac{\partial u^1}{\partial x_1}\right)^2+\frac{1}{\nu_2^2(u^1)}\left(\frac{\partial u^2}{\partial x_1}\right)^2+ \varrho_2(x_1)^2\left(\frac{\partial u^1}{\partial x_2}\right)^2+\frac{\varrho_2(x_1)^2}{\nu_2^2(u^1)}\left(\frac{\partial u^2}{\partial x_2}\right)^2\\
&= |\nabla_H u^1|^2+\frac{1}{\nu_2^2(u^1)}|\nabla_H u^2|^2.
\end{align*}
Therefore, the related Dirichlet energy takes the following form:
\begin{equation}\label{def-energy}
    E(u):=\frac{1}{2}\int_{\Omega} \|D_H u\|_{G_m}^2 \ud \mu = \frac{1}{2}\int_{\Omega} \frac{\sum_{i=1}^{n}\varrho_i^2(x)\sum_{j=1}^m\frac{1}{\nu_j^2(u)}\left(\frac{\partial u^j}{\partial x_i}\right)^2}{\Pi_{i=1}^n \varrho_i(x)} \ \ud x_1\cdots \ud x_n.
\end{equation}

The definition of the above energy stems from the theory of harmonic mappings between Riemannian manifolds. Similar Dirichlet energy functionals have been extensively studied in the manifold setting, see e.g.~\cite{es, su1, uh} for classical references and~\cite{bs, sv} for some recent results. In the regular part of a Grushin space the aforementioned energy yields the Laplace--Beltrami system of equations, see~\cite{bl, aww}; a related standard approach relies on the subriemannian geometry and sublaplacians, see~\cite{fv, ly}. 

We are now ready to formally derive the Euler--Lagrange system of equations related to the energy functional~\eqref{def-energy}. In pursuit of this objective, we follow the standard Calculus of Variations technique of inner variations. Namely, given any smooth compactly supported mapping $(\psi_1,\ldots, \psi_m): \Om\to G_m$ we find that
\begin{equation*}
\begin{aligned}
    \frac{d}{dt}&\left[\frac{1}{\nu_j^2(u+t\psi)}\left(\frac{\partial u^j}{\partial x_i} + t\frac{\partial \psi_j}{\partial x_i}\right)^2 \right] \\
    &= \frac{2\left(\frac{\partial u^j}{\partial x_i} + t\frac{\partial \psi_j}{\partial x_i}\right)\frac{\partial \psi_j}{\partial x_i} \nu_j^2(u+t\psi) - 2\nu_j(u+t\psi) \left(\sum_{k=1}^m \psi_k \frac{\partial \nu_j}{\partial y_k}(u+t\psi)\right)\left(\frac{\partial u^j}{\partial x_i} + t\frac{\partial \psi_j}{\partial x_i}\right)^2}{\nu_j^4(u+t\psi)}
\end{aligned}
\end{equation*}
and at $t=0$ we obtain
\begin{equation*}
\begin{aligned}
    \frac{d}{dt}&\left[\frac{1}{\nu_j^2(u+t\psi)}\left(\frac{\partial u^j}{\partial x_i} + t\frac{\partial \psi_j}{\partial x_i}\right)^2 \right]|_{t=0}
    =\frac{2\frac{\partial u^j}{\partial x_i}\frac{\partial \psi_j}{\partial x_i} \nu_j^2(u) - 2\nu_j(u) \left(\sum_{k=1}^m \psi_k \frac{\partial \nu_j}{\partial y_k}(u)\right)\left(\frac{\partial u^j}{\partial x_i}\right)^2}{\nu_j^4(u)}.
\end{aligned}
\end{equation*}
We apply this computation in the definition of the Euler--Lagrange system of equations to arrive at the following expression:
\begin{equation*}
\begin{aligned}
    \frac{d}{dt}E(u+t\psi)|_{t=0} 
    &= \sum_{j=1}^m\int_{\Omega} \frac{\sum_{i=1}^{n}\varrho_i^2(x)\left[\frac{\partial u^j}{\partial x_i}\frac{\partial \psi_j}{\partial x_i} \nu_j^2(u) - \nu_j(u) \left(\sum_{k=1}^m \psi_k \frac{\partial \nu_j}{\partial y_k}(u)\right)\left(\frac{\partial u^j}{\partial x_i}\right)^2\right]}{\nu_j^4(u) \Pi_{i=1}^n \varrho_i(x) } \ud x\\
    &= \sum_{j=1}^m\int_{\Omega} \frac{\langle \nabla_H u^j, \nabla_H \psi_j\rangle_{\mathbb{R}^n} \nu_j^2(u) - \nu_j(u) \left(\sum_{k=1}^m \psi_k \frac{\partial \nu_j}{\partial y_k}(u)\right)\lvert  \nabla_H u^j \rvert^2}{\nu_j^4(u) \Pi_{i=1}^n \varrho_i(x) } \ud x\\
    &=-\sum_{j=1}^m \int_{\Omega} \textrm{div}_{G_n}\left[\frac{\nabla_H u^j}{\nu_j^2(u) \Pi_{i=1}^n \varrho_i(x)}\right] \psi_j \ dx - \sum_{k=1}^m \int_{\Omega} \sum_{j=1}^m \frac{\lvert \nabla_H u^j \rvert^2 \frac{\partial \nu_j}{\partial y_k}(u)}{\nu_j^3(u) \Pi_{i=1}^n \varrho_i(x)}  \psi_k \ud x.
\end{aligned}
\end{equation*}
Therefore, the Euler--Lagrange system of equations reads 
\begin{equation}\label{system}
    \textrm{div}_{G_n}\left[\frac{\nabla_H u^j}{\nu_j^2(u) \Pi_{i=1}^n \varrho_i(x)}\right] + \sum_{k=1}^{m} \frac{\lvert \nabla_H u_k \rvert^2 \frac{\partial \nu_k}{\partial y_j}(u)}{\nu_k^3(u)\Pi_{i=1}^n \varrho_i(x)} = 0, \quad \textrm{for all } j=1,2,\ldots,m.
\end{equation}

Let us consider the case $n=m=2$ and the following vector fields in the Grushin plane $G_2^\alpha$ and, respectively, second Grushin plane $G_2^\beta$:
\begin{align}
 &X_1=\frac{\partial}{\partial x}\quad  Y_1=|x|^{\alpha}\frac{\partial}{\partial y},\quad \alpha\geq 0, \label{def-XYvectors}\\
 &X_2=\frac{\partial}{\partial x}\quad  Y_2=|x|^{\beta}\frac{\partial}{\partial y},\quad \beta\geq 0. \nonumber
\end{align}
Then the system of equations~\eqref{system} reads as follows, whenever the regularity of functions $u^1$ and $u^2$ allows it:
\begin{equation}\label{system-EL}
\begin{cases}
{\rm div}_{G_2^{\alpha}}\left(\frac{\nabla_H u^1}{|x|^\alpha}\right)+\beta \frac{|\nabla_H u^2|^2 u^1}{|x|^\alpha |u^1|^{2\beta+2}}=0 \\
{\rm div}_{G_2^{\alpha}}\left(\frac{\nabla_H u^2}{|x|^\alpha|u^1|^{2\beta}}\right)=0.
\end{cases}
\end{equation}
Similarly, by assuming the appropriate regularity of the expressions in~\eqref{system-EL} we may write out the following strong form of the above system ($\Delta_{G_2^{\alpha}}=X_1^2+Y_1^2$):
\begin{equation}\label{str-system-EL}
\begin{cases}
|u^1|^{2\beta+2}\Big(\Delta_{G_2^{\alpha}} u^1 -\alpha \frac{x}{|x|^2}u^1_x\Big) +\beta u^1|\nabla_H u^2|^2=0 \\
|u^1|^2\Big(\Delta_{G_2^{\alpha}} u^2  - \alpha \frac{x}{|x|^2}u^2_x \Big) - 2\beta u^1 \langle \nabla_H u^1, \nabla_H u^2 \rangle=0.
\end{cases}
\end{equation}

Below we refer to solutions of the system of equations~\eqref{system-EL}, under the appropriate regularity assumptions, as to \emph{the weak harmonic mappings}, and in the case of the system of equations~\eqref{str-system-EL} as \emph{the strong harmonic mappings}, respectively, see Definitions~\ref{weak-solution} and~\ref{str-solution}. 

\subsection{Existence of minimizers of the Dirichlet energy}\label{sect14}
The purpose of this section is to discuss the existence of minimizers in the setting of mappings between the Grushin planes, which we show by modifying the standard direct methods. However, the uniqueness of minimizers turns out to be more delicate matter which we only comment below and postpone till further studies.

Recall the energy functional~\eqref{def-energy} specified for mappings $u = (u^1,u^2):\Om \subset G_2^{\alpha} \to G_2^{\beta}$ with $\alpha \in [0,1)$, $\beta \ge 0$ and for an open connected set $\Om$:
\begin{equation}\label{def-energy-2}
    E(u) = \frac{1}{2}\int_{\Om}\Big (|\nabla_H u^1|^2 + \frac{|\nabla_H u^2|^2}{|u^1|^{2\beta}}\Big) \ud \mu 
\end{equation}
When $\beta =0$, then we retrieve the energy functional for mappings from open sets in $G_2^\alpha$ to $\R^2$, previously considered in~\cite{aww}. 

In order to analyze the energy $\eqref{def-energy-2}$ let us introduce the following auxiliary energy functional
defined for $\alpha \in [0,1)$, $\beta \ge 0$ and the parameter $M>0$ whose role is to quantify the cut-off level for a weight $|u^1|^{-2\beta}$ and, thus, to control the growth of the integrand when $u^1$ grows unbounded:
\begin{equation}\label{M dirichlet energy for maps between grushin planes}
    E^M(u) := \int_{\Om}\left( |\nabla_H u^1|^2 + \max\{M,  |u^1|^{-2\beta}\}|\nabla_H u^2|^2 \right) \ud \mu.
\end{equation}

We clearly see that if $u^1 \in L^{\infty}(\Om)$ and $u^1$ is not identically zero in $\Om$, then letting $M = \|u^1\|_{L^\infty(\Om)}^{-2\beta}$ yields $E^M(u) = E(u)$. 

\begin{remark}
(1) For our studies of harmonic mappings as the critical points of the Dirichlet energy~\eqref{def-energy-2} the natural function setting is the one of the Sobolev mappings $HW^{1,2}(\Om, G_2^\beta)$. However, for such mappings the energy $E^M(u)$ may not always be finite, in general. Nevertheless, it is finite if we require that the image of a map $u$ 
does not intersect the singular line in $G_2^\beta$, i.e., if we impose an additional assumption that $|u^1| \ge \delta > 0$ for some constant $\delta>0$. 
\smallskip

\noindent (2) There exists yet another wide class of Grushin harmonic mappings for which it holds that $E^M<\infty$ for the appropriate choice of $M>0$. Namely, let $u$ be a harmonic map obtained as the composition of the canonical quasisymmetries $\phi_\alpha, \phi_\beta$ between the Grushin planes $G_2^\alpha$, respectively, $G_2^\beta$ and the Euclidean planes and a holomorphic function $h:\varphi_{\alpha}(\Om) \to \mathbb{C}$, see~\eqref{map-mey} and the discussion in Section~\ref{sect2} for more detailed studies of such Grushin harmonic mappings:
\begin{equation*}
    u = \varphi_{\beta}^{-1} \circ h \circ \varphi_{\alpha}.
\end{equation*}
Then for $\beta \in [0,1)$ and for each $\Om' \subset \subset \Om$ we have that $E^M(u)<\infty$ on $\Om'$ and the proof is analogous to the proof of Proposition~\ref{prop-holom}. Indeed, the $L^2_{\textrm{loc}}(\Om,\mu)$ integrability of $|\nabla_H u^1|$ is part of the assertion of the proposition, whereas the $L^2_{\textrm{loc}}(\Om,\mu)$ integrability of $|u^1|^{-\beta} |\nabla_H u^2|$ follows from the Grushin--Cauchy--Riemann system of equations~\eqref{CR equations in Grushin setting}.
\smallskip

\noindent (3) One may find a variety of examples of maps $u=(u^1,u^2) \in HW^{1,2}(\Om, \mu; G_2^{\beta})$ such that $E^M(u) = \infty$. Indeed, it suffices that $|u^1|^{-2\beta} \notin L^1(\Om,\mu)$ and there exists $\delta>0$ such that $|\nabla_H u^2| \ge \delta$ a.e. in $\Om$. For example, let
$$
 u^1(x,y) := y^n,\quad u^2(x,y) := x\quad \hbox{ for }n \in \N\,\hbox{ such that }\,2\beta n \ge 1.
$$
 One directly verifies that $u \in HW^{1,2}(\Om,G_2^{\beta})$ and $E^M(u) = \infty$. Furthermore, similar examples arise by letting $u(x,y)= (W(x,y),Q(x,y))$, with polynomial components $W,Q$, satisfying $E^M(u) = \infty$.
\end{remark}

We now proceed to the main result of this section. 

\begin{theorem}\label{thm-exist-min}
Let $\alpha \in [0,1)$, $\beta \ge 0$, $M>0$ and $\Om \subset G_2^{\alpha}$ be a domain in $G_2^\alpha$ with the (Euclidean) Lipschitz boundary $\partial \Om$. Let further  $u_0 \in HW^{1,2}(\Om,\mu; G_2^\beta)$ such that $E^M(u_0) < \infty$, where $E^M$ denotes the energy functional in~\eqref{M dirichlet energy for maps between grushin planes}. Then the following minimization problem admits a solution:
\begin{equation*}
    \inf{\{E^M(u) \ | \ u-u_0 \in HW^{1,2}(\Om,\mu; G_2^\beta)\}}.
\end{equation*}
\end{theorem}
\begin{proof}
The proof is inspired by the idea of the $\Gamma$--convergence, see for instance~\cite{br} for a detailed description. 

Let map $u_0 \in HW^{1,2}(\Om,\mu; G_2^\beta)$ be such that $E^M(u_0) < \infty$. We associate with the energy functional $E^M$ a sequence of auxiliary energies, show the existence for each of such energies and then study the convergence of the corresponding minimizers.
\smallskip

\noindent {\bf Claim 1:}
\emph{ Let $M>0$ and $\ep>0$ and consider the following approximate energy functional
\begin{equation}\label{approximate dirichlet energy between grushin planes}
  E_{\epsilon}^M(u) := \int_{\Om} \left(|\nabla_H u^1|^2 + \max{\{(|u^1|^2+\epsilon)^{-\beta},M\}}|\nabla_H u^2|^2\right) \ud \mu.
\end{equation}
Then the following minimization problem admits a solution in $HW^{1,2}(\Om,\mu; G_2^\beta)$:
\begin{equation*}
    \inf{\{E_{\epsilon}^M(u) \ | \ u-u_0 \in HW^{1,2}(\Om,\mu; G_2^\beta)\}}.
\end{equation*}
}
 The proof of Claim 1 is a standard adaptation of Calculus of Variations direct methods, see e.g. Chapter 3 in~\cite{da} or Section 2 in \cite{aww} for the Grushin space setting. It amounts to showing weak lower semicontinuity of $E_{\ep}^M$ and the existence of the minimizing sequence, see Lemma 2.4 and Theorem 2.6 in~\cite{aww} for similar reasoning for the Euclidean target space.

Fix a sequence of approximate energies $(E^M_{\epsilon_k})$ with $\ep_k \to 0$ as $k\to \infty$ and, by  Claim 1, consider the corresponding sequence of Sobolev minimizers $(u_{\epsilon_k})$. Then, the inequality $((u^1_0)^2 + \epsilon)^{-\beta} \le |u^1_0|^{-2\beta}$ holding for all $\epsilon>0$, allows us to infer that
\begin{equation}\label{existence of minimisers M energy pf 1}
    E^M_{\epsilon_k}(u_{\epsilon_k}) \le E^M_{\epsilon}(u_0) \le E^M(u_0) <\infty,\quad \hbox{ for each }k=0,1\ldots.
\end{equation}
Hence, $\liminf_{\epsilon_k \to 0^+}E^M_{\epsilon_k}(u_{\epsilon_k}) < \infty$. By possibly extracting a subsequence, we assume that
\[
\lim_{\epsilon_k \to 0^+}E^M_{\epsilon_k}(u_{\epsilon_k}) = \liminf_{\epsilon_k \to 0^+}E^M_{\epsilon_k}(u_{\epsilon}).
\]
As in the standard proof of existence of the minimizing sequence we observe that 
$$
\sup_{\epsilon_k} \| u_{\epsilon_k}\|_{HW^{1,2}(\Om,\mu; G_2^\beta)} < \infty.
$$
Therefore, by the reflexivity of $HW^{1,2}(\Om,\mu; G_2^\beta)$ together with Observation \ref{inclusion into Euclidean Sobolev space} and the Rellich--Kondrachov embedding theorem, we extract a further subsequence of minimizers, denoted for sake of simplicity again by  $(u_{\epsilon_k})$, such that $u_{\epsilon_k} \rightharpoonup \bar{u}$ weakly in $HW^{1,2}(\Om,\mu; G_2^\beta)$ and $u_{\epsilon_k} \to \bar{u}$ a.e. in $\Om$. Moreover, the standard reasoning involving the Mazur Lemma allows us to infer that $\bar{u} - u_0 \in HW^{1,2}_0(\Om,\mu)$. Next, the weak lower semicontinuity of the norm gives us the following estimate:
\begin{equation}\label{existence of minimisers M energy pf 2}
    \int_{\Om} |\nabla_H \bar{u}^1|^2 \ud \mu \le \liminf_{k \to \infty} \int_{\Om} |\nabla_H u^1_{\epsilon_k}|^2 \ud \mu. 
\end{equation}
Furthermore, it holds that 
\begin{equation}\label{existence of minimisers M energy pf 3}
    \int_{\Om} \max{\{|\bar{u}^1|^{-2\beta},M\}}|\nabla_H \bar{u}^2|^2 \ud \mu \le \liminf_{k \to \infty} \int_{\Om} \max{\{((u^1_{\epsilon_k})^2 + \epsilon)^{-\beta},M\}}|\nabla_H u^2_{\epsilon_k}|^2 \ud \mu.
\end{equation}
This observation follows directly from the next claim:
\smallskip

\noindent {\bf Claim 2:} \emph{  Let $\Om \subset G_2^{\alpha}$ be an open, bounded set. Let $v_n \to v$ a.e. in $\Om$ and $u_n \rightharpoonup u$ weakly in $L^2(\Om,\mu)$. Then
    \begin{equation*}
        \liminf_{n \to \infty} \| v_n u_n\|_{L^{2}(\Om,\mu)} \ge \| vu \|_{L^2(\Om,\mu)}.
    \end{equation*}
}
In order to prove Claim 2 we fix $M>0$ and associate with the sequence $(v_n)$ the following sequence of cut-off functions:
    \begin{equation*}
    v_n^M := \begin{cases}
    v_n \qquad \qquad \ \ \quad |v_n| \le M,\\
    M\sgn{(v_n)} \qquad |v_n| > M,
    \end{cases}
    \end{equation*}
and we define $v^M$ analogously. For a given $\varphi \in L^2(\Om,\mu)$ the direct computations gives us the convergence of following sequence:
\begin{equation*}
\begin{aligned}
    \Big| \int_{\Om} v_n^M u_n \varphi \ud \mu - \int_{\Om} v^M u \varphi \ud \mu \Big| &\le \Big| \int_{\Om} v_n^M u_n \varphi \ud \mu - \int_{\Om} v^M u_n \varphi \ud \mu \Big| + \Big| \int_{\Om} v^M u_n \varphi \ud \mu - \int_{\Om} v^M u \varphi \ud \mu \Big|\\
    & = \Big| \int_{\Om} (v_n^M - v^M)u_n \varphi \ud \mu \Big| + \Big| \int_{\Om} (u_n - u) v^M \varphi \ud \mu \Big| \to 0.
\end{aligned}
\end{equation*}
In the last step we also use the Lebesgue dominated convergence theorem and the weak convergence $u_n \rightharpoonup u$ weakly in $L^2(\Om,\mu)$. Hence, $v_n^Mu_n \rightharpoonup v^Mu$ weakly in $L^2(\Om,\mu)$. Moreover, by the weak semicontinuity of the norm and by the definition of $v_n^M$ we get that:
\begin{equation*}
    \liminf_{n \to \infty} \| v_n u_n \|_{L^2(\Om,\mu)} \ge \liminf_{n \to \infty} \| v_n^Mu_n \|_{L^2(\Om,\mu)} \ge \|v^M u\|_{L^2(\Om,\mu)}.
\end{equation*}
Finally, by appealing to the Lebesgue monotone convergence theorem we have that $\|v^M u\|_{L^2(\Om,\mu)} \nearrow \|v u\|_{L^2(\Om,\mu)}$ as $M \nearrow \infty$, we conclude the proof of Claim 2.
\smallskip

\noindent By combining \eqref{existence of minimisers M energy pf 1}, \eqref{existence of minimisers M energy pf 2} and \eqref{existence of minimisers M energy pf 3} we show the finiteness of the energy $E^M(\bar{u})$: 
\begin{equation}\label{existence of minimisers M energy pf 4}
\begin{aligned}
    E^M(\bar{u}) &\le \liminf_{k \to \infty}\int_{\Om} |\nabla_H u^1_{\epsilon_k}|^2 \ud \mu + \liminf_{k\to \infty}\int_{\Om} \max{\{((u^1_{\epsilon_k})^2 + \epsilon)^{-\beta},M\}}|\nabla_H u^2_{\epsilon_k}|^2 \ud \mu \\
    &\le \liminf_{k \to \infty} E_{\epsilon_k}^M(u_{\epsilon_k}) < \infty.
\end{aligned}
\end{equation}
Conversely, analogous reasoning as in \eqref{existence of minimisers M energy pf 1}, implies that:
\begin{equation*}
   E^M_{\epsilon_k}(u_{\epsilon_k}) \le E^M_{\epsilon_k}(\bar{u}) \le E^M(\bar{u}).
\end{equation*}
Whence by the latter estimates and by~\eqref{existence of minimisers M energy pf 4} it holds:
\begin{equation}\label{existence of minimisers M energy pf 5}
    \liminf_{k \to \infty}E^M_{\epsilon_k}(u_{\epsilon_k}) \le \limsup_{k \to \infty}E^M_{\epsilon_k}(u_{\epsilon_k}) \le E^M(\bar{u}) \le \liminf_{k \to \infty}E^M_{\epsilon_k}(u_{\epsilon_k}) < \infty
\end{equation}
and, therefore, $\lim_{k \to \infty}E_{\epsilon_k}^M(u_{\epsilon_k}) = E^M(\bar{u})$.

On the other hand, fix any $u \in HW^{1,2}(\Om,\mu)$ such that $u - u_0 \in HW^{1,2}(\Om,\mu)$. Then again, by the same reasoning we obtain that $E^M_{\epsilon_k}(u_{\epsilon_k}) \le E^M(u)$ for every $k \in \mathbb{N}$. This combined with the inequality \eqref{existence of minimisers M energy pf 5} yields $E^M(\bar{u}) \le E^M(u)$, as claimed and completes the proof of Theorem~\ref{thm-exist-min}.
\end{proof}

\subsection{Weak harmonic mappings and the Caccioppoli estimates}\label{sect15}

 Let $\Om\subset G_2^{\alpha}$ be an open bounded subset in the Grushin plane $G_2^{\alpha}$ and $\phi=(\phi^1, \phi^2)\in C^{\infty}_{0}(\Om,\R^2)$ be a vector-valued test function. In order to derive the weak formulation of the system~\eqref{system-EL} we multiply its both sides by the test functions $\phi^1$ and $\phi^2$, respectively, and perform integration by parts assuming for a moment that it is possible. As a consequence we arrive at the following system of equations which hold for any $\phi^1, \phi^2 \in C^{\infty}_{0}(\Om,\R)$:
 \begin{equation}\label{w-system-EL}
\begin{cases}
-\int_{\Om} \left\langle \frac{\nabla_H u^1}{|x|^\alpha}, \nabla_H \phi^1 \right \rangle\,\ud x \ud y+\beta \int_{\Om} \frac{|\nabla_H u^2|^2 u^1}{|x|^\alpha |u^1|^{2\beta+2}}\phi^1\,\ud x \ud y=0 \\
-\int_{\Om} \left \langle \frac{\nabla_H u^2}{|x|^\alpha|u^1|^{2\beta}}, \nabla_H \phi^2 \right \rangle\,\ud x \ud y=0.
\end{cases}
\end{equation}

Next, we discuss and justify the regularity assumptions for a map $u=(u^1,u^2)$ in order to be a solution to the system~\eqref{w-system-EL}. 

Recall that the weighted measure $\ud \mu$ is defined as follows $\ud \mu=\frac{\ud x \ud y}{|x|^\alpha}$ for $\alpha\geq 0$. However, if $0\leq \alpha \leq 1$, then the weight $|x|^{-\alpha}$ is locally integrable in $\R^2$ and, moreover, the measure $\mu$ is the $A_p$ Muckenhoupt measure. Nevertheless, if $\alpha\geq 1$, then $\mu$ is also Muckenhoupt on sets non-intersecting the singular set, i.e. outside the $y$-axis. Let us observe that if we are not interested in $\mu$-integrable constant functions, then lack of the $L^1_{loc}-$ property of the weight is not an obstacle for our studies. Indeed, all the terms in the above formulations of the harmonic system of equations \eqref{system-EL}-\eqref{w-system-EL} become zero for a constant mapping. This discussion allows us to formulate the following definition of the weak solutions.

Suppose that a mapping $u=(u^1,u^2):\Om\subset G_2^{\alpha}\to G_2^\beta$ satisfies
\begin{equation}\label{def-sol-space}
  u^1, u^2\in L^\infty_{loc}(\Om),\quad D_H u \in L^2_{loc}(\Om, \ud \mu),
\end{equation}
where $\alpha \geq 0$. 
The latter condition equivalently reads as follows:
\begin{equation}\label{def-DH}
 \|D_H u\|^2_{G^2_\beta}:=|\nabla_H u^1|^2+\frac{|\nabla_H u^2|^2}{|u^1|^{2\beta}}\in L^1_{loc}(\Om,\ud \mu).
\end{equation}
Let us discuss when solutions to the system~\eqref{w-system-EL} are well-defined. If we impose the integrability assumptions~\eqref{def-sol-space} and that 
\[
\int_{\Om} \frac{|\nabla_H u^2|^2}{|u^1|^{4\beta}} \ud \mu<\infty, \tag{A1}
\]
then both of the following divergence operators in~\eqref{w-system-EL} are defined in the weak sense:
\begin{align*}
 &{\rm div}_{G_2^{\alpha}}\left(\frac{\nabla_H u^1}{|x|^\alpha}\right):=-\int_{\Om} \left\langle \frac{\nabla_H u^1}{|x|^\alpha}, \nabla_H \phi^1 \right \rangle \ud x \ud y \leq \|\nabla_H u^1\|_{L^2(\Om,\mu)}\|\nabla_H \phi\|_{L^2(\Om,\mu)}\\
 & {\rm div}_{G_2^{\alpha}}\left(\frac{\nabla_H u^2}{|x|^\alpha|u^1|^{2\beta}}\right):=-\int_{\Om} \left \langle \frac{\nabla_H u^2}{|x|^\alpha|u^1|^{2\beta}}, \nabla_H \phi^2 \right \rangle \ud x \ud y \leq \left \|\frac{|\nabla_H u^2|}{|u^1|^{2\beta}} \right \|_{L^2(\Om,\mu)}\|\nabla_H \phi\|_{L^2(\Om,\mu)}.
\end{align*}
Moreover, in order to ensure the finiteness of the remaining integral in the first equation of the system~\eqref{w-system-EL} we assume that 
\[
\int_{\Om} \frac{|\nabla_H u^2|^2}{|u^1|^{2\beta+1}} \ud \mu<\infty. \tag{A2}
\]

The relation between the two introduced integrability conditions depends on $\beta$. Namely, since $u^1\in L^{\infty}_{loc}$ then for $0\leq \beta \leq \frac12$ it holds that (A2) implies (A1), while if $\beta>\frac12$ then (A1) implies (A2). Therefore, we introduce the following assumption for a map $u$ as in~\eqref{def-sol-space}:
\begin{equation}\label{ass-lem-HW0}
\int_{\Om} \frac{|\nabla_H u^2|^2}{|u^1|^{2\beta+1}} \ud \mu<\infty \quad \hbox{for }\,0\leq \beta\leq \frac12;\qquad 
\int_{\Om} \frac{|\nabla_H u^2|^2}{|u^1|^{4\beta}} \ud \mu<\infty \quad \hbox{for }\, \beta> \frac12.
\end{equation} 

The above discussion leads us to the following definition of Grushin-harmonic mappings.
 
\begin{defn}[weak harmonic mappings]\label{weak-solution} 
 Let $\alpha, \beta>0$ and $\Omega \subset G_2^{\alpha}$ be open and bounded. We say that a mapping $u=(u^1,u^2):\Omega \to G_2^{\beta}$, such that $u \in HW^{1,2}(\Omega,\mu; G_2^\beta)$ and $u^1, u^2 \in L^{\infty}_{loc}(\Omega)$ is a \textit{weak harmonic mapping}, if it satisfies the system of equations~\eqref{w-system-EL} in the weak sense and the integrability condition~\eqref{ass-lem-HW0} holds.
\end{defn}
 
Since it is often convenient to deal with the larger class of test functions than $C_0^{\infty}$, for example $HW_{0}^{1,2}$, we will now investigate conditions allowing us to test the system~\eqref{w-system-EL} with the Sobolev test functions.
\begin{lem}\label{lem-HW0}
 Let a map $u:\Om\subset G_2^\alpha\to G_2^\beta$ be as in~\eqref{def-sol-space}, i.e.
\[
 u^1, u^2\in L^\infty_{loc}(\Om),\quad D_H u \in L^2_{loc}(\Om, \ud \mu)
\] 
and satisfy the integrability condition~\eqref{ass-lem-HW0}. 

If $u$ is a solution of the system~\eqref{w-system-EL} for test functions in $C_{0}^{\infty}(\Om, \R^2)$, then so is when testing~\eqref{w-system-EL} with compactly supported mappings $\phi=(\phi^1,\phi^2)$ such that 
\begin{equation}\label{ass-lem-HW}
 \phi^1, \phi^2\in L^\infty_{loc}(\Om) \hbox{  and }|\nabla_H \phi^1|, |\nabla_H \phi^2| \in L^2_{loc}(\Om, \ud \mu)
\end{equation}
provided that either $\alpha\geq 0$ and $u$ is a solution on an open bounded subset $\Om'\Subset\Om \setminus S$, or $0\leq \alpha <1$.
\end{lem}
\begin{proof}
 Let $\phi$ be a compactly supported function as in the assertion~\eqref{ass-lem-HW} of the lemma. Suppose that the sequence $(\phi_i)=(\phi_i^1,\phi_i^2)$ of functions in $C^\infty_{0}(\Om,\R^2)$ is such that 
 \[
  \|\phi-\phi_i\|_{L^\infty(\Om)}\to 0,\quad \|\nabla_H \phi^j- \nabla_H \phi^j_i\|_{L^2(\Om,\mu)}\to 0\quad \hbox{ for }j=1,2\hbox{ as }i \to \infty.
 \]
The first convergence is the direct application of the approximation of bounded compactly supported functions by compactly supported smooth functions. The latter approximation is available due to Remark~\ref{rem-C0-density}, when $\alpha\in [0,1)$. For general $\alpha\geq 0$, the similar approximation can be inferred on an open bounded set $\Om'\Subset \Om \setminus S$, since outside set $S$ the Sobolev spaces $HW^{1,2}$ and the Euclidean $W^{1,2}$ are equivalent.

Upon substituting functions $\phi^1$ and $\phi^1_i$ into the first equation of the system~\eqref{w-system-EL} and subtracting the corresponding equations we obtain the following estimate
\begin{align*}
&\int_{\Om} \frac{|\nabla_H u^1|}{|x|^{\alpha}} |\nabla_H (\phi^1-\phi^1_i)|\,\ud x \ud y
+\beta \int_{\Om} \frac{|\nabla_H u^2|^2}{|x|^\alpha |u^1|^{2\beta+1}}|\phi^1-\phi^1_i| \,\ud x \ud y \\
& \leq  \|\nabla_H u^1\|^2_{L^2(\Om,\mu)}\|\nabla_H (\phi^1-\phi^1_i)\|^2_{L^2(\Om,\mu)}+\beta  \left \|\frac{|\nabla_H u^2|^2}{|u^1|^{2\beta+1}} \right \|^2_{L^2(\Om,\mu)}\|\phi^1-\phi^1_i\|^2_{L^\infty(\Om)}\to 0,\quad \hbox{ as }i \to \infty,
\end{align*}
by the choice of the sequence $(\phi^1_i)$ and conditions~\eqref{def-sol-space} and~\eqref{ass-lem-HW0}. Similarly, we show the convergence for the second equation in~\eqref{w-system-EL} and, thus, complete the proof of Lemma~\ref{lem-HW0}.
\end{proof}

Therefore, under the assumptions of Lemma~\ref{lem-HW0}, it is now justified to consider the functions $\phi^i:=u^i\eta^2$ for $\eta\in C_0^{\infty}(\Om,\R)$ and $i=1,2$ and it holds that both $\phi^i\in HW_{0}^{1,2}(\Om, \R)$.

We are in a position to derive the key energy estimates for the weak Grushin harmonic mappings.
\begin{lem}[Caccioppoli inequality 1]\label{lem-Cac}
Let $\Om\subset G_2^\alpha$ be an open set in the $\alpha$-Grushin plane such that the exponent $\alpha$ satisfies that
either
\begin{itemize}
\item $\alpha\in[0,1)$, or
\item $\alpha\geq 0$ and it holds that $\overline{\Om}\cap S \not=\emptyset$. 
\end{itemize}
Moreover, let $u=(u^1, u^2) \in HW^{1,2}(\Om, \mu; G_2^\beta)$ be a weak harmonic mapping, i.e., $u^1, u^2\in L^\infty_{loc}(\Om)$ satisfy the system~\eqref{w-system-EL} and condition~\eqref{ass-lem-HW0}. Then the following Caccioppoli inequality holds for any $\eta \in C_0^{\infty}(\Om,\R)$ such that $0\leq \eta \leq 1$
 \begin{equation}\label{first-Cac}
   \int_{\Om} \|D_H u(x)\|_{G_2^\beta}^2 \eta^2(x) \ud \mu(x) \leq c(\beta) \int_{\Om}\|u(x)\|^2_{G_2^\beta} |\nabla_H \eta(x)|^2 \ud \mu(x),
 \end{equation}
 where $c(\beta)=4(1+\max\{1, 2\beta\})$ and $\|u\|_{G_2^\beta}$ stands for the Riemannian length of the vector $u=(u^1,u^2)$ in the target space $G_2^\beta$, i.e.
 \begin{equation}\label{defn-subr-unorm}
  \|u(x)\|_{G_2^\beta}^2=|u^1(x)|^2+\frac{|u^2(x)|^2}{|u^1(x)|^{2\beta}},\quad x\in \Om\subset G_2^\alpha,\quad u^1(x)\not=0.
 \end{equation}
 \end{lem}
 Note that when defining  $\|u(x)\|_{G_2^\beta}$ we slightly abuse notation used previously for introducing $\|D_H u(x)\|_{G_2^\beta}$, see~\eqref{def-DH}.  
\begin{proof}
Upon applying the test functions $\phi^i=u^i\eta^2$ for $\eta\in C_0^{\infty}(\Om,\R)$ such that $0\leq \eta \leq 1$ in $\Om$ and $i=1,2$, respectively,  in the first and the second equation of the weak formulation~\eqref{w-system-EL} we obtain the following system of inequalities:
 \begin{equation}\label{aux1-lem-Cac}
\begin{cases}
 \int_{\Om} |\nabla_H u^1|^2 \eta^2 \ud \mu\leq 4\int_{\Om}|u^1|^2 |\nabla_H \eta|^2 \ud \mu +2\beta \int_{\Om} \frac{|\nabla_H u^2|^2}{ |u^1|^{2\beta}}\eta^2 \ud \mu \\
 \int_{\Om} \frac{|\nabla_H u^2|^2}{|u^1|^{2\beta}} \eta^2 \ud \mu \leq 4\int_{\Om}\frac{|u^2|^2}{|u^1|^{2\beta}} |\nabla_H \eta|^2 \ud \mu.
\end{cases}
\end{equation}
Observe that integrability condition~\eqref{ass-lem-HW0} together with the $L^2(\ud \mu)$-integrability of $u^1$ imply that the both inequalities above are well-defined. Indeed, by the H\"older inequality it holds that
\[
  \int_{\Om} \frac{|\nabla_H u^2|^2}{|u^1|^{2\beta}} \eta^2 \ud \mu \leq \int_{{\rm supp }\,\eta} \frac{|\nabla_H u^2|^2}{|u^1|^{2\beta+1}} |u^1| \ud \mu \leq \|u^1\|_{L^\infty({\rm supp }\,\eta)}\,\int_{\Om} \frac{|\nabla_H u^2|^2}{|u^1|^{2\beta+1}}\ud \mu<\infty.
\]
Next, we apply the second estimate in~\eqref{aux1-lem-Cac} in the first one and get the following inequality:
\begin{equation}\label{aux2-lem-Cac} 
 \int_{\Om} |\nabla_H u^1|^2 \eta^2 \ud \mu \leq 4\max\{1, 2\beta\} \int_{\Om} \Big(|u^1|^2+\frac{|u^2|^2}{|u^1|^{2\beta}}\Big) |\nabla_H \eta|^2 \ud \mu.
\end{equation}
Finally, we add up inequality~\eqref{aux2-lem-Cac} and the second inequality in~\eqref{aux1-lem-Cac}  to arrive at~\eqref{first-Cac}:
\begin{align*}
  \int_{\Om} |D_H u|^2 \eta^2 \ud \mu & = \int_{\Om} \Big( |\nabla_H u^1|^2 + \frac{|\nabla_H u^2|^2}{|u^1|^{2\beta}}\Big) \eta^2 \ud \mu \\
  &\leq 4(1+\max\{1, 2\beta\}) \int_{\Om}\Big(|u^1|^2+\frac{|u^2|^2}{|u^1|^{2\beta}}\Big) |\nabla_H \eta|^2 \ud \mu.
\end{align*}
If $\dist(u^1(\Om), S)>0$, then $\ud x\otimes \ud x+\frac{1}{|x|^\beta}\ud y\otimes \ud y$ is the Riemannian tensor in the Grushin plane $G_2^\beta$ and so the Riemannian length $\|u(\cdot)\|_{G_2^\beta}$ is well defined at all points of $\Om$.
\end{proof}

\subsection{Strong harmonic mappings}\label{sect16}

The goal of this section is to derive the weak formulation for the non-divergence form of the system~\eqref{str-system-EL}, the so-called strong harmonic mappings. As in the case of the weak harmonic mappings, let us assume that  $\Om\subset G_2^{\alpha}$ is an open bounded subset in the Grushin plane $G_2^{\alpha}$ and $\phi=(\phi^1, \phi^2)\in C^{\infty}_{0}(\Om,\R^2)$ is a vector-valued test function. We multiply the both sides of the equations of the system~\eqref{str-system-EL} by the test functions $\phi^1$ and $\phi^2$, respectively, and perform the integration by parts. However now the non-divergence form requires slightly more computations than the system~\eqref{w-system-EL} and, therefore, we present more details. 
%
%

The left-hand side of the first equation of~\eqref{str-system-EL} takes the following form 
\begin{align*}
&  \int_{\Omega}|u^1|^{2\beta+2} \Big(\Delta_{G_2^{\alpha}} u^1 -\alpha \frac{x}{|x|^2}u^1_x\Big)\phi^1 \ud \mu \nonumber \\
& = \int_{\Omega}|u^1|^{2\beta+2}\Big(u^1_{xx}+|x|^{2\alpha}u^1_{yy} - \frac{\alpha}{x}u^1_x\Big)|x|^{-\alpha}\phi^1 \ud x \ud y \nonumber \\
&=\int_{\Omega} |u^1|^{2\beta+2}u^1_{xx}|x|^{-\alpha} \phi^1 \ud x \ud y + \int_{\Omega} |u^1|^{2\beta+2}u^1_{yy}|x|^{\alpha}\phi^1 \ud x \ud y-\int_{\Omega} |u^1|^{2\beta+2}u^1_x \frac{\alpha}{x|x|^{\alpha}}\phi^1 \ud x \ud y \nonumber \\
 &= \int_{\Omega}|u^1|^{2\beta+2} u^1_x\left(\frac{\alpha}{x|x|^{\alpha}}\right) \phi^1 \ud x \ud y - \int_{\Omega}(2\beta+2)|u^1|^{2\beta}u^1 (u^1_x)^2|x|^{-\alpha}\phi^1 \ud x \ud y -\int_{\Omega} |u^1|^{2\beta+2}u^1_x|x|^{-\alpha} \phi^1_x \ud x \ud y \nonumber \\
&\phantom{AA}-\int_{\Omega} (2\beta+2) |u^1|^{2\beta}u^1(u^1_y)^2|x|^{\alpha}\phi^1 \ud x \ud y - \int_{\Omega}|u^1|^{2\beta+2}u^1_y|x|^{\alpha}\phi^1_y \ud x \ud y-\int_{\Omega} |u^1|^{2\beta+2}u^1_x \left( \frac{\alpha}{x|x|^{\alpha}}\right) \phi^1 \ud x \ud y \nonumber \\
    &= - \int_{\Omega}(2\beta+2)|u^1|^{2\beta}u^1|\nabla_H u^1|^2 \phi^1 \ud \mu- \int_{\Omega}|u^1|^{2\beta+2} \langle \nabla_H u,\nabla_H \phi^1 \rangle \ud \mu.
\end{align*}
Hence, the first equation in~\eqref{str-system-EL} reads:
\begin{equation}\label{str-system-EL-1}
 \int_{\Omega} |u^1|^{2\beta+2} \langle \nabla_H u^1,\nabla_H \phi^1 \rangle \ud \mu = \int_{\Omega} \beta u^1 |\nabla_H u^2|^2 \phi^1 \ud \mu  - (2\beta+2) \int_{\Omega}|u^1|^{2\beta}u^1 |\nabla_H u^1|^2 \phi^1 \ud \mu.
\end{equation}
We now turn our attention to the second equation in~\eqref{str-system-EL} and, similarly as above, upon direct computations we get the following formula:
\begin{equation*}
\begin{aligned}
&\int_{\Omega}(u^1)^2 \Big(\Delta_{G_2^{\alpha}} u^2 -\alpha \frac{x}{|x|^2}u^2_x\Big)|x|^{-\alpha}\phi^2 \ud x \ud y \\
&=- 2\int_{\Omega}u^1_x u^2_x u^1 |x|^{-\alpha}\phi^2 \ud x \ud y -\int_{\Omega}(u^1)^2 u^2_x|x|^{-\alpha}\phi^2_x \ud x \ud y-2\int_{\Omega} u^1_y u^2_y u^1|x|^{\alpha}\phi^2 \ud x \ud y - \int_{\Omega}(u^1)^2u^2_y|x|^{\alpha}\phi^2_y \ud x \ud y\\
 &= - \int_{\Omega}(u^1)^2 \langle \nabla_H u^2,\nabla_H \phi^2 \rangle \ud \mu - 2\int_{\Omega} u^1 \langle \nabla_H u^1,\nabla_H u^2 \rangle\phi^2 \ud \mu.
\end{aligned}
\end{equation*}
Therefore, the second equation of the system~\eqref{str-system-EL} reads
\begin{align}
\int_{\Omega}(u^1)^2 \langle \nabla_H u^2,\nabla_H \phi^2 \rangle \ud \mu &= -\int_{\Omega} 2\beta u^1 \langle \nabla_H u^1,\nabla_H u^2\rangle \phi^2 \ud \mu - 2\int_{\Omega} u^1 \langle \nabla_H u^1,\nabla_H u^2 \rangle\phi^2  \ud \mu \nonumber \\
&= -(2\beta+2) \int_{\Omega} u^1\langle \nabla_H u^1,\nabla_H u^2 \rangle \phi^2 \ud \mu. \label{str-system-EL-2}
\end{align}

Based on the equations~\eqref{str-system-EL-1} and~\eqref{str-system-EL-2} the weak solutions should be at least in the 
Sobolev space $HW^{1,2}(\Omega,\mu)$ but also, since the power $u^1$ appears under the integrals it is handy to assume that $u^1\in L^{\infty}(\Omega)$ as well. Note that such an assumption is in accordance with the vast literature on the harmonic mappings in the Euclidean setting, see e.g.~\cite{es,uh}, though not necessary for the studies, see Chapter 8 in \cite{jo}. Historically, the $L^{\infty}$ harmonic maps appear in the context of continuous harmonic maps, see Section 8.4 in \cite{jo}, although in general harmonic maps between Riemannian manifolds need not to be continuous, see e.g.~\cite{su1,su2}. On the other hand, the discussion in Section~\ref{sect-w-Harnack} below supports our boundedness assumption, as in Theorem~\ref{weak-harnack-u2} we show that the second component function $u^2$ of the map $u$ satisfies the weak Harnack estimate and so, in particular, is bounded. 
%
%
We summarize the above discussion in the following definition.

\begin{defn}[strong harmonic mappings]\label{str-solution} Let $\alpha, \beta>0$ and $\Omega \subset G_2^{\alpha}$ be open and bounded. We say that a mapping $u=(u^1,u^2):\Omega \to G_2^{\beta}$, such that $u \in HW^{1,2}(\Omega,\mu; G_2^{\beta})$ and $u^1 \in L^{\infty}(\Omega)$ is a \textit{strong harmonic mapping}, if the following system holds 
for all test mappings $\phi = (\phi^1,\phi^2) \in C^{\infty}_0(\Om,\R^2)$:
\begin{equation}\label{weak-system-EL}
\begin{cases}
& \int_{\Omega} |u^1|^{2\beta+2} \langle \nabla_H u^1,\nabla_H \phi^1 \rangle \ud \mu+ (2\beta+2) \int_{\Omega}|u^1|^{2\beta}u^1 |\nabla_H u^1|^2 \phi^1 \ud \mu = \int_{\Omega} \beta u^1 |\nabla_H u^2|^2 \phi^1 \ud \mu\\
&\int_{\Omega} (u^1)^2 \langle \nabla_H u^2,\nabla_H \phi^2 \rangle \ud \mu +(2\beta+2) \int_{\Omega} u^1 \langle \nabla_H u^1,\nabla_H u^2 \rangle \phi^2 \ud \mu=0.
\end{cases}
\end{equation}
\end{defn}

\begin{remark}\label{rem-sing-str}
(1)\,\,We observe that at points where $u^1(\Om)\cap S\not=\emptyset$, i.e. when $u^1\not=0$ we can divide the both sides of the equations in~\eqref{str-system-EL} by $u^1$ and upon repeating the discussion in this section obtain the following variant of the system~\eqref{weak-system-EL}, to which in what follows we will often appeal:
\begin{equation}\label{weak-system-EL2}
\begin{cases}
& \int_{\Omega} |u^1|^{2\beta}u^1 \langle \nabla_H u^1,\nabla_H \phi^1 \rangle \ud \mu+ (2\beta+1) \int_{\Omega}|u^1|^{2\beta} |\nabla_H u^1|^2 \phi^1 \ud \mu = \int_{\Omega} \beta |\nabla_H u^2|^2 \phi^1 \ud \mu\\
&\int_{\Omega} u^1 \langle \nabla_H u^2,\nabla_H \phi^2 \rangle \ud \mu +(2\beta+1) \int_{\Omega} \langle \nabla_H u^1,\nabla_H u^2 \rangle \phi^2 \ud \mu=0.
\end{cases}
\end{equation}

(2)\,\,Observe that by the derivation of system~\eqref{weak-system-EL}, strong harmonic mappings are weak harmonic, if  we additionally assume that $u^2\in L_{loc}^{\infty}$ and the assumption~\eqref{ass-lem-HW0} holds. The latter condition trivially holds for mappings $u$ omitting the singular set in their target domains.

(3)\,\,Let $\alpha \in [0,1)$ and $u=(u^1,u^2)$ be a solution to~\eqref{weak-system-EL}. It turns out that for both equations in~\eqref{weak-system-EL} the function $\phi v$ is an admissible test function for $\phi \in C_0^{\infty}(\Om)$. Indeed, as in the proof of Proposition~\ref{prop-dens} we may show that the following convolutions are uniformly bounded:
\begin{equation}\label{remark: unform bound of convolutions}
\int_{\mathbb{R}^2} \varphi_{\epsilon}(x-y)(\psi v)(y) \ud x \ud y \le M \int_{\mathbb{R}^2} \varphi_{\epsilon}(x-y) \ud x \ud y = M,
\end{equation}
 where $M:= \sup{(\psi v)}$. Let now $(\psi_k)$ be a sequence of functions in $C^{\infty}_0(\Om)$ such that $\psi_k \to \phi v$ in $HW^{1,2}(\Om,\mu)$. Then, the estimate~\eqref{remark: unform bound of convolutions} together with the Lebesgue dominated convergence theorem imply that
\begin{equation}\label{remark about test functions: first type of convergence}
 \psi_k|\nabla_H u^2|^2|x|^{-\alpha} \le M |\nabla_H u^2|^2|x|^{-\alpha} \in L^1(\Om) \hbox{ and }\psi_k|\nabla_H u^2|^2|x|^{-\alpha} \to \phi v|\nabla_H u^2|^2|x|^{-\alpha} \textrm{ a.e. in } \Om,
\end{equation}
and, therefore,
\[
 \int_{\Omega} \beta \psi_k |\nabla_H u^2|^2 \ud \mu \to \int_{\Omega} \beta \phi v |\nabla_H u^2|^2 \ud \mu.
\]
The H\"older inequality gives us the following convergence:
    \begin{align}
  &  \Big| \int_{\Omega}u^1|u^1|^{2\beta} \langle \nabla_H u^1,\nabla_H \psi_k \rangle \ud \mu- \int_{\Omega}u^1|u^1|^{2\beta} \langle \nabla_H u^1,\nabla_H (\phi v) \rangle \ud \mu \Big| \label{remark-test-111} \\
   &\phantom{AA}\lesssim_{\|u^1\|_{L^{\infty}(\Om)}^{1+2\beta}} \int_{\Om}|\nabla_H u^1||\nabla_H \psi_k - \nabla_H (\phi v)| \ud \mu \le \| u^1 \|_{HW^{1,2}(\Om,\mu)} \| \psi_k -\phi v\|_{HW^{1,2}(\Om,\mu)} \to 0,\quad k\to \infty.
    \end{align}
The remaining integrals appearing in the system~\eqref{weak-system-EL} are of the types similar either to~\eqref{remark about test functions: first type of convergence} or~\eqref{remark-test-111} and, thus, can be handled in the similar manner. 
\end{remark}

Similarly to the weak harmonic mappings also the strong one possess the Caccioppoli estimate which we now present and discuss, cf. Lemma~\ref{lem-Cac} for weak harmonic mappings.

\begin{lem}[Caccioppoli inequality 2]\label{lem-Cac-str} Let $\alpha \in [0,1)$ and $\Om \subset G_2^{\alpha}$ be an open set. Let further $u=(u^1,u^2):\Omega \to G_2^{\beta}$ be a strong harmonic mapping satisfying Definition~\ref{str-solution}, in particular $u \in HW^{1,2}(\Omega,\mu; G_2^{\beta})$ and $u^1 \in L^{\infty}(\Omega)$. 
%
Suppose that $u$ satisfies the following orthogonality condition at $\mu$-almost every point in $\Om$ for a function $K \in L^2_{\textrm{loc}}(\Om,\mu)$
\begin{equation}\label{orth-cond}
    |\langle \nabla_H u^1, \nabla_H u^2 \rangle| \le K(x,y)|u^1|^{\beta+2}.
\end{equation}
Then the following Caccioppoli inequality holds:
\begin{equation}\label{Cacp ineq ineq}
    \int_{\Om} \| D_H u \|_{G_2^{\beta}}^2 \eta^2 \ud \mu \le c(\beta) \int_{\Om} \| u \|_{G_2^{\beta}}^2 |\nabla_H \eta|^2 \ud \mu,\quad \hbox{for }\eta \in C_0^{\infty}(\Om),
\end{equation}
where $c(\beta)=4(1+\beta)^2$. Furthermore, the similar inequality holds for the $u^2$ component of map $u$: 
\begin{equation}\label{Cacp ineq for u^2}
    \int_{\Om} \frac{|\nabla_H u^2|^2}{|u^1|^{2\beta}} \eta^2 \ud \mu \le 4 \int_{\Om} \frac{|u^2|^2}{|u^1|^{2\beta}} |\nabla_H \eta|^2 \ud \mu.
\end{equation}
\end{lem}
As for the interplay between conditions~\eqref{orth-cond} and~\eqref{ass-lem-HW0} imposed in the Caccioppoli estimate for the weak harmonic mappings in Lemma~\ref{lem-Cac} notice that, by the Cauchy--Schwarz inequality, it holds that  
\[
\int_{\Om} \left(\frac{|\langle \nabla_H u^1, \nabla_H u^2 \rangle|}{|u^1|^{\beta+2}}\right)^2 \ud \mu\leq \int_{\Om} \frac{|\nabla_H u^1|^2}{|u^1|^3}\,\frac{|\nabla_H u^2|^2}{|u^1|^{2\beta+1}} \ud \mu,
\]
and so the condition~\eqref{ass-lem-HW0} by itself is not enough to ensure the $L^2_{loc}$ integrability of the left-hand side.
%
\begin{ex}
The harmonic mappings obtained as conjugates of holomorphic mappings satisfy condition~\eqref{orth-cond} with $K\equiv 0$, see Proposition~\ref{prop-holom}. Similarly, strong harmonic mappings whose target domain does not intersect the singular set, and so $|u^1|\geq c>0$, satisfy~\eqref{orth-cond}. 

Moreover, consider a map $u=(u^1, u^2)$ given by the following formula:
\begin{equation*}
    u^1(x,y) = \sin{(y-{x}_{\alpha})}, \qquad u^2(x,y) = \frac{1}{2}(y-{x}_{\alpha}) - \frac{1}{4}\sin{(2y-2{x}_{\alpha})}, \hbox{ where } x_{\alpha}:=\frac{x|x|^{\alpha}}{\alpha+1}.
\end{equation*}
One directly verifies that $u$ satisfies the system~\eqref{str-system-EL} in the domain $\{(x,y)\in G^2_\alpha:\, |{x}_{\alpha}|+ |x|^{k \alpha} < |y| < \frac{\pi}{4}\}$ for any $k\in [0,\frac23]$.
 
It turns out that the map $u$ satisfies~\eqref{orth-cond} for $\beta=1$ and $K \in L^{\infty}(\Om)$:
\begin{align*}
    |\langle \nabla_H u^1, \nabla_H u^2 \rangle| &\le 2|x|^{2\alpha} = \frac{2|x|^{2\alpha}}{|\sin{(y-{x}_{\alpha})}|^3}|u^1(x,y)|^3 \le \frac{2|x|^{2\alpha}}{(|y| - |{x}_{\alpha}|)^3}|u^1(x,y)|^3\\
    & < 2\frac{|x|^{2\alpha}}{|x|^{3k\alpha}}|u^1(x,y)|^3 \lesssim |u^1(x,y)|^3,
\end{align*}
where we used the elementary inequality $|\sin{t}| \ge \frac{2}{\pi}|t|$ for $t \in [-\frac{2}{\pi},\frac{2}{\pi}]$.
\end{ex}

\begin{proof}[Proof of Lemma~\ref{lem-Cac-str}] The proof is similar to the one for Lemma~\ref{lem-Cac} and therefore we comment only the key steps of the reasoning. 
Let $\eta \in C_0^{\infty}(\Om)$, $\delta>0$. Then the following mappings $\phi_{\delta} = (\phi^1_{\delta},\phi^2_{\delta})$ are test mappings for the system~\eqref{weak-system-EL}:
\begin{equation*}
 \phi^1_{\delta} := \frac{1}{1+\beta} \frac{|u^1|}{(|u^1|+\delta)^{2\beta+1}} \eta^2, \qquad
 \phi^2_{\delta} := \frac{u^1}{(|u^1|+\delta)^{2\beta+2}} u^2\eta^2.
\end{equation*}
This observation together with application of the Lebesgue monotone convergence theorem when $\delta\to 0^{+}$ allows us to conclude the Caccioppoli estimates~\eqref{Cacp ineq ineq} and~\eqref{Cacp ineq for u^2}.
\end{proof}
The following consequence of the Caccioppoli estimates will be useful in the proof of the weak Harnack estimates, see Theorem~\ref{weak-harnack-u2}.

\begin{cor}\label{Caccioppoli inequality for positive part remark}
    Let $\alpha \in [0,1)$ and $\Om \subset G_2^{\alpha}$ be an open set. Let further $u=(u^1,u^2):\Omega \to G_2^{\beta}$ be a strong harmonic mapping satisfying Definition~\ref{str-solution} and condition~\eqref{orth-cond} with
 $K \in L^2_{\textrm{loc}}(\Om,\mu)$. Then for each $\eta \in C_0^{\infty}(\Om)$ it holds that:
\begin{equation}\label{Cacp ineq for postive part/negative of u^2}
    \int_{\Om} \frac{|\nabla_H u^2_{\pm}|^2}{|u^1|^{2\beta}} \eta^2 \ud \mu \le 4 \int_{\Om} \frac{|u^2_{\pm}|^2}{|u^1|^{2\beta}} |\nabla_H \eta|^2 \ud \mu,
\end{equation}
where $u^2_{+}:=\max\{u^2, 0\}$ and $u^2_{-}:=\max\{-u^2, 0\}$ denote the positive and the negative part of coordinate function $u^2$, respectively.
\end{cor}
\begin{proof}
    The proof is analogous to the one of Lemma~\ref{lem-Cac-str} and relies on the following slight modification of the test mapping $\phi_{\delta}$:
 \begin{equation*}
 \phi^1_{\delta} := (1+\beta)^{-1}(|u^1|+\delta)^{-2\beta} \eta^2, \qquad
 \phi^2_{\delta} := (|u^1|+\delta)^{-2\beta-2}u^1u^2_{\pm}\eta^2.
\end{equation*}
\end{proof}

Both the Caccioppoli estimates: for the weak and strong harmonic mappings allow us to infer a variant of the Liouville theorem.
\begin{prop}[Liouville theorem]\label{thm-Liouv} Let $\alpha \in [0,1)$ and let $u=(u^1,u^2):G_2^{\alpha} \to G_2^{\beta}$ be either:
\begin{itemize}
\item
 an entire weakly harmonic map in $HW^{1,2}(G_2^{\alpha}, \ud \mu; G_2^{\beta}) \cap L^\infty_{loc}(\Om)$ such that $u$ additionally satisfies condition~\eqref{ass-lem-HW0}, or
\item  an entire strong harmonic map in $HW^{1,2}(G_2^{\alpha}, \ud \mu;G_2^{\beta})$ such that $u^1\in L^\infty_{loc}(\Om)$, satisfying~\eqref{orth-cond} with $K \in L^2_{\textrm{loc}}(\Om,\mu)$. 
\end{itemize}
If $\| u(\cdot) \|_{G_2^{\alpha}} \in L^2(G_2^{\alpha},\mu)$, then the image $u(G_2^{\alpha})$ is a point.
\end{prop}
\begin{proof}
  Let $\eta \in C_0^{\infty}(G_2^{\alpha})$ be such that $\textrm{supp}(\eta) \subset B(0,2r)$ for some $r>0$, satisfying $0 \le \eta \le 1$, $\eta \equiv 1$ in $B(0,r)$ and such that $|\nabla_H \eta| \le \frac{C}{r}$, where $B(0,r)$ and $B(0,2r)$ denote the Grushin balls centered at $0$ with radius $r$ and $2r$, respectively. That such function exists in 
 $G^2_{\alpha}$ follows, for instance, from Observation 2.1 in \cite{aww} or Lemma 5.4 in~\cite{fl83}. By direct application of Lemma~\ref{lem-Cac}, respectively Lemma~\ref{lem-Cac-str} with the test function $\eta$ we obtain the following estimate:
    \begin{equation*}
        \int_{B(0,r)}\| D_H u(x) \|_{G_2^{\beta}}^2 \ud \mu(x) \lesssim_{\beta} \frac{C^2}{r^2} \int_{B(0,2r)} \| u(x) \|_{G_2^{\beta}}^2 \ud \mu(x) \lesssim_{\beta} \frac{1}{r^2} \big\|\,\|u(\cdot)\|_{G_2^{\beta}}\,\big\|^2_{L^2(G^2_{\alpha},\mu)}\to 0, \hbox{ as } r\to \infty.
    \end{equation*}
This completes the proof of Proposition~\ref{thm-Liouv}.
\end{proof}

\section{Examples of harmonic mappings}\label{sect2}

The purpose of this section is to provide a number of examples of Grushin harmonic mappings satisfying system~\eqref{system-EL}. Since the non-divergence form~\eqref{str-system-EL} is more convenient for the direct verification if a given map is Grushin harmonic, in this section we use~\eqref{str-system-EL}.  Our presentation is divided into the following categories of examples:
\begin{itemize}
\item[(1)] Harmonic mappings arising from the Euclidean plane obtained by pull-back of the Euclidean harmonic mappings via the canonical quasisymmetric map between a $G_2^\alpha$ and $\C$.
\item[(2)] Harmonic mappings coming from the Grushin conformal mappings.
\item[(3)] Harmonic mappings obtained by the separation of variables method.
\end{itemize}

Recall the so--called \textit{canonical quasisymmetry} of a Grushin plane $G_2^{\gamma}$  and the complex plane, $\varphi_{\gamma}:G_2^{\gamma} \to \R^2$, defined with the formula
\begin{equation}\label{map-mey}
  \varphi_{\gamma}(x,y):= \left(\frac{x|x|^{\gamma}}{\gamma+1}, y\right)= (x_{\gamma}, y)
\end{equation}
where we use a parameter $\gamma>0$ in order to distinguish it from the $\alpha$ and $\beta$ cases.

Let $h=(h^1,h^2):\Om\subset \R^2 \to \mathbb{R}^2$ be a Euclidean harmonic map, i.e. $\Delta h^1=0$ and $\Delta h^2=0$ in $\Om$. Set the composition map
\[
u:=\varphi_{\beta}^{-1} \circ h \circ \varphi_{\alpha}: \phi_\alpha(\Om)\subset G_2^{\alpha} \to G_2^{\beta}.
\]
By applying the definition of map $\phi_\beta$ in~\eqref{map-mey} we find that
\begin{equation}\label{comp-map}
 u=\varphi_{\beta}^{-1} \circ h \circ \varphi_{\alpha}=\Big((\beta+1)^{\frac{1}{\beta+1}}h^1\,|h^1|^{-\frac{\beta}{\beta+1}},h^2\Big)\circ \varphi_\alpha.
\end{equation}
This way of generating the Grushin harmonic mappings is sensitive to affine translations and so, not all planar (Euclidean) harmonic mappings give rise to the Grushin harmonic ones. Indeed, suppose that $h=(h^1,h^2)$ defines the Grushin harmonic map $u=(u^1, u^2)$ as in~\eqref{comp-map}. Then $(h^1, h^2+x)$ remains planar harmonic, however the map 
\[
 v=\varphi_{\beta}^{-1} \circ (h^1, h^2+x) \circ \varphi_{\alpha}=\left(u^1,u^2+x_\alpha\right)
\]
satisfies that $|\nabla_H v^2|^2=(u^2_x + |x|^{\alpha})^2 + |x|^{2\alpha}(v_y^2)^2\not = |\nabla_H u^2|^2$ and so, fails in general the first equation in~\eqref{str-system-EL}.

Let us move toward some positive examples and assume that $\Omega \subset G_2^{\alpha}$ and $\Omega' \subset G_2^{\beta}$ are open sets. By direct computations we find the first order derivatives of $u$ in a set $\{u^1 \neq 0\}$:
\begin{equation}\label{first order derivatives of u}
\begin{aligned}
 &u^1_x = \frac{|x|^\alpha}{|u^1|^{\beta}} h^1_{w_1}(w_1,w_2)|_{w=\varphi_{\alpha}(x,y)}, \quad u^1_y = \frac{1}{|u^1|^{\beta}} h^1_{w_2}(w_1, w_2)|_{w=\varphi_{\alpha}(x,y)}\\
 &u^2_x = |x|^{\alpha}h^2_{w_1}(w_1,w_2)|_{w=\varphi_{\alpha}(x,y)}, \quad \ \ \ u^2_y = h^2_{w_2}(w_1,w_2)|_{w=\varphi_{\alpha}(x,y)}.
\end{aligned}
\end{equation}
\begin{ex}\label{example holomorphic 1}
    Let $h=(e^{x}\cos{y},e^x\sin{y})$. Then $u=(u^1,u^2):=\varphi_{\beta}^{-1} \circ h \circ \varphi_{\alpha}$ satisfies pointwisely system \eqref{str-system-EL} in the set $\{(x,y): x \neq 0, \  \cos{y} \neq 0\}$. Indeed, we apply~\eqref{first order derivatives of u} to obtain that
\begin{equation*}
\begin{aligned}
    &u^1_x = \frac{|x|^\alpha}{|u^1|^{\beta}} e^{x_{\alpha}}\cos{y}, \quad u^1_y = -\frac{1}{|u^1|^{\beta}} e^{x_{\alpha}}\sin{y},\\
    &u^2_x = |x|^{\alpha}e^{x_{\alpha}}\sin{y}, \quad \ \ \ u^2_y = e^{x_{\alpha}}\cos{y}.
\end{aligned}    
\end{equation*}
Furthermore, the analogous computations result in the second--order partial derivatives of $u$ as follows: 
\begin{equation*}
\begin{aligned}
  &u^1_{xx}=|u^1|^{-\beta}|x|^{2\alpha}e^{x_{\alpha}}\cos{y}+\alpha \frac{|x|^{\alpha}}{x}|u^1|^{-\beta}e^{x_{\alpha}}\cos{y}-\beta \frac{|x|^{2\alpha}}{|u^1|^{2\beta}u^1}(e^{x_{\alpha}}\cos{y})^2,\\
    &u^1_{yy}= -\beta \frac{1}{|u^1|^{2\beta}u^1}(e^{x_{\alpha}}\sin{y})^2 - |u^1|^{-\beta}e^{x_{\alpha}}\cos{y},\\
    &u^2_{xx}= \frac{\alpha|x|^{\alpha}}{x} e^{x_{\alpha}}\sin{y}+|x|^{2\alpha}e^{x_{\alpha}}\sin{y},\\
    & u^2_{yy}=-e^{x_{\alpha}}\sin{y}.
\end{aligned}
\end{equation*}
In a consequence, we find that 
\begin{equation*}
\begin{aligned}
    &|u^1|^{2\beta+2}\Big(\Delta_{G_2^{\alpha}} u^1 -\alpha \frac{x}{|x|^2}u^1_x\Big) +\beta u^1|\nabla_H u^2|^2=0,\\
    &\Delta_{G_2^{\alpha}} u^2  - \alpha \frac{x}{|x|^2}u^2_x = \langle \nabla_H u^1, \nabla_H u^2 \rangle=0,
\end{aligned}
\end{equation*}
and thus, map $u$ solves the system~\eqref{str-system-EL}.
%
\end{ex}

\begin{ex}\label{example holomorphic 2}
     Let $h=(x^2-y^2,2xy)$. Then $u=(u^1,u^2):=\varphi_{\beta}^{-1} \circ h \circ \varphi_{\alpha}$ satisfies system \eqref{str-system-EL} in a set $\{x \neq 0 , \ x \neq y\}$. The computations go along the same lines as those in Example~\ref{example holomorphic 1} and thus, we omit them.
\end{ex}

Note also that we may directly check that the perturbed maps $(u^1,u^2+x_\alpha)$ in Examples~\ref{example holomorphic 1} and~\ref{example holomorphic 2} are not Grushin harmonic. 

Observe that in both examples above, harmonic maps $h$ are actually complex holomorphic functions after intrinsic change of domains and targets to $\mathbb{C}$. That is no mere coincidence, in fact the following observation holds.

\begin{prop}\label{prop-holom}
    Let $\Om \subset G_2^{\alpha}$, $\Om' \subset G_2^{\beta}$ be open sets and let $\alpha, \beta \in [0,1)$. Furthermore, let $h=(h^1,h^2):\varphi_{\alpha}(\Om) \to \varphi_{\beta}(\Om')$ be a holomorphic function with $h^1 \not\equiv 0$ and set 
$$
 u:=(u^1,u^2):=\varphi_{\beta}^{-1} \circ h \circ \varphi_{\alpha}: \Om \to \Om'.
$$
Then $u\in HW^{1,2}_{loc}(\Om, \mu; G_2^\beta)$, $u^1\in L^\infty_{loc}(\Om)$ and $u$ is (locally) a strong harmonic mapping, i.e. $u$ is a weak solution to~\eqref{str-system-EL} on every open subset $\Om'\Subset \Om$. In particular, the assertion holds also in the case $\Om$ intersects the singular set $S$. 
\end{prop}

In other words, the above proposition shows that every holomorphic function in the planar domain is (strong) Grushin harmonic up to compositions with the appropriate canonical quasisymmetric maps (i.e. the Meyerson maps). Nevertheless, it turns out that not every map satisfying system~\eqref{str-system-EL} is the conjugate of a holomorphic function, as illustrated by the following example.

\begin{ex}\label{example of not holomorphic}
    Let $C \in \mathbb{R}$ and $v:G^2_{\alpha} \to \R$ be a non-constant Grushin harmonic function as considered in~\cite{aww}. i.e. $v$ satisfies the following equation in the weak sense:
\begin{equation*}
0= {\rm div}_{G_2^\alpha}\left(\frac{\nabla_H v}{|x|^{\alpha}}\right)=\left(\frac{v_x}{|x|^\alpha}\right)_{\!x}+\left(v_y\right)_y.
\end{equation*}
 Consider the map $u=(v,C):G_2^{\alpha} \to \mathbb{R}\times \{C\} \subset G_2^{\beta}$. It is easy to see that $u$ satisfies system of equations~\eqref{str-system-EL}, while at the same time a map $h = \varphi_{\beta} \circ u \circ \varphi_{\alpha}^{-1}$ fails the Cauchy--Riemann system of equations. 
\end{ex}

\begin{proof}[Proof of Proposition~\ref{prop-holom}]
As in the assertion of the proposition, let $h=(h^1,h^2):\varphi_{\alpha}(\Om) \to \varphi_{\beta}(\Om')$ be a holomorphic function and $u=(u^1,u^2):=\varphi_{\beta}^{-1} \circ h \circ \varphi_{\alpha}: \Om \to \Om'$. 

Since $u$ is continuous, it holds that $u^1\in L^{\infty}_{\textrm{loc}}(\Om)$. In order to prove that $u\in HW^{1,2}_{loc}(\Om, \mu; G_2^\beta)$, we first note that by~\eqref{first order derivatives of u}, we have that $X_1 u^2, Y_1 u^2 \in L^2_{\textrm{loc}}(\Om,\mu)$. It remains to verify that $X_1 u^1, Y_1 u^1\in L^2_{\textrm{loc}}(\Om,\mu)$, which reduces to checking that $X_1 u^1, Y_1 u^1 \in L^2(B_R,\mu)$ for any Euclidean or Grushin ball $B_R:=B((x_0,y_0),R)\Subset \Om$ centered at a point $(x_0,y_0) \in \Om$. 

Let $(x_0,y_0) \in \Om$ and consider two cases: $u^1(x_0,y_0) \neq 0$ and  $u^1(x_0,y_0)=0$.
\smallskip

\noindent {\bf Case 1: $u^1(x_0,y_0) \neq 0$.} Then the proposition follows by continuity of $u^1$ and by~\eqref{first order derivatives of u}. Indeed,  let $\varphi_{\alpha}(\Omega), \varphi_{\beta}(\Omega') \subset \mathbb{C}$ be open sets and let $h:\varphi_{\alpha}(\Omega) \to \varphi_{\beta}(\Omega')$ be a holomorphic function. By applying formulas \eqref{first order derivatives of u} in the set of points  $\{(x,y)\in \Om: x \neq 0 \hbox{ and } u^1(x,y) \neq 0\}$ we find the second order partial derivatives of $u$:
\begin{equation}\label{second order derivatives of u}
\begin{aligned}
    &u^1_{xx}=-\frac{\beta|x|^{2\alpha}}{|u^1|^{2\beta}u^1}(h^1_x({x}_{\alpha},y))^2 + \frac{\alpha|x|^{\alpha}}{x|u^1|^{\beta}}h^1_x ({x}_{\alpha},y)+\frac{|x|^{2\alpha}}{|u^1|^{\beta}}h^1_{xx}({x}_{\alpha},y),\\
    &u^1_{yy}= -\frac{\beta}{|u^1|^{2\beta}u^1}(h^1_y({x}_{\alpha},y))^2+\frac{1}{|u^1|^{\beta}}h^1_{yy}({x}_{\alpha},y),\\
    &u^2_{xx}= \frac{\alpha |x|^{\alpha}}{x}h^2_x ({x}_{\alpha},y) + |x|^{2\alpha}h^2_{xx}({x}_{\alpha},y),
    \qquad  u^2_{yy}=h^2_{yy}({x}_{\alpha},y).
\end{aligned}
\end{equation}
We employ formulas \eqref{first order derivatives of u}-\eqref{second order derivatives of u}, the Cauchy--Riemann equations, as well as, the fact that the components of holomorphic functions are harmonic, to find that 
\begin{equation*}
    \Delta_{G_2^{\alpha}} u^2  - \alpha \frac{x}{|x|^2}u^2_x = \langle \nabla_H u^1, \nabla_H u^2 \rangle=0,
\end{equation*}
and
\begin{equation*}
    |u^1|^{2\beta+2}\Big(\Delta_{G_2^{\alpha}} u^1 -\alpha \frac{x}{|x|^2}u^1_x\Big) +\beta u^1|\nabla_H u^2|^2=0.
\end{equation*}
Therefore, the mapping $u$ satisfies the system of equations~\eqref{str-system-EL} and the assertion is proven.
\smallskip

\noindent {\bf Case 2: $u^1(x_0,y_0)=0$.} By the Cheng--Yau estimate (1-2) in \cite{zh}, see also~\cite{cy}, we have that  
\begin{equation}\label{prop, holomorphic are weakly GH ineq 1}
    \frac{|\nabla h^1({x}_{\alpha},y)|}{|h^1({x}_{\alpha},y)|} \lesssim \frac{1}{d{(({x}_{\alpha},y),Z)}}, 
\end{equation}
for all $(x,y) \in B_R \setminus Z$, where $Z := \{(x,y)\,:\, h^1(x,y) =0\}$. Denote by $d({x}_{\alpha},y):=d{(({x}_{\alpha},y),Z)}$. Therefore, by \eqref{first order derivatives of u} and \eqref{prop, holomorphic are weakly GH ineq 1} we get the following estimate:
\begin{align}
    \int_{B_R} |\nabla_H u^1|^2 \ud \mu &=\int_{B_R} \frac{|x|^\alpha}{|u^1|^{2\beta}} (h^1_x({x}_{\alpha},y))^2 + \frac{|x|^{\alpha}}{|u^1|^{2\beta}} (h^1_y({x}_{\alpha},y))^2 \ud x \ud y \label{ex23-grad-est}\\
    &\approx_{\beta} \int_{B_R} \frac{|x|^\alpha}{|h^1|^{\frac{2\beta}{\beta+1}}} (h^1_x ({x}_{\alpha},y))^2 + \frac{|x|^{\alpha}}{|h^1|^{\frac{2\beta}{\beta+1}}} (h^1_y({x}_{\alpha},y))^2 \ud x \ud y \lesssim \int_{B_R} d({x}_{\alpha},y)^{\frac{-2\beta}{\beta+1}}|x|^{\alpha} \ud x \ud y. \nonumber
\end{align}
Set $\gamma:=\frac{2\beta}{\beta+1} \in [0,1)$ and denote by $d((a,b)):=\sqrt{a^2+b^2}$ the Euclidean norm of a point $(a,b)\in \R^2$. Since, it holds that $(d_x)^2+(d_y)^2=1$, we have 
\[
 |x|^{2\alpha}\leq |x|^{2\alpha} d_x^2+\max\{ \sup_{B_R}|x|^{2\alpha},1\}d_y^2\lesssim_{R^{2\alpha}} |x|^{2\alpha}d_x^2+d_y^2.
\]
Therefore,
\begin{equation*}
    \int_{B_R} d({x}_{\alpha},y)^{-\gamma}|x|^{\alpha} \ud x \ud y \lesssim_R \int_{B_R}d({x}_{\alpha},y)^{-\gamma}\sqrt{|x|^{2\alpha}d_x^2+d_y^2} \,\ud x \ud y. 
\end{equation*}
We are in a positions to apply the coarea formula and obtain that the above integral is finite. Let $M:=\sup_{B_R}d({x}_{\alpha},y)$ and note that $|\nabla d({x}_{\alpha},y)| =1$ a.e. in $B_R$, and moreover it holds that $|\nabla (d \circ \varphi_{\alpha})| =\sqrt{|x|^{2\alpha}d_x^2 + d_y^2}$ by the chain rule. Hence we have that
\begin{align*}
 &\int_{B_R}d({x}_{\alpha},y)^{-\gamma}\sqrt{|x|^{2\alpha}d_x^2+d_y^2}\, \ud x \ud y \\
 &= \int_0^M \int_{(d \circ \varphi_{\alpha})^{-1}(t)} (d\circ \varphi_{\alpha})(s)^{-\gamma} \ud \mathcal{H}^1(s) \ud t 
    =\int_0^M \int_{(d \circ \varphi_{\alpha})^{-1}(t)} t^{-\gamma} \ud \mathcal{H}^1(s) \ud t \lesssim \int_0^M t^{-\gamma} \ud t <\infty,
\end{align*}
and, therefore, the proof of the $L^2(\mu)$-integrability in~\eqref{ex23-grad-est} is completed.

In order to show that $u \in HW^{1,2}_{\textrm{loc}}(\Om,\mu; G_2^\beta)$ it suffices to show that $u \in W^{1,1}_{\textrm{loc}}(\Om, \ud x\ud y)$. Indeed, once the existence of the weak derivatives $u_x$ and $u_y$ is established, we may infer the existence of distributional derivatives $X_1u$ and $Y_1 u$ by interpreting $|x|^{\alpha} = \lim_{\delta \to 0^+}(\sqrt{x^2+\delta})^{\alpha}$ and by using the Lebesgue dominated convergence theorem. In the course of the proof of the latter observation, one also finds that $Y_1 u = |x|^{\alpha}u_y$. Since such a reasoning follows the standard lines, we omit the details.

Based on the discussion so far, it suffices to show that $u^1, u^2$ are ACL on a.e. line intersecting $\Om$. Since $u^2\in C^1_{loc}(\Om)$, the ACL property for $u^2$ follows. In order to show that property for $u^1$ we appeal to the proof of the Banach--Zarecki theorem, see e.g. pages 274--275 in~\cite{ye}. Namely, we directly verify that  $u^1$ is continuous and, since zeros of $h^1$ are isolated or $h^1$ is constantly zero, we know that $u^1$ is a.e. differentiable. Moreover, the measure of the image of the set of points where $u^1$ is non--differentiable is zero, as $u^1(\{u^1=0\})=\{0\}$. Furthermore, it holds that $u^1_x(\cdot,y_0) \in L^1$ on any set $[a,b]\times \{y_0\} \subset \Om$, by the Fubini theorem and previous considerations about integrability of the horizontal gradient. In consequence of this approach, we handle the ACL condition on horizontal lines. 
The argument for absolute continuity on the vertical lines $\{x_0\} \times [c,d]$ goes along the same lines and therefore, we omit it. Hence, the Sobolev regularity $u \in HW^{1,2}_{\textrm{loc}}(\Om,\mu)$ and $u^1\in L^{\infty}_{\textrm{loc}}(\Om)$ follows.
What remains to be shown, is that $u$ satisfies system of equations~\eqref{str-system-EL}. 

First we point out that, since the mapping $h$ satisfies the Cauchy--Riemann equations, we may deduce from~\eqref{first order derivatives of u} that $u$ satisfies the following system of equations at almost every point in $\Om$:
\begin{equation}\label{CR equations in Grushin setting}
\begin{cases}
|u^1|^{\beta} u^1_x = |x|^{\alpha}u^2_y, \\
|u^1|^{\beta}|x|^{\alpha}u^1_y = -u^2_x. 
\end{cases}
\end{equation}
This together with the Cauchy--Riemann equations for $h$ and the direct observation that $u^1|u^1|^{\beta} = (\beta+1)h^1({x}_{\alpha},y)$ implies that the first equation in~\eqref{str-system-EL} holds weakly. Indeed, let $\phi=(\phi^1,\phi^2) \in C^{\infty}_0(\Om,\mathbb{R}^2)$ be any. Then, we directly verify the following identity:
\begin{equation}\label{Prop Obs rev, top left comp}
\begin{aligned}
&  \int_{\Omega}u^1|u^1|^{2\beta}\langle \nabla_H u^1,\nabla_H \phi^1 \rangle \ud \mu \\
& = \int_{\Om} u^1|u^1|^{2\beta} u^1_x \phi^1_x |x|^{-\alpha} \ud x \ud y + \int_{\Om} u^1|u^1|^{2\beta} u^1_y \phi^1_y |x|^{\alpha} \ud x \ud y\\
  &=\int_{\Om} u^1|u^1|^{\beta} u^2_y \phi^1_x \,\ud x \ud y - \int_{\Om} u^1|u^1|^{\beta} u^2_x \phi^1_y \,\ud x \ud y\\
    &= (\beta+1)\int_{\Om} h^1({x}_{\alpha},y) h^2_y({x}_{\alpha},y) \phi^1_x \,\ud x \ud y - (\beta+1)\int_{\Om} h^1({x}_{\alpha},y) |x|^{\alpha}h^2_x({x}_{\alpha},y) \phi^1_y \,\ud x \ud y\\
    &=-(\beta+1)\int_{\Om} |x|^{\alpha}h^1_x({x}_{\alpha},y)h^2_y({x}_{\alpha},y) \phi^1 \,\ud x \ud y + (\beta+1)\int_{\Om} |x|^{\alpha} h^1_y({x}_{\alpha},y) h^2_x({x}_{\alpha},y) \phi^1 \,\ud x \ud y\\
    &-(\beta+1)\int_{\Om} |x|^{\alpha}h^1({x}_{\alpha},y)h^2_{yx}({x}_{\alpha},y) \phi^1 \,\ud x \ud y + (\beta+1)\int_{\Om} |x|^{\alpha} h^1({x}_{\alpha},y) h^2_{xy}({x}_{\alpha},y) \phi^1 \,\ud x \ud y\\
    &=-(\beta+1)\int_{\Om} |x|^{\alpha}h^1_x({x}_{\alpha},y)h^2_y({x}_{\alpha},y) \phi^1 \,\ud x \ud y + (\beta+1)\int_{\Om} |x|^{\alpha} h^1_y({x}_{\alpha},y) h^2_x({x}_{\alpha},y) \phi^1 \,\ud x \ud y\\
    &=-(\beta+1)\int_{\Om} |x|^{\alpha}(h^2_y ({x}_{\alpha},y))^2 \phi^1 \,\ud x \ud y - (\beta+1)\int_{\Om} |x|^{\alpha} (h^2_x({x}_{\alpha},y))^2 \phi^1 \,\ud x \ud y\\
    &=-(\beta+1)\int_{\Om}|x|^{\alpha}\phi^1 |\nabla h^2({x}_{\alpha},y)|^2 \,\ud x \ud y = -(\beta+1)\int_{\Om}\phi^1|\nabla_H u^2|^2 \ud \mu.
\end{aligned}   
\end{equation}
On the other hand it holds that 
\begin{equation}\label{Prop Obs rev, top right comp}
\begin{aligned}
& \int_{\Omega} \beta \phi^1 |\nabla_H u^2|^2 \ud \mu - \int_{\Omega}(2\beta+1)|u^1|^{2\beta}\phi^1|\nabla_H u^1|^2 \ud \mu \\
    &= \int_{\Omega} \beta \phi^1 |\nabla_H u^2|^2 |x|^{-\alpha} \ud x \ud y - (2\beta+1)\int_{\Om}\phi^1|\nabla_H u^2|^2 |x|^{-\alpha} \ud x \ud y\\
    &=-(\beta+1)\int_{\Om}\phi^1|\nabla_H u^2|^2 |x|^{-\alpha} \ud x \ud y
\end{aligned}
\end{equation}
and hence, by comparing \eqref{Prop Obs rev, top left comp} and \eqref{Prop Obs rev, top right comp} we arrive at the weak formulation of the first equation~\eqref{str-system-EL}.

Similar approach allows us to show that also the second equation in~\eqref{str-system-EL} holds weakly for a mapping $u$:
\begin{equation}\label{Prop Obs rev, bottom left comp}
\begin{aligned}
  \int_{\Omega}u^1 \langle \nabla_H u^2,\nabla_H \phi^2 \rangle \ud \mu 
  &  = \int_{\Omega}u^1 u^2_x \phi^2_x |x|^{-\alpha} \ud x \ud y + \int_{\Om} u^1 u^2_y \phi^2_y |x|^{\alpha} \ud x \ud y\\
    &=\int_{\Omega}u^1 h^2_x({x}_{\alpha},y) \phi^2_x \ud x \ud y + \int_{\Om} u^1 h^2_y({x}_{\alpha},y) \phi^2_y |x|^{\alpha} \ud x \ud y\\
    &=-\int_{\Omega}u^1_x h^2_x({x}_{\alpha},y) \phi^2 \ud x \ud y
    -\int_{\Omega}u^1 h^2_{xx}({x}_{\alpha},y)  \phi^2 |x|^{\alpha} \ud x \ud y\\
    &-\int_{\Om} u^1_y h^2_y({x}_{\alpha},y) \phi^2 |x|^{\alpha} \ud x \ud y - \int_{\Om} u^1 h^2_{yy}({x}_{\alpha},y) \phi^2 |x|^{\alpha} \ud x \ud y\\
    &=-\int_{\Omega}u^1_x u^2_x \phi^2 |x|^{-\alpha} \ud x \ud y 
    -\int_{\Om} u^1_y u^2_y \phi^2 |x|^{\alpha} \ud x \ud y \\
    &= -\int_{\Omega}u^1_x u^2_x \phi^2 \ud \mu
    -\int_{\Om} u^1_y u^2_y \phi^2 |x|^{2\alpha} \ud \mu \\
    &= -\int_{\Om}\phi^2 \langle \nabla_H u^1, \nabla_H u^2 \rangle \ud \mu.
\end{aligned}
\end{equation}
Recall the computations in Case 1 to see that $\langle \nabla_H u^1, \nabla_H u^2 \rangle = 0$ a.e. in $\Om$ and hence, $u$ satisfies weakly also the second equation in~\eqref{str-system-EL}. This completes the proof of the Proposition~\ref{prop-holom}.
\end{proof}
In the next corollaries we address relations between the conformal and the Grushin harmonic mappings. We distinguish two notions of conformality: first proposed in Definition 2.1 in~\cite{wa}, see also~\cite{ac}, and the \textit{metrical conformality}, see Definition 2.2 in~\cite{gjr} for $H=1$ and Theorem 4.1 therein. The approach in~\cite{ac,wa} utilizes Riemannian geometry to define conformality through limits of the appropriate orthonormal matrix on the singular line in the Grushin plane. On the other hand, the approach in~\cite{gjr} considers the so-called metric- and geometric- conformal maps defined through the modulus of curve families and the distortion of a map.

\begin{cor}\label{riem conformal maps are harmonic}
    Let $\alpha \in [0,1)$ and $\Om,\Om' \subset G_2^{\alpha}$ be open sets. Let $u \in HW^{1,\frac{2(\alpha+1)}{\alpha}}(\Om,\Om')$ be an orientation preserving conformal map in the sense of Definition 2.1 in \cite{wa}. Then $u$ is (locally) a strong harmonic mapping, i.e. $u$ is a weak solution to~\eqref{str-system-EL} on every open subset $\Om'\Subset \Om$. 
\end{cor}
\begin{cor}\label{metric conformal maps are harmonic}
    Let $\alpha \in [0,1)$ and $u:G_2^{\alpha} \to G_2^{\alpha}$ be an orientation preserving metrically conformal map
    in the sense of Definition 2.2 in~\cite{gjr} for $H=1$.  Then $u$ is (locally) a strong harmonic mapping, i.e. $u$ is a weak solution to~\eqref{str-system-EL} on every open subset $\Om'\Subset \Om$. 
\end{cor}
\begin{proof}[Proofs of Corollaries~\ref{riem conformal maps are harmonic} and~\ref{metric conformal maps are harmonic}]
By Theorem 3.5 in \cite{wa} we deduce that if $u$ is either as in the assumptions of Corollary~\ref{riem conformal maps are harmonic} or as in the assumptions of Corollary~\ref{metric conformal maps are harmonic}, then $u$ is in the form 
$$
u:= \varphi_{\alpha}^{-1} \circ h \circ \varphi_{\alpha}
$$ 
for some biholomorphism $h:\varphi_{\alpha}(\Om) \to \varphi_{\alpha}(\Om')$ satisfying $\{h^1 = 0\} \cap \varphi_{\alpha}(\Om) = \{x=0\}\cap \varphi_{\alpha}(\Om)$. Hence, the claim of Corollary \ref{riem conformal maps are harmonic} follows directly from Proposition~\ref{prop-holom}. Moreover, the assertion of Corollary~\ref{metric conformal maps are harmonic} follows upon combining~\cite[Theorem 4.1]{gjr} together with Example 4.1 and Theorem 3.5 in~\cite{wa}.
\end{proof}

Next we observe that there exist strong harmonic mappings, satisfying~\eqref{str-system-EL} outside the singular set, whose Sobolev extensions with respect to the Lebesgue (non-weighted measure) fail to be a weak solution to~\eqref{str-system-EL}. This is the consequence of the fact that for $\alpha \ge 1$, the weighted measure $\ud \mu=|x|^{-\alpha}\ud x\ud y$ is not locally finite when intersecting the singular set, see Sections 1 and 2.1 in~\cite{aww}. 
\begin{cor}\label{prop: biholomorphisms fixing the y axis are intrisicaly gh but not wgh}
   Let $\alpha \ge \beta \ge 0$ and $\Om, \Om' \subset G_2^{\alpha}$ be open sets. Let further $h=(h^1,h^2):\varphi_{\alpha}(\Om) \to \varphi_{\beta}(\Om)$ be a biholomorphism (i.e., holomorphic homeomorphism) satisfying $\{h^1 = 0\} \cap \varphi_{\alpha}(\Om) = \{x=0\}\cap \varphi_{\alpha}(\Om)$.  Then $u=(u^1,u^2):= \varphi_{\beta}^{-1} \circ h \circ \varphi_{\alpha}$ satisfies \eqref{str-system-EL} in $\Omega\setminus S$ and $u$ is locally bounded with $|\nabla_H u^i|\in L^{\infty}_{loc}$ for $i=1,2$.
\end{cor}

\begin{proof}
By Proposition~\ref{prop-holom} we know that $u$ is (locally) a strong harmonic mapping. Moreover, by continuity of $u$, we conclude that $u \in L^{\infty}_{\textrm{loc}}(\Om)$, while by formulas \eqref{first order derivatives of u} we obtain the $L^{\infty}_{\textrm{loc}}(\Om)$ regularity of the distributional derivatives of $u^2$, i.e., the $X_1 u^2$, $Y_1 u^2$ derivatives.

In order to show that $u^1_x,  |x|^{\alpha} u^1_y \in L^{\infty}_{\textrm{loc}}(\Om)$ we follow the reasoning in the proofs of Lemma 3.3 and Theorem 3.5 in~\cite{wa}. Namely, since $h:\varphi_{\alpha}(\Om) \to \varphi_{\beta}(\Om)$ is a biholomorphism with $\{h^1 = 0\} \cap \varphi_{\alpha}(\Om) = \{x=0\}\cap \varphi_{\alpha}(\Om)$ one continuously extends the following partial derivatives of $u^1$ to the whole $\Om$:
\begin{equation*}
    u^1_x = \frac{|x|^\alpha}{|u^1|^{\beta}} h^1_x ({x}_{\alpha},y), \quad u^1_y = \frac{1}{|u^1|^{\beta}} h^1_y({x}_{\alpha},y).
\end{equation*}
In consequence, $u^1_x$ and $u^1_y$ are continuous functions also on the singular set and to prove that we have used the assumption $\alpha \ge \beta \ge 0$, cf., (23) in~\cite{wa}. Hence, what remains to be shown, is the ACL property of $u$ whose proof is similar to the one of Case 2 in Proposition~\ref{prop-holom} and, therefore, we will omit the details.
%
%
\end{proof}

 In the remaining two examples we provide harmonic mappings constructed by using the separation of variables method. 

\begin{ex}\label{Example by frobenius method}
    Let $u=(u^1,u^2)$ be a map defined as follows
 \[
 u^1(x,y):=F_1({x}_{\alpha})G_1(y), \quad u^2(x,y):=F_2({x}_{\alpha})G_2(y)\quad \hbox{ where }{x}_{\alpha}:=\frac{x|x|^{\alpha}}{\alpha+1}
\] 
and functions $F_1, F_2, G_1, G_2 \in C^2(\Om)$ for an open set $\Om\subset G_2^\alpha$. Upon substituting $u$ into the second equation in~\eqref{str-system-EL} we obtain the following ODE:
  \begin{equation*}
        F_1G_1F_2''G_2 + F_1G_1F_2G_2'' = 2\beta F_1'G_1F_2'G_2+F_1G_1'F_2G_2',
 \end{equation*}
   which separates in two ODEs 
\begin{equation*}
       F_1F_2'' = 2\beta F_1'F_2', \qquad G_1G_2'' = 2\beta G_1'G_2'.
\end{equation*}
 By direct computations, we obtain
    \begin{equation}\label{Example frobenius computations 2}
        |F_2'| = |A||F_1|^{2\beta}, \quad |G_2'| = |B||G_1|^{2\beta}, \qquad A,B \in \mathbb{R}.
    \end{equation}
While the solutions with separated variables can be studied for $\beta > 0$, in order to simplify the discussion let us assume that $\beta =1$.  Then the first equation in \eqref{str-system-EL} reads:
\begin{equation*}
    (F_1G_1)^3[F_1''G_1 + F_1G_1''] + (F_2'G_2)^2 + (F_2G_2')^2 = 0
\end{equation*}
and by~\eqref{Example frobenius computations 2} this equation simplifies as follows: 
\begin{equation}\label{Example frobenius computations 4}
    (F_1)^3(G_1)^4F_1'' + (F_1)^4(G_1)^3 G_1'' + A^2 (F_1)^4 (G_2)^2 + B^2(F_2)^2 (G_1)^4 = 0.
\end{equation}
In order to perform the further analysis let us consider the following cases.

\textbf{Case 1.1.} If $B=0$, then $G_2$ is constant and we may as well assume that $G_2(y) \equiv 1$ for the sake of simplicity of presentation. Thus, the equation~\eqref{Example frobenius computations 4} reads
\begin{equation*}
    (F_1)^3(G_1)^4F_1'' + (F_1)^4(G_1)^3 G_1'' + A^2 (F_1)^4 = 0
\end{equation*}
and separates into:
\begin{equation*}
    F_1'' = 0, \qquad G_1^3 G_1''+ A^2 = 0.
\end{equation*}
We directly solve these equations and the find that
\begin{equation*}
    F_1({x}_{\alpha}) = D {x}_{\alpha} + E, \quad G_1(y) = \pm \frac{1}{\sqrt{H}} \sqrt{H^2(y+I)^2 - A^2}, \qquad H>0, \ A, D, E, I \in \mathbb{R}.
\end{equation*}
Finally, we substitute such $F_1$ into~\eqref{Example frobenius computations 2} and obtain that $F_2'(t) = \pm A(D t + E)^2$, which gives us the following family of solutions
\begin{equation*}
  F_2(t) = \frac{1}{3}AD^2 t^3 + ADEt^2 + AE^2t+I.
\end{equation*}
In a consequence the harmonic map $u$ takes the following form: 
\begin{equation*}
    u^1(x,y) = \pm \frac{1}{\sqrt{H}}(D {x}_{\alpha} + E)\sqrt{H^2(y+I)^2 - A^2}, \qquad u^2(x,y) = \frac{1}{3}AD^2 {x}_{\alpha}^3 + ADE {x}_{\alpha}^2 + AE^2{x}_{\alpha}+I,
\end{equation*}
with $A,D,E,I \in \mathbb{R}$, $H>0$ and $|y+I| > |A/H|$.

\textbf{Case 1.2.} Let us assume that $G_1(y) = G_2(y) \equiv 1$, i.e., the solution depends only on the $x$ variable. Then \eqref{Example frobenius computations 4} simplifies further and reads:
\begin{equation*}
    F_1'' + A^2 F_1 = 0\, \hbox{ and so }\, F_1(t) = D\cos{At} + E \sin{At}.
\end{equation*}
Hence, solving \eqref{Example frobenius computations 2} yields
\begin{equation*}
    F_2(t) = H \pm \frac{1}{4}\Big(2A(D^2+E^2)t - 2DE\cos{(2At)} + (D^2 - E^2)\sin{(2At)}\Big).
\end{equation*}
As a result we arrive at the following family of solutions depending on $x$ variable only: 
\begin{equation*}
\begin{aligned}
    &u^1(x,y) = D\cos{(A{x}_{\alpha})} + E \sin{(A {x}_{\alpha})},\\
    &u^2(x,y) = H \pm \frac{1}{4}\left(2A(D^2+E^2){x}_{\alpha} - 2DE\cos{(2A{x}_{\alpha})} + (D^2 - E^2)\sin{(2A{x}_{\alpha})}\right)
\end{aligned}
\end{equation*}
with $A,D,E,H \in \mathbb{R}$.

\textbf{Case 2.1.} Let now $A=0$, and so without loss of generality we may assume that $F_2(y) \equiv 1$. Then, by~\eqref{Example frobenius computations 2}, after an analogous discussion as in Case 1.1, equation \eqref{Example frobenius computations 4} implies that:
\begin{equation*}
    G_1'' = 0, \qquad F_1^3 F_1''+ B^2 = 0.
\end{equation*}
By direct computations analogous to the ones in Case 1.1, we get the following family of solutions:
\begin{equation*}
    u^1(x,y) = \pm \frac{1}{\sqrt{H}}(D y + E)\sqrt{H^2({x}_{\alpha}+I)^2 - B^2}, \qquad u^2(x,y) = \frac{1}{3}BD^2 y^3 + BDE y^2 + BE^2y+I,
\end{equation*}
where $B,D,E,I \in \mathbb{R}$, $H>0$ and $|\tilde{x}_{\alpha}+I| > |B/H|$.

\textbf{Case 2.2.} Finally, let us assume that $F_1(x) = F_2(x) \equiv 1$, i.e., the solution depends only on the $y$ variable. The procedure of obtaining solution is analogous to the one in Case 1.2 and yields the following family of solutions:
\begin{equation*}
\begin{aligned}
    &u^1(x,y) = D\cos{(By)} + E \sin{(By)},\\
    &u^2(x,y) = H \pm \frac{1}{4}\left(2B(D^2+E^2)y - 2DE\cos{(2By)} + (D^2 - E^2)\sin{(2By)}\right),
\end{aligned}
\end{equation*}
where $B,D,E,H \in \mathbb{R}$.
\end{ex}

\begin{ex}\label{better behaved, non homolomorphic solution}
Let $\beta = 1$ and consider a map $u = (u^1,u^2)$ in the following form:
\begin{equation*}
    u^1(x,y) = F(y+ (-1)^j{x}_{\alpha}), \qquad u^2(x,y) = G(y + (-1)^j{x}_{\alpha}),\qquad j=0,1
\end{equation*}
for a $C^2$-regular functions $F,G:\mathbb{R} \to \mathbb{R}$ and points $(x,y)\in \Om$ for some open set $\Om\subset G_2^\alpha$.
We substitute the above $u^1$ and $u^2$ in the second equation of system \eqref{str-system-EL} to obtain the following relation:
\begin{equation*}
    FG'' = 2 F'G'.
\end{equation*}
As in Example~\ref{Example by frobenius method}, the above equation is a separable ODE and implies that $G' = \pm A F^{2}$ for some $A \in \mathbb{R}$. By direct computations, plugging in the latter result into the first equation in \eqref{str-system-EL}, we get $F'' + A^2 F =0$ and hence:
\begin{equation*}
    F(t) = B\cos{At} + C \sin{At}, \qquad A,B,C \in \mathbb{R}.
\end{equation*}
Next, by using the previously obtained relation $G' = \pm A F^{2}$ we get:
\begin{equation*}
    G(t) = D \pm \frac{1}{4}\left(2A(B^2+C^2)t - 2BC\cos{(2At)} + (B^2 - C^2)\sin{(2At)}\right).
\end{equation*}
Hence, we obtain the following family of solution $u$ of system \eqref{str-system-EL}:
\begin{align*}
    u^1(x,y)& = B\cos{(Ay + (-1)^j A{x}_{\alpha})} + C \sin{(Ay + (-1)^j A{x}_{\alpha})} \\
    u^2(x,y)&= D \pm \frac{1}{4}\Big(2A(B^2+C^2)(y + (-1)^j{x}_{\alpha}) - 2BC\cos{(2Ay + (-1)^j 2A{x}_{\alpha})} \\
    &+ (B^2 - C^2)\sin{(2Ay + (-1)^j 2A{x}_{\alpha})}\Big)
\end{align*}
where $A,B,C,D \in \mathbb{R}$ and $j=0,1$.
\end{ex}

\section{The $H^{2,2}_{loc}$ second order Sobolev regularity of harmonic maps}

The goal of this section is to prove the second order $H^{2,2}_{loc}$ Sobolev regularity of harmonic mappings between the Grushin planes, see Theorem~\ref{thm-H22}. The result is then applied in the next section to study the Bochner identity for harmonic mappings and its consequences.
 
We first recall the definition of the $H^{2,2}$-Sobolev spaces of functions in the Grushin setting $G_2^\alpha$. One should pay attention to the fact that the weighted measure setting leads to the problems of defining the difference quotients and requires additional effort, see the discussion of the properties (DQ1)-(DQ3). Nevertheless, we can appeal to Lemma~\ref{lem-Sob-char}, already proven in~\cite[Lemma 3.1]{aww} which is the Grushin counterpart of the well known description of the second order weak derivatives.  Then, the proof of Theorem~\ref{thm-H22} employs standard techniques based on the choice of appropriate test functions and the analysis of their difference quotients. However, since now the target space is the Grushin plane, the proof is far more complicated than its counterpart for the harmonic mappings into $\R^2$, cf. Theorem 1.2 in~\cite{aww}. Indeed, the following are the key differences between the proof in~\cite{aww} and here:
\begin{itemize}
\item since the system of equations~\eqref{w-system-EL} is coupled and its first equation differs from the second one, our analysis of the difference quotients has to take into account the interplay between estimates coming from both equations;
\item the natural norm~\eqref{def-DH} of the differential $D_Hu$ of a harmonic map $u=(u^1,u^2)$ contains the weight $|u^1|^{-2\beta}$ for the gradient of the second component function $u^2$; this in turn makes our analysis more delicate and demanding and motivates our assumption that the image of the harmonic map does not intersect the singular set in the target space $G_2^\beta$, i.e.,  $u^1(x,y)\not=0$ for all points $(x,y)$ in the closure of the domain of $u$;
\end{itemize}
For these reasons the second order regularity of harmonic maps becomes a delicate issue on domains intersecting the singular set (the $y$-axis) in the source space $G_2^\alpha$. Therefore,  additional conditions on the gradient integrability need to be imposed to ensure the $H^{2,2}$-regularity, see conditions~\eqref{ass0-thm-H22}-\eqref{ass3-thm-H22}. 
\smallskip


Let $\Om\subset G_2^\alpha$ be an open set and $v:\Om \to \R$ be a real-valued function with weak derivatives $Xv, Yv$  existing in $\Om$ as distributions and suppose that $XYv$ and $YXv$ exist as distributions. We define the following Sobolev space
\begin{align*}\label{def-H22}
& H^{2,2}(\Om,\mu):=\{v\in L^2(\Om, \mu)\,:\ XXv, XYv, YXv, YYv \in L^2(\Om,\mu)\}, \\
& XXv=(v_x)_x, \quad XYv=(\varrho v_y)_x=\varrho_xv_y+\varrho (v_y)_x,  \quad YXv=\varrho (v_x)_y, \quad YYv=\varrho^2(v_y)_y.\nonumber
\end{align*} 
The space $H^{2,2}(\Om,\mu)$ is a Hilbert space. By analogy we define the space $H^{2,2}_{loc}(\Om,\mu)$ as the space of functions in $H^{2,2}_{loc}(K,\mu)$ for any compact set $K\Subset \Om$. 
The space of mappings $u:\Om\to G_2^{\beta}$ in $H^{2,2}(\Om,G_2^{\beta}, \mu)$ consists of mappings whose component functions $u^1, u^2$ belong to $H^{2,2}(\Om,G_2^{\beta}, \mu)$.

Before we present conditions allowing to infer that a Grushin-harmonic function on $\Om\subset G_2$ belongs to the $H^{2,2}$-space, we recall Lemma 3.1 in~\cite{aww} and address the new phenomena occuring for the target space $G_2^\beta$ . 
Let us remark, that the second order regularity in the Grushin setting has been studied e.g. in~\cite{dm} and also Section 6 in~\cite{ddfm} for the setting of the Grushin plane with the weight $\varrho(x)=x$.

Let us briefly recall the key elements of the difference quotients method, see~\cite[Section 3]{aww} for more details. Let $\Om\subset G_n$ and $v:\Om\to \R$ be a function. Let $h>0$ and $e^i$ for $i=1,\ldots, n$ denote vectors in the standard vector basis in $\R^n$.  We define 
\[
\Delta_h^{i} v(p):=\Delta_h^{e_i} v(p):=\frac{v(p+he^i)-v(p)}{h},\quad p\in \Om.
\]
The following properties hold, cf. Chapter 8.1 in~\cite{giu} for the Euclidean setting:
\begin{itemize}
\item[(DQ1)] $\Delta_h^{i} (uv)(p)=u(p+he^i) \Delta_h^{i} v(x)+v(x)\Delta_h^{i}u(x)$.
\item[(DQ2)] Set $\Om_{|h|}:=\{p\in\Om\,:\,\dist(p,\partial \Om)>|h|\}$ for an open bounded set $\Om\subset G_2$. If $v\in H^{1,2}\cap L^p(\Om,\mu)$ for some $p> \frac{2}{1-\alpha}\geq 2$, then $\Delta_h^{i} v\in H^{1,2}(\Om_{|h|},\mu)$ for any $i=1,2$ and 
\[
(\Delta_h^{i} v)_{x_i}=\Delta_h^{i}(v_{x_i}).
\]
\item[(DQ3)] For any test function $\phi\in C_0^{\infty}(\Om_{|h|})$ we have that
\[
 \int_{\Om} v \Delta_h^{i} \phi\, \ud x\ud y= -\int_{\Om} \phi \Delta_{-h}^{i} v\, \ud x\ud y.
\]
\end{itemize}

Recall Lemma 3.1 in~\cite{aww}. For the sake of completion of the presentation, let us note that as in the Euclidean setting the converse of the assertions below also holds, however we are not stating them since they will not be used in what follows (see e.g. Lemma 7.23 in~\cite[Chapter 7.11]{gt}).

\begin{lem}\label{lem-Sob-char}
  Let $\Om\subset G_2^\alpha$ be a domain and $\Om'\Subset \Om$. Furthermore, let $v\in L^2_{loc}(\Om,\mu)$ with weak derivatives $v_x, |x|^\alpha v_y\in L^2_{loc}(\Om, \mu)$.
\smallskip

\noindent If there exist constants $\sigma<\dist(\Om', \partial \Om)$ and $C>0$ such that:
  \begin{align*}
  &\sup_{0<|h|<\sigma} \int_{\Om'}|\Delta_h^{1} v_x|^2 \frac{\ud x \ud y}{|x|^{\alpha}}\leq C,
   \hbox{  then the weak derivative }X^2v\hbox{ exists and }\|X^2 v\|_{L^2(\Om',\mu)}\leq C, \\
  &\sup_{0<|h|<\sigma} \int_{\Om'}|\Delta_h^{1} (\varrho v_y)|^2 \frac{\ud x \ud y}{|x|^{\alpha}}\leq C,
   \hbox{  then the weak derivative }XYv\hbox{ exists and }\|XY v\|_{L^2(\Om',\mu)}\leq C, \\
  &\sup_{0<|h|<\sigma} \int_{\Om'}|\Delta_h^{2} v_x|^2\,|x|^{\alpha}\, \ud x \ud y\leq C,
   \hbox{  then the weak derivative }YXv\hbox{ exists and }\|YX v\|_{L^2(\Om',\mu)}\leq C, \\
  &\sup_{0<|h|<\sigma} \int_{\Om'}|\Delta_h^{2} v_y|^2\, |x|^{\alpha}\, \ud x \ud y\leq C,
   \hbox{ then the weak derivative }Y^2v\hbox{  exists and }\|Y^2 v\|_{L^2(\Om',\mu)}\leq C.
  \end{align*}
%
\end{lem}
Note that  the change of the weight in the last two assertions is due to the fact that $\varrho=|x|^{\alpha}$ depends only on the $x$ variable, hence $|\Delta_h^{e_2} (\varrho u_x)|=\varrho |\Delta_h^{e_2} u_x|$ and $\Delta_h^{e_2} (\varrho u_y)|=\varrho |\Delta_h^{e_2} u_y|$.

\begin{rem}\label{rem-Sob-char2}
The aforementioned definition~\eqref{def-DH} of the differential $D_Hu$ of a harmonic map $u=(u^1,u^2)$ results in stronger estimates for the difference quotients of the component function $u^2$ than the one in Lemma~\ref{lem-Sob-char} above, see for instance the left-hand side of~\eqref{main-est-Dh}. However, that condition implies the corresponding one in the lemma above due to our assumption that $u^1\in L^\infty_{loc}$, see~\eqref{def-sol-space}. Indeed, suppose that for a domain $\Om'\Subset \Om$ as in Lemma~\ref{lem-Sob-char}, it holds
\[
\sup_{0<|h|<\sigma} \int_{\Om'}\frac{|\Delta_h^{1} u^2_x|^2}{|u^1(x+h,y)|^{2\beta}} \frac{\ud x \ud y}{|x|^{\alpha}}\leq C.
\]
Then
\begin{align*}
\sup_{0<|h|<\sigma} \int_{\Om'} |\Delta_h^{1} u^2_x|^2 \frac{\ud x \ud y}{|x|^{\alpha}} &\leq \sup_{0<|h|<\sigma}\int_{\Om'}\frac{|\Delta_h^{1} u^2_x|^2}{|u^1(x+h,y)|^{2\beta}} |u^1(x+h,y)|^{2\beta} \frac{\ud x \ud y}{|x|^{\alpha}} \\
 &\leq \|u^1\|_{L^\infty(\Om')}\sup_{0<|h|<\sigma} \int_{\Om'}\frac{|\Delta_h^{1} u^2_x|^2}{|u^1(x+h,y)|^{2\beta}} \frac{\ud x \ud y}{|x|^{\alpha}}\leq C\|u^1\|_{L^\infty(\Om')},
 \end{align*}
 and so the corresponding assumption for $v=u^2$ in Lemma~\ref{lem-Sob-char} holds. 
%
 \end{rem}

Recall that $S$ denotes the singular set of a Grushin plane $G_2^\alpha$, i.e. the $y$-axis. For the readers convenience we also recall the statement of Theorem~\ref{thm-H22}.
\smallskip
\\
\noindent
{\em {\bf Theorem~\ref{thm-H22}.}\,
Let $\Om\subset G^{\alpha}_2$ be an open bounded set and $u=(u^1, u^2):\Om \to G^2_{\beta}$ be a weakly Grushin-harmonic map in $H^{1,2}_{loc}(\Om,G^{\beta}_2,\mu)$, i.e., $u$ satisfies the system of equations~\eqref{w-system-EL}, and  such that $\overline{u^1(\Om)}\cap S=\emptyset$.  Moreover, assume that $\alpha, \beta \in [0,1)$ and there exist
 \begin{equation}
 p> \frac{2}{1-\alpha}\,\,\hbox{ and }\,\,\ \gamma \geq \max\left\{ \frac{2+p+p\alpha}{2\alpha},\, p(2-\alpha)\right\}\,(>2)\qquad \tag{H2-A} \label{ass0-thm-H22}
\end{equation}
 such that for any ball $B(R)\cap S\not=\emptyset$ it holds that:
\leqnomode
\begin{align}
&\int_{B(R)} \frac{|\nabla_H u^2|^p}{|u^1|^{(\beta+1)p}} \ud \mu<\infty, \qquad
\int_{B(R)} \frac{|u^1_x|^{\frac{2p}{p-2}}}{|x|^{\frac{2\alpha}{p-2}}} \ud \mu<\infty \tag{H2-B} \label{ass1-thm-H22} \\
&\int_{B(R)} \frac{|D_H u|^p}{|x|^{\alpha(\gamma-1)}} \ud \mu<\infty \tag{H2-C}\label{ass2-thm-H22} \\
& \int_{B(R)} \frac{|u^1_x|^{\frac{p}{p-2}}}{|x|^{\frac{p}{p-2}}} \ud \mu<\infty. \tag{H2-D}\label{ass3-thm-H22}
\end{align}
If assumptions \eqref{ass0-thm-H22}-\eqref{ass3-thm-H22} hold, then the second order derivatives $X^2 u^1, XY u^1$ and $X^2 u^2, XY u^2$ are in $L^{2}_{loc}(\Om, \mu)$. 

\noindent If assumption~\eqref{ass1-thm-H22} holds for some $p>2$, then the second order derivatives $YXu^1, Y^2u^1$ and $YXu^2, Y^2u^2$ are in $L^{2}_{loc}(\Om, \mu)$. 

\noindent
Furthermore, if an open set $\Om'\subset \Om$ satisfies $\overline{\Om'}\cap S=\emptyset$, then it holds that $u^1, u^2 \in H^{2,2}(\Om', \R, \mu)$ only under the assumption~\eqref{ass1-thm-H22} holding for some $p>2$.}

\begin{remark}\label{rem44} Let us discuss relations between the assumptions of Theorem~\ref{thm-H22} in the general case and in some special important cases.

(1) Direct analysis based on the straightforward computations reveals that there are no implications between assumptions~\eqref{ass1-thm-H22}-\eqref{ass3-thm-H22} for general $p>2$ and $\alpha, \beta \in (0,1)$. However, assumption~\eqref{ass0-thm-H22} and the first condition in~\eqref{ass1-thm-H22} imply condition (A1) in~\eqref{ass-lem-HW0} in Definition~\ref{weak-solution} of a weak harmonic mapping.
\smallskip

(2) If $\alpha=\beta=0$, then $\varrho=1$, $\ud \mu=\ud x\ud y$ and our harmonic mappings become the (uncoupled) harmonic mappings in the plane. Then, one of the key estimate~\eqref{main-est-H22} of the proof below, see Step III, takes the following form: 
\begin{align*}
\int_{B(R/2)} \bigg((\Delta_{h}^1u^1_x)^2 +  (\Delta_{h}^1 u^1_y)^2 + (\Delta_{h}^1u^2_x)^2 + (\Delta_{h}^1 u^2_y)^2 \bigg) \ud x\,\ud y &\lesssim_c  \frac{1}{R^2} \int_{B(R)}  \bigg ( (\Delta_h^1 u^1)^2 + (\Delta_h^1 u^2)^2\bigg) \ud x\,\ud y\\
&\lesssim \frac{1}{R^2} \|u\|^2_{H^{1,2}(B(R))}. 
\end{align*}
Therefore, we retrieve the familiar estimates for harmonic functions, following from combining a characterization of the Sobolev space via the difference quotients, see e.g. Lemma 7.23 in~\cite{gt} applied to $\Delta_h^j u_x^i$ and $\Delta_h^j u_y^i$, and Theorem 4.6 in~\cite{lin} for $p=2$. Thus no assumptions \eqref{ass0-thm-H22}-\eqref{ass3-thm-H22} are needed in that case.
\smallskip

(3) If $\beta=0$, then the estimate~\eqref{main-est-H22} reads:
\begin{align*}
&\int_{B(R/2)} \bigg((\Delta_{h}^1u^1_x)^2 +  (\Delta_{h}^1 \varrho u^1_y)^2 + (\Delta_{h}^1u^2_x)^2 + (\Delta_{h}^1 \varrho u^2_y)^2 \bigg) \frac{\ud x\,\ud y}{\varrho} \\
&\lesssim \frac{1}{R^2} \|u\|^2_{H^{1,2}(B(R))}+  \int_{B(R)} |D_H u(x+h,y)|^2 \frac{(\Delta_h^1 \varrho)^2}{\varrho(x+h)^2}  \frac{\ud x\,\ud y}{\varrho}, 
\end{align*}
and the appearance of the second term on the right-hand side leads only to assumptions~\eqref{ass0-thm-H22} and~\eqref{ass2-thm-H22}, as in Theorem 1.2 in~\cite{aww}, i.e., for $\beta=0$.
\end{remark}
 
\begin{ex}\label{ex-thm-H22}
A large class of examples illustrating Theorem~\ref{thm-H22} comes from the Grushin harmonic mappings originating from the holomorphic mappings in the plane, as in Proposition~\ref{prop-holom} and Examples~\ref{example holomorphic 1} and~\ref{example holomorphic 2}. Indeed, if $\overline{\Om} \cap S=\emptyset$, then Theorem~\ref{thm-H22} gives us that $u^1, u^2 \in H^{2,2}(\Om, \mu)$ provided that the assumption~\eqref{ass1-thm-H22} holds. Moreover, if $\overline{\Om}$ intersects the singular set, then the same assumption allows us to conclude that $YX u^i, Y^2u^i$ exist and are $L^2_{loc}$-integrable for $i=1,2$.

Recall Remark~\ref{rem-sing-str}(2) and equations~\eqref{first order derivatives of u}.  Then, by the direct computations,  conditions in the assumption~\eqref{ass1-thm-H22} read:
\begin{align*}
&\int_{B(R)} \frac{|\nabla_H u^2|^p}{|u^1|^{(\beta+1)p}} \ud \mu= \int_{B(R)} \frac{|x|^{\alpha(p-1)}}{|u^1|^{(\beta+1)p}} |\nabla h(\phi_{\alpha(x,y)})|^p\,\ud x \ud y <\infty \\
&\int_{B(R)} \frac{|u^1_x|^{\frac{2p}{p-2}}}{|x|^{\frac{2\alpha}{p-2}}} \ud \mu=\int_{B(R)} \frac{|x|^{\frac{\alpha p}{p-2}}}{|u^1|^{\frac{2\beta}{p-2}}} |h^1_{w_1}(w)|_{w=\phi_{\alpha}(x,y)}|^{\frac{2p}{p-2}}
 \,\ud x \ud y <\infty.
\end{align*}
Both integrals are finite, since $p>2$, $u^1$ omits the singular set by the assumptions of Therorem~\ref{thm-H22} and, moreover, it holds that $\alpha(p-1)+1>0$ and $\frac{\alpha p}{p-2}+1>0$ and so the ball-box theorem~\cite[Theorem 3.1]{wu} can be applied. 
\end{ex}

\begin{proof}[Proof of Theorem~\ref{thm-H22}]
 Let $u=(u^1, u^2)$ be a Grushin-harmonic map as in the assumptions of the theorem. First we study the existence of derivatives $X^2$ and $XY$. This reduces to integral estimates for $\Delta_h^1 u^i_x$ and $\Delta_h^1 (\varrho u^i_y)$ for $i=1,2$. However, since equations in the system~\eqref{w-system-EL} are not symmetric as it is in the Euclidean setting, Steps I and II below also differ. Then, in Step III we merge both estimates and only starting form there the assumptions \eqref{ass1-thm-H22}-\eqref{ass3-thm-H22} come into play, see the estimate~\eqref{main-est-H22} and the discussion following it. Moreover, our analysis distinguishes the cases whether the domain intersects the singular set or not, see the presentation for Cases 1 and 2 in Step 3. Finally, we discuss the derivatives $YX$ and $Y^2$ which are simpler than the $X^2$ and $XY$, since now the appearance of the weight $\varrho=\varrho(x)$ in difference quotients $\Delta^2_h$ does not affect the computations.
 \smallskip
 \\
 {\bf The derivatives $X^2$ and $XY$.}
\smallskip
\\
 {\bf Step I: the integral estimate for the difference quotients $\Delta_h^1 u^1_x$ and $\Delta_h^1 (\varrho u^1_y)$}.
 \smallskip
 \\
Let $\phi \in C_0^{\infty}(B(R))$ be a test function on a ball $B(R)\subset 2B(R)\Subset \Om$ and such that $\phi$ satisfies: $0\leq \phi \leq 1$ on $B(R)$, $\phi\equiv 1$ on $B(R/2)$ and $|\nabla_H \phi|\leq \frac{c}{R}$.  Since $\phi$ is compactly supported and the support can be taken to be the closure of a ball, we have that $\phi^2 \Delta_h^1 u^1\in HW^{1,2}_0(B(R),\mu)$. This discussion together with the property (DQ2) of $\Delta_h^i u^1$ yields that 
\smallskip

\centerline{
$\Delta_{-h}^1(\phi^2 \Delta_h^1 u^1)$ is a test function for the system of equations~\eqref{w-system-EL}.
}
\smallskip

\noindent In what follows we set $0<|h|<r<R$.

We test the first equation of the system~\eqref{w-system-EL} with the above function and employ the properties (DQ1)-(DQ3) of finite difference quotients, to arrive at the following identity:
\begin{align*}  
&\beta \int_{\Om} \frac{|\nabla_H u^2|^2 u^1}{\varrho |u^1|^{2\beta+2}}\,\Delta_{-h}^1(\phi^2 \Delta_h^1 u^1)\,\ud x\,\ud y \\
&\phantom{AA}=-\int_{\Om} \langle \nabla_H u^1, \nabla_H [\Delta_{-h}^1(\phi^2 \Delta_h^1 u^1)] \rangle \frac{\ud x\,\ud y}{\varrho} \\
&\phantom{AA}=-\int_{\Om} \Big( \frac{u^1_x}{\varrho}\, \Delta_{-h}^1(\phi^2 \Delta_h^1 u^1)_x+ \varrho u^1_y\,\Delta_{-h}^1(\phi^2 \Delta_h^1 u^1)_y \Big) \ud x\,\ud y \\
&\phantom{AA}= \int_{\Om} \Big( \Delta_{h}^1\left(\frac{u^1_x}{\varrho}\right) (\phi^2 \Delta_h^1 u^1)_x+ \Delta_{h}^1(\varrho u^1_y) (\phi^2 \Delta_h^1 u^1)_y \Big) \ud x\,\ud y \\
&\phantom{AA}=\int_{\Om} \Big( u^1_x(x+h,y)\Delta_{h}^1(\varrho^{-1})+\varrho^{-1} \Delta_{h}^1u^1_x \Big)\,\Big( \phi^2 \Delta_h^1 u^1_x+ 2\phi \phi_x (\Delta_h^1 u^1) \Big) \\
&\phantom{AAAA} +\Delta_{h}^1(\varrho u^1_y) \Big[ \phi^2 \big(\varrho^{-1} \Delta_{h}^1(\varrho u^1_y)- \varrho^{-1}u^1_y(x+h,y)\Delta_{h}^1\varrho\big) + 2\phi \phi_y (\Delta_h^1 u^1) \Big] \ud x\,\ud y,
\end{align*} 
where in the last equality we additionally use that $ \Delta_{h}^1u^1_y=\varrho^{-1} \Delta_{h}^1(\varrho u^1_y)- \varrho^{-1}u^1_y(x+h,y)\Delta_{h}^1\varrho$. Next, we rearrange terms in the above integral and move terms with $(\Delta_{h}^1u^1_x)^2$ and $(\Delta_{h}^1\varrho u^2_y)^2$ to the left-hand side. In a consequence, we obtain the following formula:
\begin{align*}  
&\int_{\Om} \Big((\Delta_{h}^1u^1_x)^2 +  (\Delta_{h}^1 \varrho u^1_y)^2 \Big) \phi^2 \frac{\ud x\,\ud y}{\varrho} \tag{$u^1$-est}\\
&\phantom{A}=-\beta \int_{\Om} \frac{|\nabla_H u^2|^2 u^1}{\varrho |u^1|^{2\beta+2}}\,\Delta_{-h}^1(\phi^2 \Delta_h^1 u^1)\, \ud x\,\ud y \tag{N}\\
&\phantom{A}-\int_{\Om}\! \left(u^1_x(x+h,y)(\Delta_{h}^1\varrho^{-1}) \varrho^{\frac12}\phi \right)\!\cdot\!\left((\Delta_h^1 u^1_x)\varrho^{-\frac12} \phi\right)-\left(u^1_y(x+h,y)(\Delta_{h}^1\varrho) \varrho^{-\frac12}\phi \right)\!\cdot\!\left((\Delta_h^1 \varrho u^1_y)\varrho^{-\frac12} \phi\right) \tag{I}\\
&\phantom{A}-\int_{\Om} \left(2 (\Delta_h^1 u^1) \varrho^{-\frac12}\phi_x \right)\cdot\left((\Delta_{h}^1u^1_x) \varrho^{-\frac12} \phi\right) +\left(2 (\Delta_h^1 u^1) \varrho^{\frac12}\phi_y \right)\cdot\left((\Delta_{h}^1 \varrho u^1_y) \varrho^{-\frac12} \phi\right) \tag{II}\\
&\phantom{A}-\int_{\Om} 2 \left(u^1_x(x+h,y)(\Delta_{h}^1\varrho^{-1})\varrho^{\frac12}\phi\right)\cdot\left((\Delta_h^1 u^1)\varrho^{-\frac12} \phi_x \right). \tag{III}
\end{align*} 
In order to estimate integrals (I)-(III) we use the algebraic inequality $2ab\leq \delta^{-2} a^2+\delta^2b^2$ for various $\delta>0$ and get that:
\begin{align}
{\rm (I)+(II)+(III)}&\leq \frac14 \int_{\Om} (\Delta_h^1 u^1_x)^2 \frac{\phi^2}{\varrho}+\int_{\Om} [u^1_x(x+h,y)]^2 (\Delta_{h}^1\varrho^{-1})^2 \varrho \phi^2 \label{est-I-III}\\
&\phantom{A}+\frac14 \int_{\Om} (\Delta_h^1 \varrho u^1_y)^2 \frac{\phi^2}{\varrho}+\int_{\Om} [u^1_y(x+h,y)]^2 (\Delta_{h}^1\varrho)^2 \frac{\phi^2}{\varrho}\tag{I} \nonumber \\
&+\frac14 \int_{\Om} (\Delta_h^1 u^1_x)^2 \frac{\phi^2}{\varrho}+4\int_{\Om} (\Delta_h^1 u^1)^2 \frac{\phi_x^2}{\varrho}+\frac14 \int_{\Om} (\Delta_h^1 \varrho u^1_y)^2  \frac{\phi^2}{\varrho} +4\int_{\Om} (\Delta_{h}^1 u^1)^2 \varrho \phi_y^2 \tag{II} \nonumber \\
&+\int_{\Om} [u^1_x(x+h,y)]^2 (\Delta_{h}^1\varrho^{-1})^2 \varrho \phi^2+\int_{\Om} (\Delta_h^1 u^1)^2 \frac{\phi_x^2}{\varrho}.\tag{III} \nonumber
\end{align} 
It turns out that the most challenging term is (N), arising from the coupling of~\eqref{w-system-EL}.  Note that by the property (DQ3) we get that 
\[
 \beta \int_{\Om} \frac{|\nabla_H u^2|^2 u^1}{\varrho |u^1|^{2\beta+2}}\,\Delta_{-h}^1(\phi^2 \Delta_h^1 u^1)\, \ud x\,\ud y=-\beta \int_{\Om} \Delta_{h}^1 \left(\frac{|\nabla_H u^2|^2 u^1}{\varrho |u^1|^{2\beta+2}}\right)\, \phi^2 \Delta_h^1 u^1\, \ud x\,\ud y.
\]
Next, by the product rule for the difference quotients, see the property (DQ1), we directly expand the integrand on the right-hand side above as follows
\begin{equation}\label{eq1-N}
 \Delta_{h}^1 \left(\frac{|\nabla_H u^2|^2 u^1}{\varrho |u^1|^{2\beta+2}}\right)=\frac{u^1(x+h,y)}{|u^1(x+h,y)|^{2\beta+2}} \Delta_{h}^1\left(\frac{|\nabla_H u^2|^2}{\varrho(x)}\right)+\frac{|\nabla_H u^2|^2}{\varrho(x)} \Delta_{h}^1\left(\frac{u^1}{|u^1|^{2\beta+2}}\right).
\end{equation}
Similarly, we expand the $\Delta_h^1$-expressions on the right-hand side of~\eqref{eq1-N}:
\begin{align}
 \Delta_{h}^1\left(\frac{|\nabla_H u^2|^2}{\varrho}\right)&= \Delta_{h}^1\left(\frac{(u^2_x)^2}{\varrho}+\frac{(\varrho u^2_y)^2}{\varrho}\right) \nonumber\\
 & =\Big[u^2_x(x+h,y) \Delta_{h}^1u^2_x+u^2_x \Delta_{h}^1u^2_x\Big]\varrho^{-1}+[u^2_x(x+h,y)]^2 \Delta_{h}^1 \varrho^{-1} \nonumber \\
 &\phantom{AA}+\varrho^{-1} \Delta_{h}^1\Big( (\varrho u^2_y)^2 \Big)+\big(\varrho(x+h) u^2_y(x+h,y)\big)^2 \Delta_{h}^1 \varrho^{-1} \nonumber \\
 &=\Big[u^2_x(x+h,y) \Delta_{h}^1u^2_x+u^2_x \Delta_{h}^1u^2_x\Big]\varrho^{-1}+[u^2_x(x+h,y)]^2 \Delta_{h}^1 \varrho^{-1} \nonumber \\
&\phantom{AA}+ \varrho^{-1} \Big[ (\varrho(x+h) u^2_y(x+h,y) \Delta_{h}^1 (\varrho u^2_y)
+\varrho u^2_y \Delta_{h}^1 (\varrho u^2_y)\Big] \nonumber \\
&\phantom{AA}+\big(\varrho(x+h) u^2_y(x+h,y)\big)^2 \Delta_{h}^1 \varrho^{-1}. \label{eq2-N} \\
 \Delta_{h}^1 \left(\frac{u^1}{|u^1|^{2\beta+2}}\right)&= \frac{\Delta_{h}^1 u^1}{|u^1(x+h,y)|^{2\beta+2}}+u^1 \Delta_h^1\left(\frac{1}{|u^1|^{2\beta+2}}\right) \nonumber\\
 &=\frac{\Delta_{h}^1 u^1}{|u^1(x+h,y)|^{2\beta+2}}-u^1 \frac{\Delta_h^1 |u^1|^{2\beta+2}}{|u^1(x+h,y)|^{2\beta+2}|u^1|^{2\beta+2}}.\label{eq3-N}
\end{align}
We substitute~\eqref{eq2-N} and~\eqref{eq3-N} to ~\eqref{eq1-N} and, upon rearranging the appropriate terms, we obtain the following expression:
\begin{align}\label{eq11-N}
&\Delta_{h}^1 \left(\frac{|\nabla_H u^2|^2 u^1}{\varrho |u^1|^{2\beta+2}}\right) \nonumber \\
&\phantom{AA}=\frac{u^1(x+h,y)}{|u^1(x+h,y)|^{2\beta+2}} 
\bigg[ u^2_x(x+h,y)(\Delta_{h}^1 u^2_x)\varrho^{-1}+ u^2_x \Delta_{h}^1 u^2_x \varrho^{-1} \nonumber \\
&\phantom{AAAAAAAAAAAAAAAA}+\varrho(x+h)u^2_y(x+h,y) (\Delta_{h}^1 \varrho u^2_y)\varrho^{-1}+\varrho^{-1}(\varrho  u^2_y) \Delta_h^1 (\varrho u^2_y) \nonumber \\
&\phantom{AAAAAAAAAAAAAAAA} + |\nabla_H u^2(x+h,y)|^2\Delta_{h}^1 \varrho^{-1}\bigg] \nonumber \\
&\phantom{AAAA}+\frac{|\nabla_H u^2|^2}{\varrho} \left(\frac{\Delta_{h}^1 u^1}{|u^1(x+h,y)|^{2\beta+2}}-u^1 \frac{\Delta_h^1 |u^1|^{2\beta+2}}{|u^1(x+h,y)|^{2\beta+2}|u^1|^{2\beta+2}}\right).
\end{align}
In order to complete {\bf Step I} we combine the estimate for the term (N) in~\eqref{eq11-N} with the estimates for terms (I)-(III) and apply them in the estimate ($u^1$-est) upon including the appropriate terms in the left-hand side and also include the corresponding terms of (III) in (I) and (II), respectively. Recall that $\phi$ is supported in the ball $B(R)$. Consequently, we arrive at the following inequality
\begin{align}
&\frac12 \int_{B(R)} \Big((\Delta_{h}^1u^1_x)^2 + (\Delta_{h}^1 \varrho u^1_y)^2 \Big) \phi^2 \frac{\ud x\,\ud y}{\varrho}\nonumber \\
&\phantom{A}\lesssim_c \int_{B(R)} (\Delta_h^1 u^1)^2 \frac{|\nabla_H \phi|^2}{\varrho}+\int_{B(R)} [u^1_x(x+h,y)]^2 (\Delta_{h}^1 \varrho^{-1})^2 \phi^2 \varrho+\int_{B(R)} [u^1_y(x+h,y)]^2 (\Delta_{h}^1\varrho)^2 \frac{\phi^2}{\varrho} \nonumber \\
&\phantom{A}+\beta \int_{B(R)} \Bigg[\frac{u^1(x+h,y)}{|u^1(x+h,y)|^{2\beta+2}} 
\bigg(u^2_x(x+h,y)(\Delta_{h}^1 u^2_x)\varrho^{-1}+ \varrho(x+h) u^2_y(x+h,y)\Delta_h^1 (\varrho u^2_y) \varrho^{-1}\nonumber \\
&\phantom{AAAAAAAAAAAAA}+ u^2_x \Delta_{h}^1 u^2_x \varrho^{-1}  +\varrho^{-1}(\varrho  u^2_y) \Delta_h^1 (\varrho u^2_y) + |\nabla_H u^2(x+h,y)|^2\Delta_{h}^1 \varrho^{-1}\bigg) \nonumber \\
&\phantom{AA}+\frac{|\nabla_H u^2|^2}{\varrho} \left(\frac{\Delta_{h}^1 u^1}{|u^1(x+h,y)|^{2\beta+2}}-u^1 \frac{\Delta_h^1 |u^1|^{2\beta+2}}{|u^1(x+h,y)|^{2\beta+2}|u^1|^{2\beta+2}}\right)\Bigg]\phi^2 \Delta_h^1 u^1. \label{main-est1-H22}
\end{align} 
The resulting inequality gives a bound for the integral norms of the difference quotients for $u^1_x$ and $\varrho u^1_y$, however the right-hand side estimate in~\eqref{main-est1-H22} depends on the corresponding difference quotients for $u^2_x$ and $\varrho u^2_y$. In order to eliminate them, in the next step we conduct an analogous analysis as in Step I involving the second equation in~\eqref{w-system-EL} and then, in Step III, combine the resulting estimates. 
 \smallskip
 \\
 {\bf Step II: the integral estimate for the difference quotients $\Delta_h^1 u^2_x$ and $\Delta_h^1 (\varrho u^2_y)$}.
 \smallskip
 \\
Let $\phi \in C_0^{\infty}(B(R))$ be a test function satisfying the same assumptions as in the beginning of the discussion in Step I and, similarly, as before we observe that 

\centerline{
$\Delta_{-h}^1(\phi^2 \Delta_h^1 u^2)$ is a test function for the system of equations~\eqref{w-system-EL}.
}
\smallskip

\noindent In what follows we set $0<|h|<r<R$ and test the second equation of the system~\eqref{w-system-EL} with the above function and employ the properties (DQ2)-(DQ3) of finite difference quotients:
\begin{align}  
&0 =-\int_{\Om} \left \langle \frac{1}{|u^1|^{2\beta}}\nabla_H u^2, \nabla_H [\Delta_{-h}^1(\phi^2 \Delta_h^1 u^2)] \right \rangle \ud x\,\ud y \nonumber \\
&=\int_{\Om} \Delta_h^1 \left(\frac{u^2_x}{\varrho |u^1|^{2\beta}}\right) (\phi^2 \Delta_h^1 u^2)_x + \Delta_h^1 \left(\frac{\varrho u^2_y}{|u^1|^{2\beta}}\right) (\phi^2 \Delta_h^1 u^2)_y \frac{\ud x\,\ud y}{\varrho}. \label{eq2-weak}
\end{align}
By applying the product rule for the finite quotients, see (DQ1), we have that
\begin{align*}  
\Delta_h^1 \left(\frac{u^2_x}{\varrho}\frac{1}{|u^1|^{2\beta}}\right)&= \frac{u^2_x}{\varrho} \Delta_h^1\left(\frac{1}{|u^1|^{2\beta}}\right)+\frac{1}{|u^1(x+h,y)|^{2\beta}} \Delta_h^1\left(\frac{u^2_x}{\varrho} \right)\\
\Delta_h^1 \left(\frac{\varrho u^2_y}{|u^1|^{2\beta}}\right)&=\varrho u^2_y \Delta_h^1\left(\frac{1}{|u^1|^{2\beta}}\right)+\frac{1}{|u^1(x+h,y)|^{2\beta}}  \Delta_h^1(\varrho u^2_y).
\end{align*}
Therefore, equation~\eqref{eq2-weak} takes the following form:
\begin{align}  
&0=\int_{\Om} \left[\frac{u^2_x}{\varrho} \Delta_h^1\left(\frac{1}{|u^1|^{2\beta}}\right)+\frac{u^2_x(x+h,y) \Delta_h^1 \varrho^{-1}}{|u^1(x+h,y)|^{2\beta}}+\frac{\Delta_h^1 u^2_x}{\varrho |u^1(x+h,y)|^{2\beta}}\right] \left(\phi^2 \Delta_h^1 u^2_x+2\phi \phi_x (\Delta_h^1 u^2)\right) \nonumber \\
&\phantom{AA}+ \left[\varrho u^2_y \Delta_h^1\left(\frac{1}{|u^1|^{2\beta}}\right)+\frac{\Delta_h^1(\varrho u^2_y)}{|u^1(x+h,y)|^{2\beta}}\right] \times \nonumber \\
&\phantom{AAAA} \times \Big(\phi^2 (\varrho^{-1} \Delta_{h}^1(\varrho u^2_y)- \varrho^{-1}u^2_y(x+he^1,y)\Delta_{h}^1\varrho)+2\phi \phi_y (\Delta_h^1 u^2)\Big)  \ud x\,\ud y. \label{eq2-aux1}
\end{align}
Here, we also use (DQ1) to find that $\Delta_{h}^1u^2_y=\varrho^{-1} \Delta_{h}^1(\varrho u^2_y)- \varrho^{-1}u^2_y(x+h,y)\Delta_{h}^1\varrho$. Next, we rearrange expressions in the equation~\eqref{eq2-aux1} in a similar way as in the corresponding discussion for the integral ($u^1$-est), cf. Step I above. Furthermore, we use the similar notation (I)-(III) to denote the integrals corresponding to appropriate terms in ($u^1$-est).
\begin{align*}  
&\int_{\Om} \Bigg(\frac{(\Delta_{h}^1u^2_x)^2}{|u^1(x+h,y)|^{2\beta}} +  \frac{(\Delta_{h}^1 \varrho u^2_y)^2}{|u^1(x+h,y)|^{2\beta}} \Bigg) \phi^2 \frac{\ud x\,\ud y}{\varrho} \tag{$u^2$-est}\\
&=-\int_{\Om} \left(\frac{u^2_x}{\varrho} \Delta_h^1\left(\frac{1}{|u^1|^{2\beta}}\right) |u^1(x+h,y)|^{\beta} \varrho^{\frac12} \phi\right)\left(\frac{\Delta_h^1 u^2_x}{|u^1(x+h,y)|^\beta} \frac{\phi}{\varrho^{\frac12}}\right) \tag{I}\\
&\phantom{AA}-\int_{\Om} \left(\frac{u^2_x(x+h,y)  \Delta_h^1 \varrho^{-1}}{|u^1(x+h,y)|^{\beta}} \varrho^{\frac12} \phi\right)\left(\frac{\Delta_h^1 u^2_x}{|u^1(x+h,y)|^\beta} \frac{\phi}{\varrho^{\frac12}}\right) \tag{I} \\
&\phantom{AA}+\int_{\Om} \left(\frac{u^2_y(x+h,y)  \Delta_h^1\varrho}{ |u^1(x+h,y)|^{\beta}} \frac{\phi}{\varrho^{\frac12}} \right)\left(\frac{\Delta_h^1 (\varrho u^2_y)}{|u^1(x+h,y)|^\beta} \frac{\phi}{\varrho^{\frac12}}\right) \tag{I}\\
&\phantom{AA}-\int_{\Om} \left(\varrho u^2_y\Delta_h^1\left(\frac{1}{|u^1|^{2\beta}}\right) |u^1(x+h,y)|^{\beta}
\frac{\phi}{\varrho^{\frac12}}\right)\left(\frac{\Delta_h^1 (\varrho u^2_y)}{|u^1(x+h,y)|^\beta} \frac{\phi}{\varrho^{\frac12}}\right) \tag{I}\\
&\phantom{AA}+\int_{\Om} \left(u^2_y \Delta_h^1\left(\frac{1}{|u^1|^{2\beta}}\right) |u^1(x+h,y)|^{\beta} \varrho^{\frac12} \phi\right)\left(\frac{u^2_y(x+h,y) (\Delta_h^1\varrho)}{|u^1(x+h,y)|^\beta} \frac{\phi}{\varrho^{\frac12}}\right) \\
&\phantom{AA}-\int_{\Om} \left(2 \frac{\Delta_h^1 u^2}{|u^1(x+h,y)|^{\beta}}\frac{\phi_x}{\varrho^{\frac12}} \right) \left(\frac{ u^2_x(x+h,y) \Delta_h^1 \varrho^{-1}}{|u^1(x+h,y)|^{\beta}} \varrho^{\frac12}\phi \right) \tag{III} \\
&\phantom{AA}-\int_{\Om} \left(2 \frac{\Delta_h^1 u^2}{|u^1(x+h,y)|^{\beta}}\frac{\phi_x}{\varrho^{\frac12}} \right) \left(\frac{u^2_x}{\varrho} \Delta_h^1 \left(\frac{1}{|u^1|^{2\beta}}\right)|u^1(x+h,y)|^{\beta} \varrho^{\frac12}\phi \right) \tag{III}\\
&\phantom{AA}-\int_{\Om} \left(\varrho u^2_y \Delta_h^1\left(\frac{1}{|u^1|^{2\beta}}\right)
 |u^1(x+h,y)|^{\beta} \frac{\phi}{\varrho^{\frac12}}\right) \left(2\frac{\Delta_h^1 u^2}{|u^1(x+h,y)|^\beta} \phi_y \varrho^{\frac12}\right) \tag{III}\\
&\phantom{AA}-\int_{\Om} \left(2 \frac{\Delta_h^1 u^2}{|u^1(x+h,y)|^{\beta}}\frac{\phi_x}{\varrho^{\frac12}} \right) \left(\frac{\Delta_h^1 u^2_x}{|u^1(x+h,y)|^{\beta}}\frac{\phi}{\varrho^{\frac12}}\right)  \tag{II} \\
&\phantom{AA}-\int_{\Om} \left(\frac{\Delta_h^1 (\varrho u^2_y)}{|u^1(x+h,y)|^{2\beta}} \frac{\phi}{\varrho^{\frac12}}\right) \left(2\frac{\Delta_h^1 u^2}{|u^1(x+h,y)|^\beta} \phi_y \varrho^{\frac12}\right). \tag{II}\label{eq2-aux2}
\end{align*} 

As in the corresponding junction of Step I, see the discussion following~\eqref{est-I-III}, we estimate integrals (I)-(III) by the algebraic inequality $2ab\leq \delta^{-2} a^2+\delta^2b^2$ for various $\delta>0$ and include integrals with terms $\Delta_{h}^1u^1_x$ and $\Delta_{h}^1 \varrho u^1_y$ into the left hand side of ($u^2$-est) to get the following inequality:
\begin{align}  
&\frac14 \int_{B(R)} \Bigg(\frac{(\Delta_{h}^1u^2_x)^2}{|u^1(x+h,y)|^{2\beta}} +  \frac{(\Delta_{h}^1 \varrho u^2_y)^2}{|u^1(x+h,y)|^{2\beta}} \Bigg) \phi^2 \frac{\ud x\,\ud y}{\varrho} \nonumber \\
&\lesssim_c \int_{B(R)} \frac{(\Delta_h^1 u^2)^2}{|u^1(x+h,y)|^{2\beta}} \frac{|\nabla_H \phi|^2}{\varrho} \nonumber \\
&\phantom{AA}+\int_{B(R)} \frac{[u^2_x(x+h,y)]^2  (\Delta_h^1 \varrho^{-1})^2}{|u^1(x+h,y)|^{2\beta}} \varrho \phi^2 
+\int_{B(R)} \frac{[u^2_y(x+h,y)]^2  (\Delta_h^1 \varrho)^2 \varrho^{-1}}{|u^1(x+h,y)|^{2\beta}} \phi^2 \nonumber \\
&\phantom{AA}+\int_{B(R)} |\nabla_H u^2|^2\left(\Delta_h^1 \frac{1}{|u^1|^{2\beta}}\right)^2|u^1(x+h,y)|^{2\beta} \frac{\phi^2}{\varrho}. \label{eq2-aux3}
\end{align}

Since $(\Delta_h^1 \varrho)^2\varrho^{-1}=(\Delta_h^1 \varrho^{-1})^2 \varrho^2(x+h) \varrho $, the above integrals with $u^2_x$ and $u^2_y$ can be estimated further and so~\eqref{eq2-aux3} reads:
\begin{align}  
&\frac14 \int_{B(R)} \Bigg(\frac{(\Delta_{h}^1u^2_x)^2}{|u^1(x+h,y)|^{2\beta}} +  \frac{(\Delta_{h}^1 \varrho u^2_y)^2}{|u^1(x+h,y)|^{2\beta}} \Bigg) \phi^2 \frac{\ud x\,\ud y}{\varrho} \nonumber \\
&\lesssim_c \int_{B(R)} \frac{(\Delta_h^1 u^2)^2}{|u^1(x+h,y)|^{2\beta}} \frac{|\nabla_H \phi|^2}{\varrho} + \int_{B(R)} |\nabla_H u^2(x+h,y)|^2 \frac{(\Delta_h^1 \varrho^{-1})^2}{|u^1(x+h,y)|^{2\beta}} \varrho \phi^2\nonumber \\
&\phantom{AA}+
\int_{B(R)} |\nabla_H u^2|^2 \left( \Delta_h^1 \frac{1}{|u^1|^{2\beta}}\right)^2 |u^1(x+h,y)|^{2\beta} \frac{\phi^2}{\varrho}. \label{main-est2-H22}
\end{align}
\smallskip
\\
{\bf Step III: the integral estimate for the difference quotient  $\Delta_h^1 |D_H u|^2$}.
 \smallskip
 \\
We add up the estimates~\eqref{main-est1-H22} and~\eqref{main-est2-H22}, keeping in mind that the first one contains the terms depending on $\Delta_h^1 u^2_x$ and $\Delta_h^1 (\varrho u^2_y)$ which now can be included in the left-hand side of the resulting estimate:
\begin{align}
&\frac14 \int_{B(R)} \bigg((\Delta_{h}^1u^1_x)^2 +  (\Delta_{h}^1 \varrho u^1_y)^2 + \frac{(\Delta_{h}^1u^2_x)^2}{|u^1(x+h,y)|^{2\beta}} +  \frac{(\Delta_{h}^1 \varrho u^2_y)^2}{|u^1(x+h,y)|^{2\beta}} \bigg) \phi^2 \frac{\ud x\,\ud y}{\varrho} \nonumber \tag{$u$-est} \\
&\phantom{AAA}\lesssim_c \int_{B(R)} \bigg ( (\Delta_h^1 u^1)^2 +\frac{(\Delta_h^1 u^2)^2}{|u^1(x+h,y)|^{2\beta}}\bigg) \frac{|\nabla_H \phi|^2}{\varrho} \nonumber \\
&\phantom{AAAA}+ \int_{B(R)} \bigg( |\nabla_H u^1(x+h,y)|^2 + \frac{|\nabla_H u^2(x+h,y)|^2}{|u^1(x+h,y)|^{2\beta}}\bigg) (\Delta_h^1 \varrho^{-1})^2 \varrho \phi^2 \nonumber \\
&\phantom{AAAA}+\beta \int_{B(R)} \left|\frac{\Delta_{h}^1 u^2_x }{|u^1(x+h,y)|^{\beta}} \frac{\phi}{\varrho^{\frac12}}\right| \left|\frac{u^2_x(x+h,y) \Delta_h^1 u^1}{|u^1(x+h,y)|^{\beta+1}} \frac{\phi}{\varrho^{\frac12}}\right| \nonumber \\
&\phantom{AAAA}+\beta \int_{B(R)} \left|\frac{\Delta_{h}^1 u^2_x }{|u^1(x+h,y)|^{\beta}} \frac{\phi}{\varrho^{\frac12}}\right| \left|\frac{u^2_x\Delta_h^1 u^1}{|u^1(x+h,y)|^{\beta+1}} \frac{\phi}{\varrho^{\frac12}}\right| \nonumber \\
&\phantom{AAAA}+\beta \int_{B(R)} \left|\frac{\Delta_{h}^1 (\varrho u^2_y) }{|u^1(x+h,y)|^{\beta}} \frac{\phi}{\varrho^{\frac12}}\right| \left|\frac{\varrho(x+h) u^2_y(x+h,y) \Delta_h^1 u^1}{|u^1(x+h,y)|^{\beta+1}} \frac{\phi}{\varrho^{\frac12}}\right| \nonumber \\
&\phantom{AAAA}+\beta \int_{B(R)} \left|\frac{\Delta_{h}^1 (\varrho u^2_y) }{|u^1(x+h,y)|^{\beta}} \frac{\phi}{\varrho^{\frac12}}\right|\left|\frac{(\varrho  u^2_y) \Delta_h^1 u^1}{|u^1(x+h,y)|^{\beta+1}} \frac{\phi}{\varrho^{\frac12}}\right| \nonumber \\
&\phantom{AAAA}+\beta \int_{B(R)}\frac{1}{|u^1(x+h,y)|^{2\beta+1}} |\nabla_H u^2(x+h,y)|^2|\Delta_{h}^1 \varrho^{-1}||\Delta_h^1 u^1|\phi^2 \nonumber \\
&\phantom{AAAA}+\beta \int_{B(R)} \frac{|\nabla_H u^2|^2}{\varrho} \left|\Delta_{h}^1 \left(\frac{u^1}{|u^1|^{2\beta+2}}\right)\right| |\Delta_h^1 u^1|\phi^2 \nonumber \\
&\phantom{AAAA}+\int_{B(R)} |\nabla_H u^2|^2 \left(\Delta_h^1\frac{1}{|u^1|^{2\beta}}\right)^2 |u^1(x+h,y)|^{2\beta}\frac{\phi^2}{\varrho}. \label{main-est-H22}
\end{align}
As in Steps I and II, we apply the algebraic inequality $2ab\leq \delta^{-2} a^2+\delta^2b^2$ for various $\delta>0$ and include integrals with terms $\Delta_{h}^1u^1_x$ and $\Delta_{h}^1 \varrho u^1_y$ into the left hand side of ($u$-est) to get the following inequality:
\begin{align}
&\frac14 \int_{B(R)} \bigg((\Delta_{h}^1u^1_x)^2 +  (\Delta_{h}^1 \varrho u^1_y)^2 + \frac{(\Delta_{h}^1u^2_x)^2}{|u^1(x+h,y)|^{2\beta}} +  \frac{(\Delta_{h}^1 \varrho u^2_y)^2}{|u^1(x+h,y)|^{2\beta}} \bigg) \phi^2 \frac{\ud x\,\ud y}{\varrho} \nonumber \tag{$u$-est} \\
&\phantom{AAA}\lesssim_c \int_{B(R)} \bigg ( (\Delta_h^1 u^1)^2 +\frac{(\Delta_h^1 u^2)^2}{|u^1(x+h,y)|^{2\beta}}\bigg) \frac{|\nabla_H \phi|^2}{\varrho} \nonumber \\
&\phantom{AAAA}+ \int_{B(R)} \bigg( |\nabla_H u^1(x+h,y)|^2 + \frac{|\nabla_H u^2(x+h,y)|^2}{|u^1(x+h,y)|^{2\beta}}\bigg) (\Delta_h^1 \varrho^{-1})^2 \varrho \phi^2 \nonumber \\
&\phantom{AAAA}+\beta^2 \int_{B(R)} \frac{|\nabla_H u^2(x+h,y)|^2}{|u^1(x+h,y)|^{2\beta+2}}  (\Delta_h^1 u^1)^2\frac{\phi^2}{\varrho} \nonumber \\
&\phantom{AAAA}+(\beta^2+\beta) \int_{B(R)} \frac{|\nabla_H u^2|^2}{|u^1(x+h,y)|^{2\beta+2}}  (\Delta_h^1 u^1)^2\frac{\phi^2}{\varrho} \nonumber \\
&\phantom{AAAA}+\beta \int_{B(R)}\frac{|\nabla_H u^2(x+h,y)|^2}{|u^1(x+h,y)|^{2\beta+1}} \left| \Delta_{h}^1 \varrho^{-1}\right|\,|\Delta_h^1 u^1|\phi^2 \nonumber \\
&\phantom{AAAA}+\int_{B(R)} \frac{|\nabla_H u^2|^2 }{|u^1|^{4\beta}|u^1(x+h,y)|^{2\beta}}(\Delta_h^1|u^1|^{2\beta})^2\frac{\phi^2}{\varrho}. \label{main-est-H22}
\end{align}

Note that in the last integral we also employ the following identity: $\Delta_h^1\frac{1}{|u^1|^{2\beta}}=-\frac{\Delta_h^1|u^1|^{2\beta}}{|u^1|^{2\beta}|u^1(x+h,y)|^{2\beta}}$.  Moreover, we use that
\begin{align}
&\Delta_{h}^1 \left(\frac{u^1}{|u^1|^{2\beta+2}}\right)= \frac{\Delta_{h}^1 u^1}{|u^1(x+h,y)|^{2\beta+2}}-\frac{u^1}{|u^1|^{2\beta+2}} \frac{\Delta_{h}^1 |u^1|^{2\beta+2}}{|u^1(x+h,y)|^{2\beta+2}} \nonumber \\
&\left|\Delta_{h}^1 \left(\frac{u^1}{|u^1|^{2\beta+2}}\right)\right| \leq \frac{|\Delta_{h}^1 u^1|}{|u^1(x+h,y)|^{2\beta+2}}+(2\beta+2) \frac{|\Delta_{h}^1 u^1|}{|u^1(x+h,y)|^{2\beta+2}}. \label{main-est-aux3}
\end{align}
Here, we appeal to the fact that, since $u^1\in HW^{1,2}(\Om)\cap u^1 \in L^\infty_{loc}$ and $u^1\not=0$ in $\Om$, then also $|u^1|$ and $|u^1|^{2\beta+2}$ belong to $ HW^{1,2}(\Om)$ and by applying the characterization of the Sobolev functions in Lemma~\ref{lem-Sob-char} we also justify the $L^2$-integrability in~\eqref{main-est-aux3}. Thus, we have that
\begin{equation}
|\nabla_H u^2|^2 \left| \Delta_{h}^1 \left(\frac{u^1}{|u^1|^{2\beta+2}}\right)\right|\, |\Delta_h^1 u^1|\lesssim_{\beta} \frac{|\nabla_H u^2|^2}{|u^1(x+h,y)|^{2\beta+2}}(\Delta_h^1 u^1)^2.   \label{main-est-aux4}
\end{equation}
 As for the last integral in~\eqref{main-est-H22}, we follow the same approach as above to obtain the following estimate:
\begin{equation*}
(\Delta_h^1|u^1|^{2\beta})^{2}\lesssim_{\beta} \left(\frac{\Delta_h^1 u^1}{|u^1|^{1-2\beta}}\right)^{2}.
\end{equation*}

Recall that $\phi$ satisfies: $\phi\equiv 1$ on $B(R/2)$ and $|\nabla_H \phi|\leq \frac{c}{R}$. By Lemma~\ref{lem-Sob-char}, the above integral with the term $(\Delta_h^1 u^1)^2$ is bounded by the norm $\|D_H u\|^2_{L^2(B(R), \mu)}$, up to a constant depending also on $1/R^2$.  The same holds as well for the integral with the term $\frac{(\Delta_h^1 u^2)^2}{|u^1(x+h,y)|^{2\beta}}$, upon applying Remark~\ref{rem-Sob-char2} and due to assumption that $\overline{u^1(\Om)}\cap S=\emptyset$. 

Thus, we may further simplify the estimate~\eqref{main-est-H22} and obtain that
\begin{align}
&\frac14 \int_{B(R/2)} \bigg((\Delta_{h}^1u^1_x)^2 +  (\Delta_{h}^1 \varrho u^1_y)^2 + \frac{(\Delta_{h}^1u^2_x)^2}{|u^1(x+h,y)|^{2\beta}} +  \frac{(\Delta_{h}^1 \varrho u^2_y)^2}{|u^1(x+h,y)|^{2\beta}} \bigg)  \frac{\ud x\,\ud y}{\varrho} \nonumber \tag{$u$-est} \\
&\phantom{AAA}\lesssim_c  \frac{1}{R^2} \|u\|^2_{H^{1,2}(B(R), \mu)}
+ \int_{B(R)} |D_H u(x+h,y)|^2 \frac{(\Delta_h^1 \varrho)^2}{\varrho(x+h)^2}  \frac{\ud x\,\ud y}{\varrho} \label{main-est-Dh} \\
&\phantom{AAAA}+\beta^2 \int_{B(R)} \frac{|\nabla_H u^2(x+h,y)|^2}{|u^1(x+h,y)|^{2\beta+2}}\, {\color{blue}(\Delta_h^1 u^1)^2} \frac{\ud x\,\ud y}{\varrho} \nonumber \tag{I}\\
&\phantom{AAAA}+\beta \int_{B(R)} \frac{|\nabla_H u^2|^2}{|u^1|^{2\beta+2}}{\color{blue} \frac{|u^1|^{2\beta+2}}{|u^1(x+h,y)|^{2\beta+2}}  (\Delta_h^1 u^1)^2} \frac{\ud x\,\ud y}{\varrho} \nonumber \tag{II}\\
&\phantom{AAAA}+\beta \int_{B(R)}\frac{|\nabla_H u^2(x+h,y)|^2}{|u^1(x+h,y)|^{2\beta+2}} {\color{blue}|\Delta_h^1 u^1|\,\frac{\left| \Delta_{h}^1 \varrho\right|}{\varrho(x+h)} |u^1(x+h,y)|} \frac{\ud x\,\ud y}{\varrho} \nonumber \tag{III}\\
&\phantom{AAAA}+\int_{B(R)} \frac{|\nabla_H u^2|^2 }{|u^1|^{2\beta+2}}{\color{blue} \frac{|u^1|^{2\beta}}{|u^1(x+h,y)|^{2\beta}}(\Delta_h^1 u^1)^2}\frac{\ud x\,\ud y}{\varrho}. \label{main-est-H22-2} \tag{IV}
\end{align}

Here we use colours to distinguish expressions in the above integrands to which we apply the H\"older inequality and assumption~\eqref{ass1-thm-H22} as explained below.  By following our notation convention we denote the aforementioned integrals by (I)-(IV) and analyze them one by one.
\smallskip
\\
By applying the H\"older inequality to integrals (I) and (III) with exponents $\frac{p}{2}$ and $\frac{p}{p-2}$ followed by application of~\eqref{ass1-thm-H22} we obtain that
\begin{equation}
\int_{B(R)} \frac{|\nabla_H u^2(x+h,y)|^p}{|u^1(x+h,y)|^{(\beta+1)p}} \frac{\ud x \ud y}{\varrho(x+h)}
\leq  \int_{B(R+r)} \frac{|\nabla_H u^2|^p}{|u^1|^{(\beta+1)p}} \ud \mu<\infty. \label{main-est-aux4}
\end{equation}
Similarly we handle parts of the integrands in (II) and (IV) with the $|\nabla_H u^2|^2$-terms. 
The remaining expressions in (I)-(IV) result in the following integrals, respectively:
\begin{align}
&\int_{B(R)} \left(\frac{\varrho(x+h)}{\varrho(x)}\right)^{\frac{2}{p-2}}(\Delta_h^1 u^1)^{\frac{2p}{p-2}} \ud \mu\leq (R+r)^{\frac{2\alpha}{p}}\int_{B(R)} \frac{(\Delta_h^1 u^1)^{\frac{2p}{p-2}}}{\varrho(x)^{\frac{2}{p-2}}} \ud \mu \nonumber \\
& \int_{B(R)} \left(\frac{ |u^1|}{|u^1(x+h,y)|}\right)^{(\beta+1)\frac{2p}{p-2}} (\Delta_h^1 u^1)^{\frac{2p}{p-2}} \ud \mu \tag{$\star$} \label{main-est-aux2} \\
&\|u^1\|_{L^{\infty}(B(R+r))}^{\frac{p}{p-2}}\int_{B(R)} \left(|\Delta_h^1 u^1| \frac{\left| \Delta_{h}^1 \varrho\right|}{\varrho(x+h)\varrho(x)}\varrho(x+h)^{\frac{2}{p}}\right)^{\frac{p}{p-2}} \ud x \ud y,\label{main-est-aux66} \\
&\int_{B(R)} \left(\frac{|u^1|}{|u^1(x+h,y)|}\right)^{2\beta\frac{p}{p-2}} (\Delta_h^1 u^1)^{\frac{2p}{p-2}} \ud \mu  \nonumber
 \end{align}
 Note that the second and the fourth integral in~\eqref{main-est-aux2} give rise to the same integral, as by the assumptions of the theorem $p>2$ and the image $u^1(B(R))$ is separated from the singular set and hence bounded from below on $B(2R)$. As for the third integral above, it arises, due to the boundedness of $|u^1(x+h,y)|$ and will be analyzed
at the end of this step, see the discussion following after~\eqref{est-XY-thm12}.
%
  Moreover, since by the assumption~\eqref{ass1-thm-H22} the partial derivative $u^1_x$ is integrable in the power $\frac{2p}{p-2}>2$ with respect to the measure $\mu$, we have trivially that it is also integrable with respect to the Lebesgue measure:
 $\|(u^1_x)^\frac{2p}{p-2}\|_{L^1(B(R), \ud x\ud y)} \leq R^\alpha\|(u^1_x)^\frac{2p}{p-2}\|_{L^1(B(R), \ud \mu)}$. 
Thus, the Morrey embedding theorem applies and $u^1_x$ is continuous and so is, in particular, $u^1$ with respect to variable $x$. As a consequence, the integrability of the second and the fourth integral in~\eqref{main-est-aux2} reduces to the integrability of the first integral in~\eqref{main-est-aux2}, which holds by assumption~\eqref{ass1-thm-H22}.
The analysis of the integrability of integrals (I)-(IV) in the estimate ($u$-est) will be completed once we handle the remaining two integrals, i.e., ~\eqref{main-est-Dh} and~\eqref{main-est-aux66}. 

Let us consider two cases (recall the assumption that $0<|h|<r$):
\[
 (1)\, \overline{B(R+r)}\cap S=\emptyset \quad \hbox{ and }\quad (2)\, \overline{B(R+r)}\cap S\not=\emptyset.
\]
{\bf Case 1:} In the first case the difference quotient $\frac{\Delta_{h}^1 \varrho} {\varrho(x+h)}$ exists for all points in $B(R)$ and, moreover, converges pointwisely to the derivative $\varrho_x$ of $\varrho(x)=|x|^\alpha$, as $h\to 0$. Namely, it holds that
\[
\frac{|\Delta_{h}^1 \varrho|}{\varrho(x+h)} \to_{h\to 0} \frac{|\varrho_x(x)|}{\varrho(x)}=\frac{\alpha}{|x|}<\frac{\alpha}{\dist(B(R+r), S)}.
\]
This observation allows us to provide the following estimates:
\begin{align*}
\int_{B(R)} |D_H u(x+h,y)|^2 \frac{(\Delta_h^1 \varrho)^2}{\varrho(x+h)^2}  \frac{\ud x\,\ud y}{\varrho}
&\lesssim_{\alpha, \dist(B(R+r),S)} \int_{B(R)}\,|D_H u(x+h,y)|^2 \frac{\ud x\,\ud y}{\varrho}\\
&\lesssim_{\alpha, \dist(B(R+r),S)} \|D_H u\|^2_{L^2(B(R+r), \mu)}<\infty.
\end{align*}
We follow the same reasoning to complete the analysis of the integral (III). Namely, we have by~\eqref{main-est-aux3} that 
\begin{align*}
\int_{B(R)} \frac{|\nabla_H u^2(x+h,y)|^p}{|u^1(x+h,y)|^{(\beta+1)p}} \left(\frac{\left| \Delta_{h}^1 \varrho\right|}{\varrho(x+h)}\right)^{\frac{p}{2}}\, \frac{\ud x\,\ud y}{\varrho}
&\lesssim_{\alpha, p, \dist(B(R+r),S)} \int_{B(R)}\frac{|\nabla_H u^2(x+h,y)|^p}{|u^1(x+h,y)|^{(\beta+1)p}}\,\frac{\ud x\,\ud y}{\varrho} <\infty.
\end{align*}
Thus, the integral on the left-hand side of ($u$-est) is finite and, so by applying Lemma~\ref{lem-Sob-char} and Remark~\ref{rem-Sob-char2} the assertion of Theorem~\ref{thm-H22} is proven for $X^2v$ and $XYv$ in the case (1). 
\smallskip
\\

\noindent {\bf Case 2:} In the second case, that is when the ball $B(R)$ intersects the singular set $S$, we need to ensure the integrability of integrals in~\eqref{main-est-Dh}, by imposing additional condition on mapping $u$. Namely, for $\gamma$ as in the assumptions~\eqref{ass0-thm-H22} and~\eqref{ass2-thm-H22}, it holds 
\begin{align}
&\int_{B(R)} |D_H u(x+h,y)|^2 \frac{(\Delta_h^1 \varrho)^2}{\varrho(x+h)^2}  \frac{\ud x\,\ud y}{\varrho} \nonumber \\
&\phantom{AAAA}=\int_{B(R)} \frac{|D_H u(x+h,y)|^2}{\varrho^{\frac{2}{p}\gamma}(x+h)} \frac{(\Delta_{h}^1\varrho)^2}{\varrho}\,\frac{\ud x\ud y}{\varrho^{2-\frac{2}{p}\gamma}(x+h)}\nonumber \\
&\phantom{AAAA} \leq \left(\int_{B(R)} \frac{|D_H u(x+h,y)|^p}{\varrho(x+h)^\gamma}\ud x\ud y \right)^{\frac{2}{p}}\left(\int_{B(R)}\left(\frac{(\Delta_{h}^1\varrho)^2}{\varrho}\right)^{\frac{p}{p-2}} \varrho(x+h)^{\frac{2(\gamma-p)}{p-2}}\ud x\ud y \right)^{\frac{p-2}{p}}. \label{main-est-aux22}
\end{align}
Upon the linear change of variables in the first integral, we obtain the integral which is finite by the assumption~\eqref{ass2-thm-H22}. Therefore, we are left with verifying that under that assumption on $p$ and $\gamma$ also the second integral is finite. 

The following discussion repeats the corresponding presentation in the proof of Theorem 1.2 in~\cite{aww}, but we nevertheless present it here for the readers convenience. Note that for $x\not=0$ it holds that for $h\to 0$, the integrand converges pointwisely to function 
$$
\left(\frac{\varrho_x^2}{\varrho}\right)^{\frac{p}{p-2}}\varrho^{\frac{2(\gamma-p)}{p-2}}\approx_{\alpha} |x|^{(\alpha-2)\frac{p}{p-2}+\alpha{\frac{2(\gamma-p)}{p-2}}},
$$
which is integrable on $B(R)$ provided that $\gamma>\frac{2+p+p\alpha}{2\alpha}$. Next, we note that by the Calculus 
$$
|x+h|^{1-\alpha} |\varrho(x+h)-\varrho(x)|\leq c(\alpha, R)|h|.
$$
Therefore, we find the following estimate 
\begin{align*}
 \sup_{0<|h|<R} \left(\frac{(\Delta_{h}^1\varrho)^2}{\varrho}\right)^{\frac{p}{p-2}} \varrho(x+h)^{\frac{2(\gamma-p)}{p-2}} &=  \sup_{0<|h|<R} \left(\frac{(|x+h|^{1-\alpha}\Delta_{h}^1\varrho)^2}{|x+h|^{2-2\alpha} \varrho}\right)^{\frac{p}{p-2}} \varrho(x+h)^{\frac{2(\gamma-p)}{p-2}} \nonumber \\
& \lesssim_{\alpha, R}  \sup_{0<|h|<R} |x|^{-\frac{\alpha p}{p-2}} \frac{\varrho(x+h)^{\frac{2(\gamma-p)}{p-2}}}{\varrho(x+h)^{\frac{2p(1-\alpha)}{p-2}}} \nonumber \\
 &  \lesssim_{\alpha, R} |x|^{-\frac{\alpha p}{p-2}} |x+h|^{\alpha\frac{2(\gamma-p)-2p(1-\alpha)}{p-2}}.
\end{align*}
If $p>2$ and $\gamma>p(2-\alpha)$, then the power $\delta:=\frac{\alpha}{p-2}[2(\gamma-p)-2p(1-\alpha)]>0$ and so for a suitable positive constant we have that $|x+h|^\delta \lesssim_{c(\delta)}|x|^\delta+R^\delta$. Thus,
\begin{equation}\label{est-XY-thm12}
|x|^{-\frac{\alpha p}{p-2}} |x+h|^{\alpha\frac{2(\gamma-p)-2p(1-\alpha)}{p-2}}\lesssim_{c(\delta)} |x|^{-\frac{\alpha p}{p-2}+\alpha\frac{2(\gamma-p)-2p(1-\alpha)}{p-2}}+|x|^{-\frac{\alpha p}{p-2}}R^{\alpha\frac{2(\gamma-p)-2p(1-\alpha)}{p-2}}.
\end{equation}
One directly verifies that $\gamma\geq \frac{2+p+p\alpha}{2\alpha}>\frac{p(5\alpha-2\alpha^2-1)+2}{2\alpha}$ and also $\frac{2+p+p\alpha}{2}\geq p(2-\alpha)$ for $0<\alpha<\frac12$, while the opposite inequality holds if $\frac12<\alpha<1$ and $p<\frac{2}{(1-\alpha)(2\alpha-1)}$. Hence, the integrability of the expression in~\eqref{est-XY-thm12} holds under the condition~\eqref{ass2-thm-H22}. Thus, we reduce the discussion to the previously discussed integral $\int_{B(R)} (\Delta_h^1 u^1)^{\frac{2p}{p-2}} \ud \mu$ which is finite by the assumption~\eqref{ass1-thm-H22}. As a consequence, we can appeal to the Lebesgue convergence theorem and the integrability of~\eqref{main-est-aux22} is proven.

In order to handle the integral~\eqref{main-est-aux66} we appeal to the assumption~\eqref{ass3-thm-H22} and observe that, since now $B(R)\cap S\not=\emptyset$, we get at points $(x,y)\in B(R)$ with $x\not=0$ that
\[
 |\Delta_h^1 u^1| \frac{|\Delta_{h}^1 \varrho|}{\varrho(x+h)^{1-\frac{2}{p}}\varrho(x)} \to_{h\to 0}|u^1_x| \frac{|\varrho_x(x)|}{\varrho(x)^{2-\frac{2}{p}}}=\alpha |u^1_x| |x|^{\frac{2\alpha}{p}-\alpha-1}
\] 
and hence
\[
 \int_{B(R)} \left(|u^1_x| |x|^{\frac{2\alpha}{p}-\alpha-1}\right)^{\frac{p}{p-2}}\ud x\ud y=\int_{B(R)} \left(\frac{|u^1_x|}{|x|}\right)^{\frac{p}{p-2}}\ud \mu<\infty.
\]
Thus, application of the Lebesgue convergence theorem gives the integrability of~\eqref{main-est-aux66}. 
 We complete the proof of the assertion of Theorem~\ref{thm-H22} for $X^2v$ and $XYv$ in the case (2) by appealing to Lemma~\ref{lem-Sob-char} and Remark~\ref{rem-Sob-char2}.
\smallskip
\\
 {\bf The derivatives $YX$ and $Y^2$}.
 \smallskip
 \\
 This step of the proof is simpler than the previous ones, as now the studied difference quotients are $\Delta_h^2u^i_x$ and $\Delta_h^2u^i_y$ for $i=1,2$ and thus the weight $\varrho$ does not appear explicitly inside $\Delta_h^2$. Moreover, a number of expressions studied in Steps I-III are now zero, as $\varrho$ depends only on the $x$ variable, e.g. $\Delta_h^2 \varrho^{-1} \equiv 0$ and the same for $\Delta_h^2 \varrho\equiv 0$.
 
 Similarly to the previous steps, we consider the following test function $\Delta_{-h}^2(\phi^2 \Delta_h^2 u^1)$ which belongs to $HW^{1,2}_0(B(R))$, provided that $\phi \in  HW^{1,2}_0(B(R))$, where $2B(R)\Subset \Om$ and $\phi$ is such that $0\leq \phi \leq 1$ on $B(R)$, $\phi\equiv 1$ on $B(R/2)$ and $|\nabla_H \phi|\leq \frac{c}{R}$.
 
 As previously, we test the first equation~\eqref{system-EL} with the above function and appeal to properties (DQ1)-(DQ3) of the finite difference quotients to obtain the following identity:
\begin{align*}  
&\beta \int_{\Om} \frac{|\nabla_H u^2|^2 u^1}{\varrho |u^1|^{2\beta+2}}\,\Delta_{-h}^2(\phi^2 \Delta_h^2 u^1)\,\ud x\,\ud y \\
&\phantom{AA}=-\int_{\Om} \langle \nabla_H u^1, \nabla_H [\Delta_{-h}^2(\phi^2 \Delta_h^2 u^1)] \rangle \frac{\ud x\,\ud y}{\varrho} \\
&\phantom{AA}=-\int_{\Om} \Big( \frac{u^1_x}{\varrho}\, \Delta_{-h}^2(\phi^2 \Delta_h^2 u^1)_x+ \varrho u^1_y\,\Delta_{-h}^2(\phi^2 \Delta_h^2 u^1)_y \Big) \ud x\,\ud y \\
&\phantom{AA}= \int_{\Om} \Big( \Delta_{h}^2\left(\frac{u^1_x}{\varrho}\right) (\phi^2 \Delta_h^2 u^1)_x+ \Delta_{h}^1(\varrho u^1_y) (\phi^2 \Delta_h^2 u^1)_y \Big) \ud x\,\ud y \\
&\phantom{AA}=\int_{\Om} \varrho^{-1} \Delta_{h}^2u^1_x \Big( \phi^2 \Delta_h^2 u^1_x+ 2\phi \phi_x (\Delta_h^2 u^1) \Big) +\varrho \Delta_{h}^2 u^1_y \Big( \phi^2 \Delta_{h}^2 u^1_y + 2\phi \phi_y (\Delta_h^2 u^1) \Big) \ud x\,\ud y.
\end{align*} 
Then, as in Step I above, upon omitting similar details we arrive at the inequality
\begin{align*}  
&\int_{\Om} \Big((\Delta_{h}^2u^1_x)^2 +  \varrho^2 (\Delta_{h}^2 u^1_y)^2 \Big) \phi^2 \frac{\ud x\,\ud y}{\varrho}\\
&\leq \beta \int_{\Om} \Delta_{h}^2 \left(\frac{|\nabla_H u^2|^2 u^1}{\varrho |u^1|^{2\beta+2}}\right)\, \phi^2 \Delta_h^1 u^2\, \ud x\,\ud y \\
&\phantom{A}+\frac12 \int_{\Om} (\Delta_h^2 u^1_x)^2 \frac{\phi^2}{\varrho}+2\int_{\Om} (\Delta_h^2 u^1)^2 \frac{\phi_x^2}{\varrho}+\frac12 \int_{\Om} \varrho (\Delta_h^2 u^1_y)^2 \phi^2 +2\int_{\Om} (\Delta_{h}^2 u^1)^2 \varrho \phi_y^2.
\end{align*} 
As in Step I, the most challenging is now to estimate the first integral on the right-hand side above. Upon repeating simplified computations at~\eqref{eq1-N}-\eqref{eq11-N} we obtain that 
\begin{align}\label{step4-N}
& \frac12 \int_{B(R)} \Big((\Delta_{h}^2u^1_x)^2 +  \varrho^2 (\Delta_{h}^2 u^1_y)^2 \Big) \phi^2 \frac{\ud x\,\ud y}{\varrho} \\
&\leq 2\int_{B(R)} (\Delta_h^2 u^1)^2 |\nabla_H \phi|^2 \frac{\ud x\,\ud y}{\varrho}\nonumber \\
&\phantom{A}+\beta \int_{B(R)}\Big[ \frac{u^1(x,y+h)}{|u^1(x,y+h)|^{2\beta+2}} \bigg[ u^2_x(x,y+h)(\Delta_{h}^2 u^2_x)\varrho^{-1}+ u^2_x(\Delta_{h}^2 u^2_x)\varrho^{-1} \nonumber \\
&\phantom{AAAAAAAAAAAAAAAAAAAA}+\varrho u^2_y(x,y+h) (\Delta_{h}^2 \varrho u^2_y)+\varrho u^2_y (\Delta_h^2 u^2_y) \bigg] \nonumber \\
&\phantom{AAAA}+\frac{|\nabla_H u^2|^2}{\varrho} \left(\frac{\Delta_{h}^2 u^1}{|u^1(x,y+h)|^{2\beta+2}}-u^1 \frac{\Delta_h^2 |u^1|^{2\beta+2}}{|u^1(x,y+h)|^{2\beta+2}|u^1|^{2\beta+2}}\right)\Big]\phi^2\Delta_h^2u^1\,\ud x\,\ud y. \nonumber
\end{align} 
Similarly as in Step II, we next find the estimates for $\Delta_h^2 u^2_x$ and $\Delta_h^2 u^2_y$, which boils down to testing the second equation of the system~\eqref{w-system-EL} with the test function $\Delta_{-h}^2(\phi^2 \Delta_h^2 u^2)$. In consequence of repeating the discussions at~\eqref{eq2-weak}-\eqref{main-est2-H22} and at~\eqref{main-est-Dh}-(V), upon the necessary modifications, we arrive at the following inequality 
\begin{align}  
&\frac14 \int_{B(R)} \Bigg(\frac{(\Delta_{h}^2u^2_x)^2}{|u^1(x,y+h)|^{2\beta}} +  \frac{\varrho^2(\Delta_{h}^2 u^2_y)^2}{|u^1(x,y+h)|^{2\beta}} \Bigg) \phi^2 \frac{\ud x\,\ud y}{\varrho} \nonumber \\
&\lesssim_c \int_{B(R)} \frac{(\Delta_h^2 u^2)^2}{|u^1(x,y+h)|^{2\beta}} \frac{|\nabla_H \phi|^2}{\varrho} 
+c(R)\,\int_{B(R)} |\nabla_H u^2|^2 \left( \Delta_h^2 \frac{1}{|u^1|^{\beta}}\right)^2 |u^1(x,y+h)|^{2\beta}\frac{\phi^2}{\varrho} \nonumber \\
&\lesssim_c  \frac{1}{R^2} \|u\|^2_{H^{1,2}(B(R), \mu)}+c(R) \int_{B(R)} \frac{|\nabla_H u^2|^2 }{|u^1|^{2\beta+2}}|u^1|^{2\beta} (\Delta_h^1 u^1)^2 \frac{\ud x\,\ud y}{\varrho}. \label{main-est2-H22-3}
\end{align}
Therefore, by the discussion similar to the one at the estimates~\eqref{main-est-aux3}-\eqref{main-est-aux66}, the assumption~\eqref{ass1-thm-H22} suffices to infer the finiteness of the integrals~\eqref{main-est2-H22-3}. 
\end{proof}

\section{Bochner formula for harmonic mappings}

The main goal of this section is to develop the Bochner identity for harmonic mappings between open sets in Grushin planes.
Such an identity is well-known in the setting of smooth functions on $n$-dimensional Riemannian manifolds and is nowadays one of the indispensable and fundamental tools in geometric analysis. Recall its classical formulation
\begin{equation*}
	\Delta \frac{|\nabla v|^2}{2}=\langle \nabla \Delta v, \nabla v\rangle+|\nabla^2v|^2+{\rm Ric}(\nabla v, \nabla v),
\end{equation*}
where ${\rm Ric}$ stands for the Ricci curvature tensor, see Sections 1 and 7 in~\cite{li}. A particularly important case arises, when function $v$ satisfies the Beltrami--Laplace equation on a smooth manifold: $\Delta v=0$ or, more generally, a non-homogeneous equation $\Delta v=c$ for some $c\in \R$ and the Ricci tensor is nonnegative everywhere. Then we observe that
\[
\Delta \frac{|\nabla v|^2}{2}\geq |\nabla^2v|^2\geq 0
\]
and, thus, $|\nabla v|^2$ is subharmonic. However, this no longer need to be the case in the subriemannian setting of Carnot groups, see e.g.~\cite{ga1, ga23}. Indeed, it turns out that the subharmonicity property depends on whether one considers the left- or right-invariant vector fields, see~\cite[Proposition 3.4]{ga23}. Nevertheless, in the setting of harmonic mappings from a Grushin plane to $\Rn$, a Bochner identity can be proven, see~\cite[Section 4]{aww}. Recall that such mappings satisfy the (uncoupled) system of PDEs, see the discussion in Section 2.3 in~\cite{aww}. 

Below we make next step and handle the more difficult case of harmonic mappings between open sets in $G_2^\alpha$ and $G_2^\beta$ for $\alpha \in [0,1)$ and $\beta \geq 0$. The following are the key differences between the results in~\cite{aww} and the current work:  
\begin{itemize}
\item harmonic maps  previously studied in~\cite{aww} satisfy the (uncoupled) system of PDEs and so their analysis, in fact, reduces to studies of each component function of a map. Now the system is coupled and so both component functions influence the computations;
\item the geometry of both, the source- and the target Grushin planes comes into play which is reflected in the complexity of the Bochner identity and appearance of the additional terms.
\end{itemize}

In the discussion below we need the following auxiliary differential operator. Let function $v:G_2^\alpha \to \R$ be in the Sobolev space $HW^{1,2}_{loc}$. Then in the weak sense we define the following second order operator $L$:
\begin{equation}\label{def-L-bochner}
L(v)= \varrho^2 \divGt \left(\frac{\nabla_H v}{\varrho}\right)=\varrho v_{xx}-\varrho_{x}v_{x}+\varrho^3 v_{yy},\quad \varrho(x)=|x|^\alpha.
\end{equation}
The latter strong formula for $L$ holds under the assumption that $v\in C^2_{loc}$ and $x\not=0$. However, if $x=0$ we may assume that $L(v)(x)=0$ without loss of generality, since
\[
 |x|^2L(v)=|x|^{\alpha+2}v_{xx}-\alpha |x|^{\alpha}xv_x+|x|^{3\alpha+2}v_{yy}.
\]
Otherwise at $x\not=0$, the regularity $H^{1,2}_{loc}$ together with the assumption~\eqref{ass1-thm-H22} suffice for the strong form of $L$ to hold, provided that $v$ is a component function of the weakly harmonic mapping~\eqref{w-system-EL}. Indeed, for points outside the $y$-axis, the last assertion of Theorem~\ref{thm-H22} implies the $H^{2,2}_{loc}$-regularity of Grushin-harmonic mappings provided that~\eqref{ass1-thm-H22} hold.  Furthermore, note that in the special case of $\alpha=0$, and so when the Grushin plane $G_2^\alpha$ agrees with $\R^2$, we retrieve the Laplace operator, i.e. $L=\Delta$. 

Recall that the (Grushin) Hessian matrix of a function $v$ and its norm are defined, respectively, as follows (at points where the second order derivatives of $v$ exist):
\begin{align}
& \nabla^2_H v:=\left(
 \begin{array}{cc}
XXv & XYv\\
YXv & YYv
\end{array}
\right)
\,=\,\left(
 \begin{array}{cc}
v_{xx} & \varrho_{x}v_y+\varrho v_{yx}\\
\varrho v_{xy} & \varrho^2 v_{yy}
\end{array}
\right) \nonumber \\
& \|\nabla_H^2 v\|_{H}^2:={\rm tr} \left[\nabla_H^2 v \cdot (\nabla_H^2 v)^T)\right]=v_{xx}^2+\varrho_{x}^2v_{y}^2+2\varrho\varrho_{x}v_y v_{yx}+\varrho^2v_{yx}^2+\varrho^2v_{xy}^2+\varrho^4v_{yy}^2\geq 0. \label{def-hessian-fun}
\end{align}
Hence, the norm of $D^2_H u$ of the Hessian matrix of $u$ can be given by the following formula:
\[
\|D^2_H u\|^2:=\|\nabla_H^2 u^1\|_{H}^2+\frac{\|\nabla_H^2 u^2\|_{H}^2}{|u^1|^{2\beta}}.
\]
Moreover, in order to state the Bochner formula for Grushin-harmonic mappings we need to recall the following definition of the commutator, see Definition 4.1 in~\cite{aww}. Such a commutator has been defined in ~\cite{fv} for the Grushin plane equipped with the weight $\varrho(x)=x$.
\begin{defn}\label{defns-Tcom}
Suppose the real-valued function $v\in HW^{1,2}_{loc}$ has the derivatives $XYv$ and $YXv$ as distributions. We define the following operator, given by the Lie bracket of vector fields $X$ and $Y$
\begin{equation}\label{def-T}
 Tv:=[X,Y]v=X(Yv)-Y(Xv)=(\varrho v_y)_x-\varrho(v_x)_y=\varrho_{x} v_y.
\end{equation}
Furthermore, we define the following commutator of $T$ and vector fields $X$ and $Y$:
\begin{equation}\label{def-T-commutator}
(TXv)Yv-(TYv)Xv=\varrho\varrho_{x}(v_yv_{yx}-v_{x}v_{yy}).
\end{equation}
\end{defn}
 Note that it holds that $Tv=\varrho_{x} v_y=\alpha |x|^{\alpha}x^{-1} v_y$ and so $Tv$ is well-defined and $L^1_{loc}$-integrable, for example, for Grushin-harmonic mappings whose source domain does not intersect the singular set $S$ in $G_2^\alpha$; also if $v$ is $HW^{1,2}$ component function of a mapping satisfying assumptions~\eqref{ass0-thm-H22} and~\eqref{ass2-thm-H22} of Theorem~\ref{thm-H22}. Similarly analysis gives us conditions for the commutator in~\eqref{def-T-commutator} to be well defined. Namely, it holds that
\[
 (TXv)Yv-(TYv)Xv=\varrho\varrho_{x}(v_yv_{yx}-v_{x}v_{yy})=(\varrho_x v_y)(\varrho v_y)_x- (\varrho_x v_y)^2-
 (\varrho_x v_x)(\varrho v_y)_y.
\]
Therefore, the above expression is well-defined provided that $Tv$, $XYv$ and $Y^2v$ exist, which holds, for example, if $v$ is a component function of a weak harmonic mapping satisfying assumptions of Theorem~\ref{thm-H22}. Moreover, if $\alpha>\frac13$ then $\varrho_xv_x\in L^1_{loc}$ for $v\in HW^{1,2}_{loc}(\Om, \mu)$; also for $\alpha\geq \frac23$ we have that $\varrho_xv_x\in L^2_{loc}$.

Recall the statement of the Bochner formula, see Introduction.
\smallskip
\\
{\em {\bf Theorem~\ref{thm-Bochner}.}  Let $\alpha \in [0,1)$ and $\Om\subset G_2^{\alpha}$ be an open set. Suppose that $u:\Om\to G_2^\beta$ such that $u\in H^{2,2}_{loc}(\Om, \mu; G_2^\beta)$ is a weak harmonic map, i.e., $u$ satisfies Definition~\ref{weak-solution} in $\Om$ as well as on any ball $B(R)\Subset \Om\setminus S$ we have that
\begin{equation*}
\int_{B(R)} \frac{|\nabla_H u^2|^p}{|u^1|^{(\beta+1)p}} \ud \mu<\infty, \qquad
\int_{B(R)} |\nabla_H u^1|^{\frac{2p}{p-2}} \ud \mu<\infty \quad\hbox{for some } p>2.\tag{H2-B'} 
\end{equation*}
Then the following Bochner identity holds in the weak sense in any ball $B(R)\Subset\Om\setminus S$:  
 \begin{align}
 &L\left(\frac12 \|D_Hu\|^2\right)+\alpha|x|^{\alpha-2}\|D_Hu\|^2 \nonumber \\
 &\phantom{AAA}= \varrho \|D^2_H u\|^2+2\varrho \Bigg(\big((TXu^1)Yu^1-(TYu^1)Xu^1\big)+\frac{\big((TXu^2)Yu^2-(TYu^2)Xu^2\big)}{|u^1|^{2\beta}} \Bigg) \nonumber \\
&\phantom{AAA}+\frac{\beta \varrho}{|u^1|^{2\beta+2}}\Bigg(2u^1\Big\langle \nabla_H u^2,  \nabla_H\langle \nabla_H u^1,  \nabla_H u^2 \rangle \Big\rangle-3u^1\langle \nabla_H u^1, \nabla_H |\nabla_H u^2|^2 \rangle-2\langle \nabla_H u^1,  \nabla_H u^2 \rangle^2 \nonumber \\
&\phantom{AAAAAAAAA}+|\nabla_H u^2|^2 \Big(\beta\|D_Hu\|^2+(2+3\beta)|\nabla_H u^1|^2 \Big) \Bigg). \label{id-Bochner}
\end{align} 
}
Let us comment the regularity assumptions in the above theorem. The assumption~\eqref{ass11-thm-H22} is a slight strengthening of the assumption~\eqref{ass1-thm-H22}, as now we require also the integrability of $u^1_y$. We need such an assumption in order to make sure that the identity~\eqref{id-Bochner} is well defined in the weak sense, see  discussions following the formulas~\eqref{aux1-Bochner} and~\eqref{aux3-Bochner} in the proof below. However, since in Theorem~\ref{thm-H22} the assumption~\eqref{ass1-thm-H22} is imposed in order to infer the $H^{2,2}_{loc}$-regularity of weak harmonic maps our assumption~\eqref{ass11-thm-H22} is only moderately restrictive. 

As a consequence of Theorem~\ref{thm-Bochner} we obtain the following subharmonicity result.

\begin{cor}\label{cor-Bochner}
 Let $u\in H^{2,2}_{loc}(\Om, \mu; G_2^\beta)$ be a weak harmonic map as in Definition~\ref{weak-solution}, defined on an open set $\Om\setminus S\subset G_2^\alpha$ and for the weight $\varrho=|x|^{\alpha}$ with $0\leq\alpha<1$ with the target space $G_2^\beta$ for $0<\beta\leq \frac16$ as well as on any ball $B(R)\Subset \Om\setminus S$ the assumption~\eqref{ass11-thm-H22} holds. 
 
If
\begin{equation}\label{cor-commutators}
T(Xu^1)Yu^1-(TYu^1)Xu^1\geq 0\,\,\hbox{ and }\,\, T(Xu^2)Yu^2-(TYu^2)Xu^2\geq 0
\end{equation}
then 
\begin{equation}\label{cor-subh}
L\left(\frac12 \|D_H u\|^2\right)+\alpha|x|^{\alpha-2}\|D_Hu\|^2+\beta \frac{|x|^\alpha}{|u^1|^{2}}\|D_H u\|^4\geq 0.
\end{equation} 
and so $\|D_H u\|$ is a subsolution to~\eqref{cor-subh}. 
\end{cor}

\begin{rem}\label{rem52}
 The analysis of the last term in~\eqref{id-Bochner} and the estimate in the beginning of the proof of Corollary~\ref{cor-Bochner}, reveals that in the hypotheses~\eqref{cor-subh} of the corollary we may alternatively consider a different operator with the term $2\frac{\beta \varrho}{|u^1|^{2\beta+2}}\langle \nabla_H u^1,  \nabla_H u^2 \rangle^2$ on the left-hand side, instead of the term $\frac{\beta \varrho}{|u^1|^{2}}\|D_H u\|^4$ and it immediately holds that $2\langle \nabla_H u^1,  \nabla_H u^2 \rangle^2\leq |u^1|^{2\beta}\|D_H u\|^4$. However, the current choice allows us to study the differential operator given in a more handy form than the alternative one. Namely, as an operator depending only on $\|D_H u\|$, its derivatives and weights in $G_2^\alpha$ and $G_2^\beta$.
\end{rem}

%
%
First we present the proof of the Bochner identity, then the proof of Corollary.

\begin{proof}[Proof of Theorem~\ref{thm-Bochner}]
For the sake of simplicity of the discussion, let us first write the weak formulation of the system~\eqref{w-system-EL} in the following equivalent but more explicit form (still understood in the weak sense):
\begin{equation}\label{w2-system-EL}
\begin{cases}
\left(\frac{u_x^1}{|x|^\alpha}\right)_x +|x|^{\alpha} u^1_{yy}+\beta \frac{|\nabla_H u^2|^2 u^1}{|x|^\alpha |u^1|^{2\beta+2}}=0 \\
\left( \frac{u^2_x}{|x|^\alpha|u^1|^{2\beta}}\right)_x+|x|^\alpha \left(\frac{u^2_y}{|u^1|^{2\beta}}\right)_y=0.
\end{cases}
\end{equation}

Let us first present the plan of the proof. Our general goal is to find a second order differential operator and the corresponding identity it satisfies, which generalizes both the classical Bochner identity and the one for harmonic mappings with $\R^m$ as the target space, cf.~\cite[Section 4]{aww}. In {\bf Step 1} we will show that both equations of the system~\eqref{w2-system-EL} can be represented in terms of the above operator $L$. Then, in {\bf Step 2} by appealing to the Bochner identity (26) in~\cite{aww} we will find a formula for $L(\frac12 |\nabla_H u^1|^2)$. Similar computations performed in {\bf Step 3} will give us the formula for $L(\frac12 \frac{|\nabla_H u^2|^2}{|u^1|^{2\beta}})$ and so, in turn, in {\bf Step 4} we will obtain the formula for $L(\frac12 \|D_Hu\|^2)$.
\smallskip

\noindent {\bf Step 1.} By direct computations applied to the first equation of the system~\eqref{w2-system-EL} we get that 
\begin{align*}
&\left(\frac{u_x^1}{\varrho}\right)_x +\varrho u^1_{yy}+\beta \frac{|\nabla_H u^2|^2 u^1}{\varrho |u^1|^{2\beta+2}}=0 \\
& \varrho u_{xx}^1 -\varrho_x u^1_{x}+\varrho^3 u^1_{yy}+\beta \varrho \frac{|\nabla_H u^2|^2 u^1}{|u^1|^{2\beta+2}}=0 \\
& \varrho^2 \divGt \left(\frac{\nabla_H u^1}{\varrho}\right)+\beta \varrho \frac{|\nabla_H u^2|^2 u^1}{|u^1|^{2\beta+2}}=0 \\
&L(u^1) + \beta \varrho \frac{|\nabla_H u^2|^2 u^1}{|u^1|^{2\beta+2}}= 0.
\end{align*}

Similarly, we transform the second equation of the system~\eqref{w2-system-EL}:
\begin{align*}
&\left( \frac{u^2_x}{|x|^\alpha|u^1|^{2\beta}}\right)_x+|x|^\alpha \left(\frac{u^2_y}{|u^1|^{2\beta}}\right)_y=0 \\
&\frac{u^2_{xx} (\varrho |u^1|^{2\beta}) - u_x^2 (\varrho |u^1|^{2\beta})_x}{\varrho^2 |u^1|^{4\beta}}+\varrho^3 \frac{u^2_{yy} |u^1|^{2\beta} -u_y^2 (|u^1|^{2\beta})_y}{\varrho^2|u^1|^{4\beta}}=0 \\
& (\varrho |u^1|^{2\beta}) u^2_{xx} - (\varrho |u^1|^{2\beta})_x u_x^2+\varrho^3 |u^1|^{2\beta} u^2_{yy}  -\varrho^3 (|u^1|^{2\beta})_y u_y^2 =0\\
& |u^1|^{2\beta}\left(\varrho u^2_{xx} - \varrho_x u_x^2+\varrho^3 u^2_{yy}\right) - \varrho (|u^1|^{2\beta})_x u_x^2-\varrho^3 (|u^1|^{2\beta})_y u_y^2 =0\\
& |u^1|^{2\beta} L(u^2) - \varrho \langle \nabla_H |u^1|^{2\beta},  \nabla_H u^2 \rangle =0.
\end{align*}
In consequence, the system of equations~\eqref{w2-system-EL} equivalently reads:
\begin{equation}\label{L-w2-system-EL}
\begin{cases}
L(u^1) + \beta \varrho \frac{|\nabla_H u^2|^2 u^1}{|u^1|^{2\beta+2}}= 0\\
|u^1|^{2\beta} L(u^2) - \varrho \langle \nabla_H |u^1|^{2\beta},  \nabla_H u^2 \rangle =0.
\end{cases}
\end{equation}
\noindent {\bf Step 2.} Recall the discussion on pg. 34 in~\cite{aww} following the geometric interpretation of the commutators $T$ in~\cite[Definition 4.1]{aww} and the formula (26) in~\cite{aww} for the operator $L$. There, by combining (26) with the definition of the commutator $T$ we obtain the following expression for $L$:  
\begin{align}
L(\frac12 |\nabla_H u^1|^2)&=\varrho \|\nabla_H^2 u^1\|_{H}^2+\langle \nabla_H u^1, \nabla_H L(u^1) \rangle+\varrho \left(\frac{\varrho_{x}}{\varrho}\right)_{x}|\nabla_H u^1|^2-\frac{\varrho_{x}}{\varrho}u^1_x L(u^1) \nonumber \\
&+2\varrho \big((TXu^1)Yu^1-(TYu^1)Xu^1\big). \label{L-Bochner-1}
\end{align}
Therefore, we need to compute $\nabla_H L(u^1)=-\beta \nabla_H\left( \varrho \frac{|\nabla_H u^2|^2 u^1}{|u^1|^{2\beta+2}}\right)$. By direct computations, we find the components of the horizontal gradient $\nabla_H L(u^1)$:
\begin{align}
\left( \varrho \frac{|\nabla_H u^2|^2 u^1}{|u^1|^{2\beta+2}}\right)_x &=
\varrho_x \frac{|\nabla_H u^2|^2 u^1}{|u^1|^{2\beta+2}}+\varrho\frac{(u^1|\nabla_H u^2|^2)_x |u^1|^{2\beta+2}-u^1|\nabla_H u^2|^2 (|u^1|^{2\beta+2})_x}{|u^1|^{4\beta+4}} \nonumber \\
&=\varrho_x \frac{|\nabla_H u^2|^2 u^1}{|u^1|^{2\beta+2}}+\frac{\varrho}{|u^1|^{2\beta+2}}\left[u^1_x |\nabla_H u^2|^2+u^1 (|\nabla_H u^2|^2)_x- (2\beta+2) u^1_x|\nabla_H u^2|^2\right] \nonumber \\
&=\frac{1}{|u^1|^{2\beta+2}}\left[\varrho_x u^1 |\nabla_H u^2|^2- (1+2\beta)\varrho u^1_x|\nabla_H u^2|^2 +\varrho u^1 (|\nabla_H u^2|^2)_x\right]. \label{ident111-Bochner}
\end{align}
Similarly, we obtain the second component function of $\nabla_H L(u^1)$:
\begin{equation}
 \left( \varrho \frac{|\nabla_H u^2|^2 u^1}{|u^1|^{2\beta+2}}\right)_y = \frac{1}{|u^1|^{2\beta+2}}\left[\varrho u^1 (|\nabla_H u^2|^2)_y- (1+2\beta)\varrho u^1_y|\nabla_H u^2|^2\right].\label{ident112-Bochner}
\end{equation}
Notice that due to the integrability assumption~\eqref{ass-lem-HW0} in the definition of the weak harmonic mappings, both identities~\eqref{ident111-Bochner} and~\eqref{ident112-Bochner} are well-defined in the weak sense and hence so is $\nabla_H L(u^1)$.
Hence, 
\begin{align}
 \langle \nabla_H u^1, \nabla_H L(u^1) \rangle &= |u^1|^{-2\beta-2}\Big[\beta(1+2\beta)\varrho (u^1_x)^2 |\nabla_H u^2|^2-\beta \varrho_x u^1 u^1_x|\nabla_H u^2|^2 - \beta \varrho u^1 u^1_x (|\nabla_H u^2|^2)_x \nonumber \\
& \phantom{AAAAAAAa}-\beta \varrho^3u^1u^1_y (|\nabla_H u^2|^2)_y +\beta(1+2\beta)\varrho^3 (u^1_y)^2|\nabla_H u^2|^2 \Big]. \label{aux1-Bochner}
\end{align}
For the identity~\eqref{aux1-Bochner} to be well-defined in the weak sense we also need $|\int_{B_r}  \langle \nabla_H u^1, \nabla_H L(u^1) \rangle \phi\, \ud \mu|<\infty$ for any $\phi \in C_0^\infty(B_r)$ and $B_r\Subset \Om$. This, upon estimating the right-hand side of~\eqref{aux1-Bochner} and taking into consideration the regularity of weak harmonic mappings in Definition~\ref{weak-solution}, reduces to estimating the following integral
\[
 \int_{B_r} \frac{|\nabla_H u^1|^2}{|u^1|}\frac{|\nabla_H u^2|^2}{|u^1|^{2\beta+1}}\ud \mu.
\]
By applying the H\"older estimate and the assumption~\eqref{ass11-thm-H22} we show the finiteness of the above expression (and any $p>2$ in~\eqref{ass11-thm-H22} suffices). We substitute computations in~\eqref{aux1-Bochner} into~\eqref{L-Bochner-1}, simplify the expression and obtain the following identity:
\begin{align*}
& L\Big(\frac12 |\nabla_H u^1|^2\Big)+\alpha |x|^{\alpha-2}|\nabla_H u^1|^2 \\
&\phantom{AAA}=\varrho \|\nabla_H^2 u^1\|_{H}^2+2\varrho \big((TXu^1)Yu^1-(TYu^1)Xu^1\big) \nonumber \\
&\phantom{AAAA}+|u^1|^{-2\beta-2}\Big[\beta(1+2\beta)\varrho (u^1_x)^2 |\nabla_H u^2|^2-\beta \varrho_x u^1 u^1_x|\nabla_H u^2|^2 - \beta \varrho u^1 u^1_x (|\nabla_H u^2|^2)_x \nonumber \\
&\phantom{AAAAAAAa}-\beta \varrho^3u^1u^1_y (|\nabla_H u^2|^2)_y +\beta(1+2\beta)\varrho^3 (u^1_y)^2|\nabla_H u^2|^2 \Big] +\frac{\varrho_{x}}{\varrho}u^1_x \left( \beta \varrho \frac{|\nabla_H u^2|^2 u^1}{|u^1|^{2\beta+2}}\right) \\
&\phantom{AAA}=\varrho \|\nabla_H^2 u^1\|_{H}^2+2\varrho \big((TXu^1)Yu^1-(TYu^1)Xu^1\big) \\
&\phantom{AAAA}+|u^1|^{-2\beta-2}\Big[\beta(1+2\beta)\varrho (u^1_x)^2 |\nabla_H u^2|^2 +\beta(1+2\beta)\varrho^3 (u^1_y)^2|\nabla_H u^2|^2  \\
&\phantom{AAAAAAAAAAa} - \beta \varrho u^1 u^1_x (|\nabla_H u^2|^2)_x -\beta \varrho^3u^1u^1_y (|\nabla_H u^2|^2)_y\Big].
\end{align*}
Upon grouping together the appropriate terms on the right-hand side and rearranging the resulting expression, we obtain the following identity
\begin{align}
& L\Big(\frac12 |\nabla_H u^1|^2\Big)+\beta \varrho \frac{u^1}{|u^1|^{2\beta+2}}\langle \nabla_H u^1, \nabla_H |\nabla_H u^2|^2\rangle +\alpha |x|^{\alpha-2}|\nabla_H u^1|^2 \nonumber \\ 
 &\phantom{AAA}=\varrho \|\nabla_H^2 u^1\|_{H}^2+ \frac{\beta(1+2\beta)}{|u^1|^{2\beta+2}}\varrho |\nabla_H u^1|^2 |\nabla_H u^2|^2+ 2\varrho \big((TXu^1)Yu^1-(TYu^1)Xu^1\big). \label{aux2-Bochner}
\end{align}
Observe that since the first two terms on the right-hand side are nonnegative, it is the sign of the remaining commutator expression that determines when $|\nabla_H u^1|^2$ is a subsolution to the second order differential operator stated on the left-hand side of~\eqref{aux2-Bochner}.  
\smallskip

\noindent {\bf Step 3.} The computations for $L(\frac12 \frac{|\nabla_H u^2|^2}{|u^1|^{2\beta}})$ turns out to be more involved than the corresponding ones for $u^1$, but nevertheless they follow the similar way as in Step 2. 

Denote by $v:=\frac12 |\nabla_H u^2|^2$. Then,
\begin{align*}
 &L\left(\frac{v}{|u^1|^{2\beta}}\right) \\
 &=\varrho \left(\frac{v}{|u^1|^{2\beta}}\right)_{xx} -\varrho_x \left(\frac{v}{|u^1|^{2\beta}}\right)_{x}+\varrho^3\left(\frac{v}{|u^1|^{2\beta}}\right)_{yy} \\
 &=\varrho \left(\frac{v_x}{|u^1|^{2\beta}}-2\beta\frac{v u^1u^1_x}{|u^1|^{2\beta+2}}\right)_{x}-\varrho_x \frac{v_x}{|u^1|^{2\beta}}+2\beta \varrho_x \frac{vu^1u^1_x}{|u^1|^{2\beta+2}}+\varrho^3\left(\frac{v_y}{|u^1|^{2\beta}}-2\beta\frac{v u^1u^1_y}{|u^1|^{2\beta+2}}\right)_{y} \\
 &=\frac{1}{|u^1|^{2\beta}}L(v)-2\beta\varrho\frac{u^1}{|u^1|^{2\beta+2}}\langle\nabla_H u^1, \nabla_H |\nabla_H u^2|^2 \rangle \\
 &\phantom{AA} -2\beta v\left(\varrho\left(\frac{u^1u^1_x}{|u^1|^{2\beta+2}}\right)_x+\varrho^3\left(\frac{u^1u^1_y}{|u^1|^{2\beta+2}}\right)_y - \varrho_x\frac{u^1u^1_x}{|u^1|^{2\beta+2}} \right)\\
 &=\frac{1}{|u^1|^{2\beta}}L(v)-2\beta\varrho\frac{u^1}{|u^1|^{2\beta+2}}\langle\nabla_H u^1, \nabla_H |\nabla_H u^2|^2 \rangle \\
 &\phantom{A}\, -2\beta \frac{v}{|u^1|^{2\beta+2}} \Big(\varrho (u^1_x)^2 + \varrho u^1u^1_{xx} -(2\beta+2)\varrho (u^1_x)^2 +\varrho^3 (u^1_y)^2 + \varrho^3 u^1u^1_{yy} -(2\beta+2)\varrho^3 (u^1_y)^2-\varrho_xu^1u^1_x \Big)
\end{align*}
Since the last term above equals
\[
 -2\beta \frac{v}{|u^1|^{2\beta+2}} \big(u^1 L(u^1) -(1+2\beta)\varrho |\nabla_H u^1|^2\big),
\]
we arrive at the following formula:
\begin{align}
 L\left(\frac12 \frac{|\nabla_H u^2|^2}{|u^1|^{2\beta}}\right)&=\frac{1}{|u^1|^{2\beta}}{\color{blue} L\left(\frac12 |\nabla_H u^2|^2\right)}-2\beta\varrho\frac{u^1}{|u^1|^{2\beta+2}}\langle\nabla_H u^1, \nabla_H |\nabla_H u^2|^2 \rangle \nonumber \\
& -2\beta \frac{u^1}{|u^1|^{2\beta+2}}\left(\frac12 |\nabla_H u^2|^2\right) L(u^1)+2\beta(1+2\beta)\varrho \frac{1}{|u^1|^{2\beta+2}}\left(\frac12 |\nabla_H u^2|^2\right) |\nabla_H u^1|^2. \label{aux3-Bochner}
\end{align}
Before we continue with the analysis of~\eqref{aux3-Bochner} let us comment on the conditions under which this formula holds in the weak sense. Namely, we need to establish that $\int_{B_r}  L\left(\frac12 \frac{|\nabla_H u^2|^2}{|u^1|^{2\beta}}\right) \phi<\infty$ for any ball $B_r\Subset \Om$ and any $\phi\in C_0^\infty$. Upon applying integration by parts twice in the definition of $L$ we arrive at three integrals involving the expression as in the assumption~\eqref{ass-lem-HW0} as well as the partial derivatives of $\phi$ and $\varrho$. Then, by H\"older inequalities and the assumption~\eqref{ass11-thm-H22} we deduce that these integrals are bounded. We omit tedious technical details of this estimate.

In order to handle the distinguished blue expression in~\eqref{aux3-Bochner}, we apply the formula~\eqref{L-Bochner-1} for the operator $L$ to function $v$ and get that
\begin{align}
L(\frac12 |\nabla_H u^2|^2)&=\varrho \|\nabla_H^2 u^2\|_{H}^2+{\color{blue} \langle \nabla_H u^2, \nabla_H L(u^2) \rangle-\frac{\varrho_{x}}{\varrho}u^2_x L(u^2)}+\varrho \left(\frac{\varrho_{x}}{\varrho}\right)_{x}|\nabla_H u^2|^2 \nonumber \\
&+2\varrho \big((TXu^2)Yu^2-(TYu^2)Xu^2\big). \label{L-Bochner-2}
\end{align}
Our next goal is to compute $\nabla_H L(u^2)$. By the second equation of the system~\eqref{L-w2-system-EL}, we find  that
\begin{equation*}
 L(u^2)=\frac{\varrho}{|u^1|^{2\beta}} \langle \nabla_H |u^1|^{2\beta},  \nabla_H u^2 \rangle
 =2\beta\varrho\frac{u^1}{|u^1|^{2}} \langle \nabla_H u^1,  \nabla_H u^2 \rangle
\end{equation*}
and, therefore,
\begin{align*}
\frac{1}{2\beta}( L(u^2) )_x&=\varrho_x \frac{u^1}{|u^1|^{2}} \langle \nabla_H u^1,  \nabla_H u^2 \rangle
-\varrho \frac{u_x^1}{|u^1|^{2}} \langle \nabla_H u^1,  \nabla_H u^2 \rangle + \varrho \frac{u^1}{|u^1|^{2}} \langle \nabla_H u^1,  \nabla_H u^2 \rangle_x\\
\frac{1}{2\beta}( L(u^2) )_y&=-\varrho \frac{u_y^1}{|u^1|^{2}} \langle \nabla_H u^1,  \nabla_H u^2 \rangle + \varrho \frac{u^1}{|u^1|^{2}} \langle \nabla_H u^1,  \nabla_H u^2 \rangle_y.
\end{align*}
This leads to the following observation for the distinguished terms in~\eqref{L-Bochner-2}: 
\begin{equation}
 \langle \nabla_H u^2, \nabla_H L(u^2) \rangle-\frac{\varrho_{x}}{\varrho}u^2_x L(u^2)=2\beta\frac{\varrho}{|u^1|^{2}}\Big( u^1\Big\langle \nabla_H u^2,  \nabla_H\langle \nabla_H u^1,  \nabla_H u^2 \rangle \Big\rangle-\langle \nabla_H u^1,  \nabla_H u^2 \rangle^2  \Big). \label{aux4-Bochner}
\end{equation}
Finally, we apply~\eqref{aux4-Bochner} in~\eqref{L-Bochner-2} and then substitute the resulting expression in~\eqref{aux3-Bochner}. In consequence we get the following formula  
\begin{align}
  &L\Big(\frac12 \frac{|\nabla_H u^2|^2}{|u^1|^{2\beta}}\Big) \nonumber \\
  &\phantom{AAA}=\frac{1}{|u^1|^{2\beta}} 
   {\color{blue} \Big[ \varrho \|\nabla_H^2 u^2\|_{H}^2+\varrho \left(\frac{\varrho_{x}}{\varrho}\right)_{x}|\nabla_H u^2|^2+2\varrho \big((TXu^2)Yu^2-(TYu^2)Xu^2\big)} \nonumber\\
  &\phantom{AAAAAAAAA} {\color{blue} +2\beta\frac{\varrho}{|u^1|^{2}}\Big( u^1\Big\langle \nabla_H u^2,  \nabla_H\langle \nabla_H u^1,  \nabla_H u^2 \rangle \Big\rangle-\langle \nabla_H u^1,  \nabla_H u^2 \rangle^2  \Big)    \Big] }\nonumber\\
&\phantom{AAA} +\frac{\varrho}{|u^1|^{2\beta+2}} \Big(\beta^2 \frac{|\nabla_H u^2|^4}{|u^1|^{2\beta}}+\beta(1+2\beta) |\nabla_H u^1|^2|\nabla_H u^2|^2-2\beta u^1 \langle\nabla_H u^1, \nabla_H |\nabla_H u^2|^2 \rangle \Big). \label{L-Bochner-full-2}
\end{align}
Here we also use the first equation of the system~\eqref{L-w2-system-EL} to substitute $L(u^1)=- \beta \varrho \frac{|\nabla_H u^2|^2 u^1}{|u^1|^{2\beta+2}}$.
\smallskip

\noindent {\bf Step 4.} We are in a position to complete the proof of the Bochner identity for the Grushin-harmonic mappings. Upon combining the identities~\eqref{aux2-Bochner} for $u^1$ and~\eqref{L-Bochner-full-2} for $u^2$, respectively, we arrive at the following formula: 
\begin{align}
 &L\Big(\frac12 \|D_Hu\|^2\Big)+\alpha|x|^{\alpha-2}\|D_Hu\|^2+3\beta\varrho \frac{u^1}{|u^1|^{2\beta+2}}
 \langle \nabla_H u^1, \nabla_H |\nabla_H u^2|^2 \rangle \nonumber \\
 &\phantom{AAA}= \varrho\left(\|\nabla_H^2 u^1\|_{H}^2+\frac{\|\nabla_H^2 u^2\|_{H}^2}{|u^1|^{2\beta}} \right) \nonumber \\
&\phantom{AAA}+2\varrho \Bigg(\big((TXu^1)Yu^1-(TYu^1)Xu^1\big)+\frac{\big((TXu^2)Yu^2-(TYu^2)Xu^2\big)}{|u^1|^{2\beta}} \Bigg) \nonumber \\
&\phantom{AAA}+2\beta\frac{\varrho}{|u^1|^{2\beta+2}}\Big( u^1\Big\langle \nabla_H u^2,  \nabla_H\langle \nabla_H u^1,  \nabla_H u^2 \rangle \Big\rangle-\langle \nabla_H u^1,  \nabla_H u^2 \rangle^2  \Big) \nonumber \\
&\phantom{AAA}+\beta \frac{\varrho}{|u^1|^{2\beta+2}}|\nabla_H u^2|^2 \Big(\beta\|D_Hu\|^2+(2+3\beta)|\nabla_H u^1|^2 \Big). 
\end{align}
Upon denoting the Hessian term by $\|D^2_H u\|^2:=\|\nabla_H^2 u^1\|_{H}^2+\frac{\|\nabla_H^2 u^2\|_{H}^2}{|u^1|^{2\beta}}$ and rearranging the last two terms together with the corresponding term on the left-hand side, we complete the proof of the Bochner identity for the Grushin-harmonic mappings.
\end{proof}

Before we prove Corollary~\ref{cor-Bochner} let us comment that related is the following Caccioppoli-type inequality for the differential operator in~\eqref{cor-subh}. It is convenient to formulate such an inequality  for the weak form of $L$, see~\eqref{def-L-bochner}. Namely, we consider functions $v\in H^{1,2}_{loc}(\Om, \mu, \R)$ for an open set $\Om\subset G_2$ such that the following inequality holds in the weak sense 
 \[
  \varrho^2 \divGt \left(\frac{\nabla_H v}{\varrho}\right)+\alpha |x|^{\alpha-2}2v+\beta \frac{|x|^\alpha}{|u^1|^{2}}4v^2\geq 0,
 \]
 where $u^1$ is a given first coordinate function of the Grushin-harmonic map $u:\Om \to G_2^\beta$. 
 
\begin{lem}\label{lem-Cacc}
 Let $\Om\subset G_2\setminus S$ be an open set and $v\in HW^{1,2}_{loc}(\Om,\mu, \R)$ satisfy inequality~\eqref{cor-subh} in the weak sense, i.e.
 \begin{equation}\label{cor-Cacc-est}
  -\int_{\Om} \left \langle \nabla_H(\varrho^2 \phi)\,,\,\frac{\nabla_H v}{\varrho} \right \rangle+2\alpha \int_{\Om}|x|^{\alpha-2}v \phi + 4\beta \int_{\Om} \frac{|x|^\alpha}{|u^1|^{2}} v^2 \phi \geq 0,\quad \hbox{for all }\phi \in HW^{1,2}_{0}(\Om).
 \end{equation}
 Then, it holds for balls $B(r)\subset B(R)\Subset \Om$ that
 \begin{equation}\label{cor-Cacc-est2}
 \int_{B(r)} |\nabla_H v|^2\varrho \leq 18 \int_{B(R)} \Big(\,\frac{4\alpha}{|x|^2}+ \frac{C^2}{(R-r)^2}\,\Big)v^2\varrho+2\beta \frac{v^3}{|u^1|^{2}}\varrho.
 \end{equation}
\end{lem}
We remark that the integrals on the right hand side above are finite for balls satisfying $B(R)\cap S=\emptyset$ such that $u^1(B(R))\cap S=\emptyset$ and under the assumption 
\begin{equation}\label{ass2-cor-Cacc-est}
\int_{B(R)}\frac{v^3}{|u^1|^{2}}\ud \mu<\infty.
\end{equation}
Otherwise one has to ensure that $\int_{B(R)}|x|^{\alpha-2}v^2<\infty$. Moreover, the additional integrability condition~\eqref{ass2-cor-Cacc-est} makes the estimate less handy and, therefore, in what follows we will not exploit it further. We omit the proof of the estimate~\eqref{cor-Cacc-est2}, as it relies on the standard techniques based on choosing the appropriate test function, namely $\phi:=\eta^2 v$ for $\eta\in HW^{1,2}_{0}(\Om)$ and such that $0\leq \eta\leq 1$ in $\Om$.
\begin{proof}[Proof of Corollary~\ref{cor-Bochner}]
By the assumption~\eqref{cor-commutators}, in order to show the subharmonicity result in~\eqref{cor-subh} we first need  to find the appropriate lower bound for the last term in the Bochner formula~\eqref{id-Bochner} such that added to $\varrho\|D_H^2u\|^2$ is nonnegative.

First, observe that 
\begin{align*}
& \nabla_H^2 u^1\cdot \nabla_Hu^2=(u^2_xu^1_{xx}+\varrho_x\varrho u^1_yu^2_y+\varrho^2u^1_{yx}u^2_y, \varrho u^2_xu^1_{xy}+\varrho^3u^2_yu^1_{yy}),\\
& \nabla_H^2 u^2\cdot \nabla_Hu^1=(u^1_xu^2_{xx}+\varrho_x\varrho u^1_yu^2_y+\varrho^2u^1_yu^2_{yx}, \varrho u^1_xu^2_{xy}+\varrho^3u^1_yu^2_{yy}),
\end{align*}
and so, by the direct computations we find the following representation of the expression below in terms of the horizontal hessian matrices of the component functions of the map $u$:
\begin{align*}
& \nabla_H \langle \nabla_H u^1,\nabla_H u^2 \rangle = \nabla_H^2 u^1\cdot \nabla_Hu^2+\nabla_H^2 u^2\cdot \nabla_Hu^1 \\
& \nabla_H |\nabla_H u^2|^2=\nabla_H \langle \nabla_H u^2,\nabla_H u^2 \rangle = 2\nabla_H^2 u^2\cdot \nabla_Hu^2.
\end{align*}
Therefore, by direct estimations and the Cauchy-Schwarz-Newton inequality applied to the last term in the Bochner formula~\eqref{id-Bochner} we find 
\begin{align*} 
&2u^1\Big\langle \nabla_H u^2,  \nabla_H\langle \nabla_H u^1,  \nabla_H u^2 \rangle \Big\rangle-3u^1\langle \nabla_H u^1, \nabla_H |\nabla_H u^2|^2 \rangle +|\nabla_H u^2|^2 \Big(\beta\|D_Hu\|^2+(2+3\beta)|\nabla_H u^1|^2 \Big) \\
&\,\, \geq -2|u^1||\nabla_H u^2| \|\nabla_H^2 u^1\| |\nabla_H u^2|-2|u^1| |\nabla_H u^2| \|\nabla_H^2 u^2\| |\nabla_H u^1|- 3|u^1|\, (2|\nabla_H u^1| \|\nabla_H^2 u^2\| |\nabla_Hu^2|) \\
&\phantom{AA} +\beta |\nabla_H u^2|^2 \|D_Hu\|^2+(2+3\beta)|\nabla_H
 u^1|^2|\nabla_H u^2|^2 \\
&\,\, \geq -2\Big( \frac{|u^1|^{1+\beta}}{\sqrt{\beta}}\|\nabla_H^2 u^1\|\Big) \Big(\sqrt{\beta}\frac{ |\nabla_H u^2|^2}{|u^1|^\beta}\Big)- 2\Big( \frac{4|u^1|}{\sqrt{2+4\beta}} \|\nabla_H^2 u^2\| \Big) \Big(\sqrt{2+4\beta}|\nabla_H u^1| |\nabla_H u^2|\Big) \\
&\,\,\phantom{AA} +\beta \frac{|\nabla_H u^2|^4}{|u^1|^{2\beta}} +(2+4\beta)|\nabla_H  u^1|^2|\nabla_H u^2|^2 \\
&\,\, \geq -\frac{|u^1|^{2(1+\beta)}}{\beta}\|\nabla_H^2 u^1\|^2- \frac{16}{2+4\beta}|u^1|^{2(1+\beta)} \frac{\|\nabla_H^2 u^2\|^2}{|u^1|^{2\beta}}\\
&\,\, \geq -\max\left\{\frac{1}{\beta}, \frac{8}{1+2\beta}\right\}|u^1|^{2(1+\beta)} \|D_H^2 u\|^2.
\end{align*}
Hence, by adding the Hessian term in the Bochner formula~\eqref{id-Bochner} we get the following lower bound for the last term in~\eqref{id-Bochner}:
\[
 \varrho \|D_H^2 u\|^2 -\left(\frac{\beta \varrho}{|u^1|^{2\beta+2}}\right)\max\left\{\frac{1}{\beta}, \frac{8}{1+2\beta}\right\}|u^1|^{2(1+\beta)} \|D_H^2 u\|^2\geq \left(1-\max\left\{1, \frac{8\beta}{1+2\beta}\right\}\right)\varrho \|D_H^2 u\|^2 \geq 0,
\]
if and only if $\beta \leq \frac16$. This together with Remark~\ref{rem52} completes the proof of the subharmonicity~\eqref{cor-subh}.
\end{proof}

\section{Weak Harnack inequalities for component functions of harmonic mappings}\label{sect-w-Harnack}

The goal of this section is to establish supremum estimates (weak Harnack estimates) on Grushin balls for the component functions of (Grushin) strongly harmonic mappings. It turns out that the key assumption is the condition~\eqref{orth-cond} which allows us to control the angle between the gradients of these component functions, see Lemma~\ref{lem-Cac-str} and the discussion following its statement. Moreover, let us point out that our weak Harnack estimates allow the balls to intersect the singular set $S$. 

Recall that by $v^{\pm}$ we denote the positive, respectively, negative part of the function $v$.

\begin{theorem}[Weak Harnack inequality for $u^2$]\label{weak-harnack-u2}
    Let $\alpha \in (0,1)$, $\beta \ge 0$, $\Om \subset G_2^{\alpha}$ be an open connected set and let $u:\Om \to G_2^{\beta}$ be a strongly harmonic map satisfying~\eqref{orth-cond} with a function $K \in L^2_{\textrm{loc}}(\Om,\mu)$.   

\noindent Set 
\[
{\mathcal P}_{\alpha,\beta}:=(2+\alpha)(1+\beta).
\]

Then for each $p>1$ and any Grushin ball $B_r = B(x_0, r) \subset \Om$ there exists a constant $C = C(\alpha,\beta,p)>0$ such that for any $\lambda \in (0,1)$ the following weak Harnack estimate holds:
\begin{equation}\label{weak harnack for u^2, p ineq}
    \sup_{\lambda B_r}|u^2_{\pm}| \le \frac{C}{1-\lambda}\Big(\vint_{B_r}|u^1|^{p{\mathcal P}_{\alpha,\beta}} \ud \mu\Big)^{\frac{\beta}{p{\mathcal P}_{\alpha,\beta}}}\Big(\vint_{B_r}\left(\frac{|u^2_{\pm}|}{|u^1|^{\beta}}\Big)^{p{\mathcal P}_{\alpha,\beta}} \ud \mu\right)^{\frac{1}{p{\mathcal P}_{\alpha,\beta}}}.
\end{equation}

Moreover, for every $p>1$ it holds that
\begin{equation}\label{weak harnack for u^2, nice ineq}
    \sup_{\lambda B_r}|u^2_{\pm}| \le \frac{C}{1-\lambda}\left(\vint_{B_r} \| u\|_{G_2^{\beta}}^{(1+\beta)p(2+\alpha)} \ud \mu \right)^{\frac{1}{p(2+\alpha)}}.
\end{equation}

\end{theorem}

\begin{remark}
 Let us observe that already in the case $\beta=0$ we obtain a new type of the weak Harnack estimates for the uncoupled harmonic mappings from the Grushin plane to the Euclidean plane, studied in our previous work~\cite{aww}.
\end{remark}

In the proof of the weak Harnack inequality, we appeal to the following variant of the Sobolev embedding theorem in the Grushin setting whose statement and the proof is based on the corresponding result in~\cite{fs87}.
\begin{obs}[cf. Theorem 4.6 in~\cite{fs87}]\label{Grushin sobolev embedding thm}
    Let $\alpha \in (0,1)$ and let $\Omega \subset G_2^{\alpha}$ be open. Let $v \in HW^{1,p}_0(\Om,\mu)$, $1 < p \le 2$. Then for each Grushin ball $B(x_0,r) \subset \Om$ it holds that
    \begin{equation*}
        \Big(\vint_{B(x_0,r)} |v|^{pk} \ud \mu \Big)^{\frac{1}{pk}} \le C(\alpha,p)r\Big(\vint_{B(x_0,r)} |\nabla_H v|^p \ud \mu\Big)^{\frac{1}{p}}
    \end{equation*}
for $k \in [1,\frac{2+\alpha}{2+\alpha-p}]$.
\end{obs}
\begin{proof} The proof follows the steps of the one of~\cite[Theorem 4.6]{fs87} and is based on adjusting it to our setting.  Fix a Grushin ball $B(x_0,r) \subset \Om$ and let $v\in HW^{1,2}_0(\Om,\mu)$. By Proposition~\ref{prop-dens} there exists a sequence $(v_n)$ such that each $v_n \in C^{\infty}_0 \cap HW^{1,2}(\Om,\mu)$ and $v_n \to v$ in $HW^{1,2}(\Om,\mu)$. Let $G_1=1$, $G_2 = 1+\alpha$ denote exponents as in the proof~\cite[Theorem 4.6]{fs87}. Similarly, to~\cite{fs87} the admissible range for the exponent $k$ is $[1,(1-\frac{p}{q(2+\alpha)})^{-1}]$, where $q>1$ is the parameter from the $q$-Muckenhoupt condition, see~\cite[Lemma 1.1]{aww} and~\cite[Theorem 4.6]{fs87}. For the detailed discussion about the admissible range of $k$ see the statements of Lemmas 4.3 and 4.4 in \cite{fs87}, as well as, the condition on top of page 547 in \cite{fs87} which reads:
    \begin{equation*}
      p-q\Big(1-\frac{1}{k}\Big)(G_1+G_2)>0,
     \end{equation*}
    and we recover the admissible range for $k$. Hence, Theorem 4.6 in \cite{fs87} yields for each $n \in \mathbb{N}$ and $k \in [1,(1-\frac{p}{q(2+\alpha)})^{-1}]$ with $q>1$:
    \begin{equation}\label{Grushin sobolev embedding proof 1}
        \Big(\vint_{B(x_0,r)} |v_n|^{pk} \ud \mu \Big)^{\frac{1}{pk}} \le C(\alpha,p,q)r\Big(\vint_{B(x_0,r)} |\nabla_H v_n|^p \ud \mu\Big)^{\frac{1}{p}},
    \end{equation}
 where we used the fact that functions $v_n$ are Lipschitz in the Carnot--Carath\'eodory metric, see Remark 2.2 and Observation 2.1 in \cite{aww}. Furthermore, if $q=1$, then the careful analysis of the proof of the Muckenhoupt property in~ \cite{aww}, see formula  (85) therein, gives us the $A_1$ property of $|x|^{-\alpha}$. Hence, the constant $C= C(\alpha,p,q)$ in \eqref{Grushin sobolev embedding proof 1} actually depends only on $\alpha$ and $p$, i.e., $C=C(\alpha,p)$. Thus, by the continuity of the $L^p$ and $L^{pk}$ norms, the assertion of Observation~\ref{Grushin sobolev embedding thm} follows for each $v_n$ and $k \in [1,\frac{2+\alpha}{2+\alpha-p}]$. We complete the proof by passing to the limit with $n\to \infty$.
\end{proof}
%

%
%
\begin{proof}[Proof of Theorem~\ref{weak-harnack-u2}]
We present the proof only for the positive part of $u^2$, i.e., for $u^2_+$, as the proof of the $u^2_-$ case is analogous.

Before we proceed with the proof, we note that throughout the reasoning we will use multiple constants depending on $\alpha, \psi,\gamma$, where $\psi$ and $\gamma$ are the numerical constants, that are denoted by $C(\alpha,\psi,\gamma)$. Their values may vary from line to line, but are always positive and finite. 

Our proof is based on the Moser iteration technique and follows closely the steps of the proof of Theorem 3.34 in~ \cite{hkm}. Let $u=(u^1,u^2)$ be a strongly Grushin harmonic map satisfying \eqref{orth-cond} with $K \in L^2_{\textrm{loc}}(\Om,\mu)$. Set:
\begin{equation*}
  r_l := \lambda + (1-\lambda)2^{-l}, \quad l=0,1,2,\ldots.
\end{equation*}
Let further $B_l:=B(x_0, r r_l)$ be a Grushin ball and $\eta_l \in C_0^{\infty}(B_l)$ be a test function such that
\[
0 \le \eta_l \le 1,\quad {\rm{supp}}(\eta_l) \subset B_{l}, \quad \eta_l \equiv 1 \hbox{ on } B_{l+1},\hbox { and }\,|\nabla_H \eta_l| \le 4\frac{2^l}{r(1-\lambda)},
\]
see, for instance, Observation 2.1 in \cite{aww} for a proof of existence of such a test function. Fix $q \ge 0$ and let
\begin{equation}\label{moser iteration, function for iteration step}
    \omega_l:=(u^2_{+})^{1+\frac{q}{2}}\eta_l \in HW^{1,2}_0(\Om,\mu).
\end{equation}
Then we have that 
\begin{equation*}
 \nabla_H \omega_l = \left(1+\frac{q}{2}\right)(u^2_+)^{\frac{q}{2}}(\nabla_H u^2_+) \eta_l + (u^2_{+})^{1+\frac{q}{2}} (\nabla_H \eta_l).
\end{equation*}
Hence, by the Minkowski inequality and by Lemma~\ref{lem-Cac-str} we obtain the following estimate:
\begin{equation}\label{weak harnack for u^2 1}
\begin{aligned}
    \left(\int_{B_l}\frac{|\nabla_H \omega_l|^2}{|u^1|^{2\beta}} \ud \mu\right)^{\frac{1}{2}} &\le \frac{2+q}{2}\left(\int_{B_l}(u^2_+)^q\frac{|\nabla_H u^2_+|^2}{|u^1|^{2\beta}} \eta_l^2 \ud \mu\right)^{\frac{1}{2}} + \left(\int_{B_l}(u^2_+)^{q}\frac{(u^2_+)^2}{|u^1|^{2\beta}}|\nabla_H \eta_l|^2 \ud \mu\right)^{\frac{1}{2}}\\
    &\le (3+q)\left(\int_{B_l}(u^2_+)^{q}\frac{(u^2_+)^2}{|u^1|^{2\beta}}|\nabla_H \eta_l|^2 \ud \mu\right)^{\frac{1}{2}}.
\end{aligned}
\end{equation}
Let $\psi \in \left(1, \frac{4+\alpha}{2+\alpha}\right)$ be a constant. By the Jensen inequality together with the $2$-Ahlfors regularity of the measure $\mu$ and Observation~\ref{Grushin sobolev embedding thm} applied with $p=\frac{2}{\psi}$ and $k=\frac{2+\alpha}{2+\alpha-\frac{2}{\psi}}$, we obtain:
\begin{equation}\label{weak harnack for u^2 2}
\begin{aligned}
   & \Big(\vint_{B_{l+1}}(u^2_+)^{q\frac{(2+\alpha)}{\psi(2+\alpha)-2}}\ud \mu\Big)^{\frac{q+2}{q}\frac{\psi (2+\alpha)-2}{2(\alpha+2)}} \\
     &\phantom{AAA}\le \Big(\vint_{B_{l+1}}(u^2_+)^{(2+q)\frac{(2+\alpha)}{\psi(2+\alpha)-2}}\, \ud \mu\Big)^{\frac{\psi (2+\alpha)-2}{2(\alpha+2)}}\\
    &\phantom{AAA}\le C(l,\alpha,\psi) \Big(\vint_{B_l}(\omega_l)^{\frac{2(2+\alpha)}{\psi(2+\alpha)-2}}\, \ud \mu\Big)^{\frac{\psi (2+\alpha)-2}{2(\alpha+2)}}
    \le  C(l,\alpha,\psi)r r_l\Big(\vint_{B_l} |\nabla_H \omega_l|^{\frac{2}{\psi}} \,\ud \mu\Big)^{\frac{\psi}{2}} \\
    &\phantom{AAA}\le C(l,\alpha,\psi)r r_l\Big(\vint_{B_l}\frac{|\nabla_H \omega_l|^2}{|u^1|^{2\beta}} \,\ud \mu\Big)^{\frac{1}{2}} \Big(\vint_{B_l} |u^1|^{\frac{2\beta}{\psi - 1}} \ud \mu\Big)^{\frac{\psi - 1}{2}}, 
\end{aligned}
\end{equation}
where in the last step we also employ the H\"older inequality. The same inequality applied again for $\gamma \in (0,\frac{2+\alpha}{\psi(2+\alpha)-2})$ leads to the following estimate:
\begin{equation}\label{weak harnack for u^2 3}
\begin{aligned}
    \vint_{B_l}(u^2_+)^{q}\frac{(u^2_+)^2}{|u^1|^{2\beta}} \ud \mu &\le \Big(\vint_{B_l} (u^2_+)^{q\Big(\frac{(2+\alpha)}{\psi(2+\alpha)-2}\Big)\frac{1}{\gamma}} \ud \mu\Big)^{\frac{\gamma(\psi (2+\alpha)-2)}{\alpha+2}}\Big(\vint_{B_l}\Big(\frac{(u^2_+)^2}{|u^1|^{2\beta}} \Big)^{\frac{2+\alpha}{2+\alpha - \gamma(\psi(2+\alpha)-2)}}\ud \mu\Big)^{\frac{2+\alpha - \gamma(\psi(2+\alpha)-2)}{2+\alpha}}.
\end{aligned}
\end{equation}
Observe that, since the constant $\psi \in \left(1,\frac{4+\alpha}{2+\alpha}\right)$, it holds that $\frac{2+\alpha}{\psi(2+\alpha)-2}>1$. Moreover, recall that we need to assume that $\gamma>1$ for the Moser iteration technique to work. For the sake of simplicity of the presentation let us introduce the following notation:
\begin{equation}
    C_l:=\Big(\vint_{B_l} |u^1|^{\frac{2\beta}{\psi - 1}} \ud \mu\Big)^{\psi - 1}, \qquad D_l:=\Big(\vint_{B_l}\left(\frac{(u^2_+)^2}{|u^1|^{2\beta}} \Big)^{\frac{2+\alpha}{2+\alpha - \gamma(\psi(2+\alpha)-2)}}\ud \mu\right)^{\frac{2+\alpha - \gamma(\psi(2+\alpha)-2)}{2+\alpha}}.\label{defn-ClDl}
\end{equation}
Furthermore, observe that by the Ahlfors $2$-regularity of measure $\mu$ and by the inclusion $B_{l} \subset B_{l-1}$ one can directly verify the following inequalities:
\begin{equation}\label{weak harnack for u^2 relation between B,C}
    C_l \le 4^{\psi-1}C_{l-1}, \qquad D_l \le 4^{{\frac{2+\alpha - \gamma(\psi(2+\alpha)-2)}{2+\alpha}}}D_{l-1}.
\end{equation}
Moreover, for notation purposes, we set:
\begin{equation*}
    A_{l}:= C(l,\alpha,\psi)(3+q)(1-\lambda)^{-1}2^l, \qquad A:= C(l,\alpha,\psi)(1-\lambda)^{-1}. 
\end{equation*}
In consequence, inequalities \eqref{weak harnack for u^2 1} - \eqref{weak harnack for u^2 3}, as well as properties of the test functions $\eta_l$ yield the following inequality which is a cornerstone for running the Moser iteration scheme:
\begin{equation}\label{weak harnack for u^2 4}
    \Big(\vint_{B_{l+1}}(u^2_+)^{\frac{q(2+\alpha)}{\psi(2+\alpha)-2}}\ud \mu\Big)^{\frac{\psi (2+\alpha)-2}{q(\alpha+2)}}
    \le A_l^{\frac{2}{q+2}}C_l^{\frac{1}{q+2}}D_l^{\frac{1}{q+2}} \Big(\vint_{B_l} (u^2_+)^{\frac{1}{\gamma}\left(\frac{q(2+\alpha)}{\psi(2+\alpha)-2}\right)} \ud \mu\Big)^{\gamma\left(\frac{\psi (2+\alpha)-2}{q(\alpha+2)}\right)\frac{q}{q+2}}.
\end{equation}
Here, $\gamma \in (1,\frac{2+\alpha}{\psi(2+\alpha)-2})$. Based on~\eqref{weak harnack for u^2 4} we define the  exponents $q_l$ such that $q_l$ satisfy the following condition: 
\[
 q_l \frac{2+\alpha}{\psi(2+\alpha)-2}=2\gamma^{l+1},\quad l=0,1,\ldots.
\]
Thus, we may inductively define exponents $k_0 = 2$ and $k_{l+1} := \gamma k_l$, obtaining that $k_l=2\gamma^ l$ for  $l=0,1,\ldots$. Under this notation the inequality~\eqref{weak harnack for u^2 4} reads:
\begin{equation}\label{weak harnack for u^2 5}
    \Big(\vint_{B_{l+1}}(u^2_+)^{k_{l+1}} \ud \mu \Big)^{\frac{1}{k_{l+1}}} \le A_l^{\frac{2}{q_l+2}}C_l^{\frac{1}{q_l+2}}D_l^{\frac{1}{q_l+2}}\Big(\vint_{B_{l}}(u^2_+)^{k_{l}} \ud \mu\Big)^{\frac{1}{k_{l}} \frac{q_l}{q_l+2}},\quad l=0,1,2,\ldots,
\end{equation}
where here we also use that $r_l \le 1$. Having estimate~\eqref{weak harnack for u^2 5}, we are in a position to iterate this inequality upon introducing the following auxiliary expression which simplifies the presentation:
\[
 P_{i,l}:=\frac{1}{q_i}\prod_{n=i}^l\frac{q_n}{q_n+2}=\frac{1}{q_i}\frac{\gamma^{i+1}}{\gamma^{i+1}+\frac{2+\alpha}{\psi(2+\alpha)-2}}\cdot \frac{\gamma^{i+2}}{\gamma^{i+2}+\frac{2+\alpha}{\psi(2+\alpha)-2}}\cdot \cdots \cdot \frac{\gamma^{l+1}}{\gamma^{l+1}+\frac{2+\alpha}{\psi(2+\alpha)-2}}.
\]
In this notation, the iterations of inequality~\eqref{weak harnack for u^2 5} lead to the following estimate: 
\begin{align}\label{weak harnack for u^2 6 robocza}
 \Big(\vint_{B_{l+1}}(u^2_+)^{k_{l+1}} \ud \mu \Big)^{\frac{1}{k_{l+1}}} &\!\! \le \! C(\alpha,\psi)^{\sum_{i=0}^l P_{i,l}}\Big(\prod_{i=0}^l(3+q_i)^{2P_{i,l}}\Big)4^{\psi\sum_{i=1}^l iP_{i,l}}
    (4^{{\frac{2+\alpha - \gamma(\psi(2+\alpha)-2)}{2+\alpha}}})^{\sum_{i=1}^l iP_{i,l}} \times \nonumber \\
  & \times (C_0 D_0)^{\sum_{i=0}^l P_{i,l}}  (1-\lambda)^{-2\sum_{i=0}^l  P_{i,l}}\left(\vint_{B_{r}}(u^2_+)^{2} \ud \mu\right)^{\frac{1}{2} \prod_{n=0}^l \frac{q_n}{q_n+2}}.
\end{align}
Observe that the term $4^{\psi\sum_{i=1}^l i P_{i,l} }$ arises as product of the appropriate terms from~\eqref{weak harnack for u^2 relation between B,C}, the factor $2^l$ in the definition of $A_l$ and as a consequence of the iterating procedure. Moreover, we may further simplify the estimate~\eqref{weak harnack for u^2 6 robocza}, by noting that 
\begin{equation*}
    \frac{1}{2}\Big(\prod_{n=i+1}^l \frac{q_n}{q_n+2} - \prod_{n=i}^l \frac{q_n}{q_n+2}\Big) = \frac{1}{q_i}\Big(\prod_{n=i}^l \frac{q_n}{q_n+2}\Big)\,\,\hbox{ and }\,\, \frac{1}{2}\left(1-\frac{q_l}{q_l+2}\right) = \frac{1}{q_l+2},
\end{equation*}
for $i=0,1,2,\ldots,l-1$. Hence
\begin{equation*}
    0<\sum_{i=0}^l \frac{1}{q_i}\Big(\prod_{n=i}^l\frac{q_n}{q_n+2}\Big) = \frac{1}{2}\Big(1 - \prod_{n=0}^l \frac{q_n}{q_n +2}\Big) < \frac{1}{2}.
\end{equation*}
Furthermore, we directly obtain the following inequalities:
\[
\prod_{i=0}^l(3+q_i)^{\frac{2}{q_i}} \le C(\alpha,\psi,\gamma)\,\, \hbox{ and }\,\,\sum_{i=1}^{\infty}\frac{i}{q_i} \le C(\alpha,\psi,\gamma).
\]
All together the above estimates allow us to simplify~\eqref{weak harnack for u^2 6 robocza} as follows:
\begin{equation}\label{weak harnack for u^2 6}
\begin{aligned}
    \Big(\vint_{B_{l+1}}(u^2_+)^{k_{l+1}} \ud \mu \Big)^{\frac{1}{k_{l+1}}} &\le C(\alpha,\psi,\gamma)  (1-\lambda)^{-1} \left(C_0 D_0\right)^{\sum_{i=0}^l P_{i,l}} \Big(\vint_{B_{r}}(u^2_+)^{2} \ud \mu\Big)^{\frac{1}{2} \prod_{n=0}^l \frac{q_n}{q_n+2}}\\
    &=C(\alpha,\psi,\gamma) (1-\lambda)^{-1} \left(C_0 D_0\right)^{\frac{1}{2}(1-\prod_{n=0}^l\frac{q_n}{q_n+2})} \left(\vint_{B_{r}}(u^2_+)^{2} \ud \mu\right)^{\frac{1}{2} \prod_{n=0}^l \frac{q_n}{q_n+2}}.
\end{aligned}
\end{equation}

In order to complete the proof of Theorem~\ref{weak-harnack-u2} we need to further estimate the right-hand side of~\eqref{weak harnack for u^2 6}.  By the H\"older inequality we get the following estimate:
\begin{equation}\label{weak harnack for u^2 proof corrections 1}
\begin{aligned}
    \vint_{B_{r}}(u^2_+)^{2} \ud \mu = \vint_{B_{r}}\frac{(u^2_+)^{2}}{|u^1|^{2\beta}} |u^1|^{2\beta} \ud \mu \le \Big(\vint_{B_{r}}\left(\frac{(u^2_+)^{2}}{|u^1|^{2\beta}}\Big)^{\frac{1}{2-\psi}} \ud \mu\right)^{2-\psi} \Big(\vint_{B_{r}}|u^1|^{\frac{2\beta}{\psi-1}} \ud \mu\Big)^{\psi-1}.
\end{aligned}
\end{equation}
Since $\frac{(\psi-1)(2+\alpha)}{\psi(2+\alpha)-2}<1<\gamma$, it holds that $\frac{1}{2-\psi}<\frac{2+\alpha}{2+\alpha-\gamma(\psi(2+\alpha)-2)}$ and, thus, by the Jensen inequality applied to the first integral on the right-hand side of~\eqref{weak harnack for u^2 proof corrections 1} and~\eqref{defn-ClDl} applied with $l=0$, we obtain that:
\begin{equation*}
    \vint_{B_{r}}(u^2_+)^{2} \ud \mu \le C_0D_0.
\end{equation*}
This discussion allows us to express the estimate~\eqref{weak harnack for u^2 6} in the following form:
\begin{align*}
    \Big(\vint_{B_{l+1}}&(u^2_+)^{k_{l+1}} \ud \mu \Big)^{\frac{1}{k_{l+1}}} \\
    &\le C(\alpha,\psi,\gamma) (1-\lambda)^{-1} \Big(\vint_{B_r}\left(\frac{(u^2_+)^2}{|u^1|^{2\beta}} \Big)^{\frac{2+\alpha}{2+\alpha - \gamma(\psi(2+\alpha)-2)}}\ud \mu\right)^{\frac{2+\alpha - \gamma(\psi(2+\alpha)-2)}{2(2+\alpha)}}\Big(\vint_{B_{r}}|u^1|^{\frac{2\beta}{\psi-1}} \ud \mu\Big)^{\frac{\psi-1}{2}},
\end{align*}
and by letting $l \to \infty$ we arrive at the following inequality:
\begin{equation}\label{weak harnack for u^2 final ineq 2}
    \sup_{B(x_0,\lambda r)} u^2_+ \le \frac{C(\alpha,\psi,\gamma)}{1-\lambda} \Big(\vint_{B_r}\left(\frac{(u^2_+)^2}{|u^1|^{2\beta}} \Big)^{\!\frac{2+\alpha}{2+\alpha - \gamma(\psi(2+\alpha)-2)}}\ud \mu\right)^{\!\frac{2+\alpha - \gamma(\psi(2+\alpha)-2)}{2(2+\alpha)}}\left(\vint_{B_{r}}|u^1|^{\frac{2\beta}{\psi-1}} \ud \mu\right)^{\!\frac{\psi-1}{2}}.
\end{equation}
Finally, we are in a position to finish the proof of the theorem and show weak Harnack estimates~\eqref{weak harnack for u^2, p ineq} and~\eqref{weak harnack for u^2, nice ineq}.
%
In order to show~\eqref{weak harnack for u^2, p ineq} we need to find $\psi, \gamma>1$ such that the following equation holds:
\begin{equation*}
    \frac{2+\alpha}{2+\alpha-\gamma(\psi(2+\alpha)-2)} = \frac{p(2+\alpha)(1+\beta)}{2}=\frac12 p{\mathcal P}_{\alpha, \beta}.
\end{equation*}
Since $p>1>\frac{1-\beta}{1+\beta}$, it holds that the following $\gamma>1$:
\[
  \gamma := \frac{(2+\alpha)(1+\beta)-\frac{2}{p}}{\frac{2\beta}{p}+\alpha(1+\beta)}.
\]
Furthermore, since $\beta \geq 0$, we trivially have that $p>1>\frac{\beta}{\beta+1}$ and it follows that $p>\frac{p}{p(\beta+1)-\beta}$. This observation allows us to define $\psi>1$ such that the following inequality is satisfied:
\begin{equation*}
  \gamma= \frac{(2+\alpha)(1+\beta)-\frac{2}{p}}{\frac{2\beta}{p}+\alpha(1+\beta)}<\frac{2+\alpha}{\psi(2+\alpha)-2}=\frac{p(2+\alpha)(1+\beta)}{2\beta +p\alpha(1+\beta)},
\end{equation*}
and so the assertion~\eqref{weak harnack for u^2, p ineq} holds.  
Since it trivially holds that
\begin{equation*}
    \frac{|u^2_+|}{|u^1|^\beta} \le \|u\|_{G_2^{\beta}} \,\,\hbox{ and }\,\, |u^1| \le \|u\|_{G_2^{\beta}}
\end{equation*}
the second assertion~\eqref{weak harnack for u^2, nice ineq} follows as well for $\beta>0$ completing the proof of Theorem~\ref{weak-harnack-u2} in this case. The discussion for the case $\beta=0$ goes along the same lines and therefore, we omit its proof. 
\end{proof}

    
\begin{cor} 
    Let $\alpha \in (0,1)$ and $\beta \ge 0$. Suppose that $u=(u^1,u^2) \in HW^{1,2}(\Om,\mu)$ be a strongly harmonic mapping defined on a domain $\Om\subset G_2^\alpha$ and satisfying~\eqref{orth-cond} with a function $K \in L^2_{\textrm{loc}}(\Om,\mu)$. Moreover, assume that for some $p>1+\beta$ we have $\| u \|_{G_2^{\beta}} \in L^{p(2+\alpha)}_{\textrm{loc}}(\Om,\mu)$. Then $u \in L^{\infty}_{\textrm{loc}}(\Om)$.
\end{cor} 
\begin{proof} It is enough to show the assertion on a ball, since then the covering argument gives the assertion on compact subsets of domain $\Om$.

The case of $\beta =0$ follows directly from the boundedness of $u^1$ asserted in Definition~\ref{str-solution} and inequality~\eqref{weak harnack for u^2, nice ineq}.

Let $\beta>0$ and consider $s>1>\frac{\beta}{\beta+1}$. It follows that $s>\frac{s}{s(\beta+1)-\beta}$. Moreover, since $p>1+\beta$, the exponent $s$ can be chosen so that
\[ 
1<\frac{s}{s(\beta+1)-\beta}<s<\frac{p}{\beta+1}.
\]
%
Hence, we may apply the estimate~\eqref{weak harnack for u^2, nice ineq} with $s>1$ and the Jensen inequality with the exponent $\frac{p}{s(\beta+1)}>1$ and obtain the assertion of the corollary:
\[
    (\sup_{\frac12 B_r}|u^2_{\pm}|)^{\frac{p}{s(1+\beta)}} \lesssim \left(\vint_{B_r} \| u\|_{G_2^{\beta}}^{(1+\beta)s(2+\alpha)} \ud \mu \right)^{\frac{1}{s(2+\alpha)}\frac{p}{s(1+\beta)}}\lesssim \left(\vint_{B_r} \| u\|_{G_2^{\beta}}^{p(2+\alpha)} \ud \mu \right)^{\frac{1}{s(2+\alpha)}}<\infty.
\]
%
\end{proof}
%
%

Following very similar reasoning as above we may also prove the weak Harnack estimates for $u^1_{\pm}$ the positive, respectively, negative part of the component function $u^1$ of a Grushin strongly harmonic mapping $u=(u^1, u^2)$.

\begin{theorem}[Weak Harnack inequality for $u^1$]\label{weak harnack for u^1 thm}
 Let $\alpha \in (0,1)$, $\beta \ge 0$, $\Om \subset G_2^{\alpha}$ be an open connected set and let $u:\Om \to G_2^{\beta}$ be a strongly harmonic map satisfying~\eqref{orth-cond} with a function $K \in L^{p(2+\alpha)}_{\textrm{loc}}(\Om,\mu)$.   

Then for each $p>1$ and any Grushin ball $B_r = B(x_0, r) \subset \Om$ there exists a constant $C = C(\alpha,\beta,p)>0$ such that for any $\lambda \in (0,1)$ the following weak Harnack estimate holds:
\begin{equation}\label{weak harnack for u^1 main inequality}
\sup_{\lambda B_r} |u^1_{\pm}| \le \frac{C\phi(K,r)}{(1-\lambda)^{\frac{2}{1+\beta}}}
\left(\vint_{B_{r}} \| u \|_{G_2^{\beta}}^{p(2+\alpha)(1+\beta)}\right)^{\frac{1}{p(2+\alpha)(1+\beta)}},
\end{equation}
where
\[
\phi(K,r):=\Big(1+ r^2\Big(\vint_{B_{r}} K^{p(2+\alpha)} \ud \mu \Big)^{\frac{1}{p(2+\alpha)}}\Big)^{\frac{1}{2(1+\beta)}}.
\]
\end{theorem}

We omit the proof and only comment that the key difference between the above statement and the estimate~\eqref{weak harnack for u^2, nice ineq} is appearance of the expression $\phi(K,r)$ in~\eqref{weak harnack for u^1 main inequality}. It arises as consequence of defining the corresponding integral expressions $C_l$ and $D_l$ in~\eqref{defn-ClDl} as follows:
\begin{equation*}
   C_l :=1+ r^2\left(\vint_{B_{l}} K^{p(2+\alpha)} \ud \mu \right)^{\frac{1}{p(2+\alpha)}}, \quad D_l:= \left(\vint_{B_{l}} \| u \|_{G_2^{\beta}}^{p(2+\alpha)(1+\beta)}\ud \mu \right)^{\frac{2}{p(2+\alpha)}},\qquad \textrm{for }l=0,1,2,\ldots.
\end{equation*}

\end{document}